\documentclass[smallextended]{svjour3}
\usepackage{amsmath}
\usepackage{amssymb}
\usepackage{bm}
\usepackage{cite}
\usepackage[colorlinks,urlcolor=blue,citecolor=blue,linkcolor=blue]{hyperref}
\usepackage{mathrsfs}
\usepackage{tcolorbox}
\usepackage{soul}
\usepackage{dsfont}

\usepackage[colorinlistoftodos]{todonotes}
\usepackage{enumitem}

\newcommand{\innp}[1]{\left\langle #1 \right\rangle}

\newcommand{\mA}{\mathbf{A}}

\newcommand{\mB}{\mathbf{B}}
\newcommand{\mC}{\mathbf{C}}
\newcommand{\mU}{\mathbf{U}}
\newcommand{\bE}{\mathbf{E}}
\newcommand{\mX}{\mathbf{X}}

\newcommand{\zeros}{\textbf{0}}
\newcommand{\vx}{\mathbf{x}}
\newcommand{\vxb}{\Bar{\mathbf{x}}}
\newcommand{\vh}{\mathbf{h}}

\newcommand{\cx}{\mathcal{X}}

\newcommand{\cL}{\mathcal{L}}
\newcommand{\cS}{\mathcal{S}}

\newcommand{\cls}{\mathcal{L}_{\mathcal{S}}}

\newcommand{\vxh}{\mathbf{\hat{x}}}

\newcommand{\vy}{\mathbf{y}}
\newcommand{\vz}{\mathbf{z}}
\newcommand{\vv}{\mathbf{v}}
\newcommand{\ve}{\mathbf{e}}

\newcommand{\vb}{\mathbf{b}}
\newcommand{\valpha}{\bm{\alpha}}

\newcommand{\vDelta}{\bm{\Delta}}
\newcommand{\vg}{\mathbf{g}}
\newcommand{\vu}{\mathbf{u}}

\newcommand{\vmu}{\bm{\mu}}
\newcommand{\vnu}{\bm{\nu}}

\newcommand{\rr}{\mathbb{R}}
\newcommand{\ee}{\mathbb{E}}
\newcommand{\Compl}{\mbox{Compl}}
\newcommand{\Schatten}{\mathscr{S}}

\newtheorem{fact}[theorem]{Fact}
\newtheorem{assumptions}[theorem]{Assumptions}
\newtheorem{observation}[theorem]{Observation}

\DeclareMathOperator*{\argmin}{argmin}

\DeclareMathOperator*{\argsup}{argsup}

\makeatletter
\newcommand{\subalign}[1]{%
  \vcenter{%
    \Let@ \restore@math@cr \default@tag
    \baselineskip\fontdimen10 \scriptfont\tw@
    \advance\baselineskip\fontdimen12 \scriptfont\tw@
    \lineskip\thr@@\fontdimen8 \scriptfont\thr@@
    \lineskiplimit\lineskip
    \ialign{\hfil$\m@th\scriptstyle##$&$\m@th\scriptstyle{}##$\hfil\crcr
      #1\crcr
    }%
  }%
}
\makeatother

\newif\ifmarkup
\markupfalse

\newcommand{\markupadd}[1]{%
\ifmarkup
\textcolor{blue}{#1}%
\else
#1%
\fi
}

\newcommand{\markupdelete}[1]{%
\ifmarkup
\textcolor{red}{#1}%
\else
\fi
}

\title{Complementary Composite Minimization, Small Gradients in General Norms, and Applications \markupdelete{to Regression Problems}\thanks{JD was supported by the NSF grant CCF-2007757 and by the Office of the Vice
Chancellor for Research and Graduate Education at the University of Wisconsin–Madison with funding from the
Wisconsin Alumni Research Foundation. CG was partially supported by INRIA through the INRIA Associate Teams project, CORFO through the Clover 2030 Engineering Strategy - 14ENI-26862, and FONDECYT Project 1210362.}}

\titlerunning{Complementary Composite Minimization}        % if too long for running head

\author{Jelena Diakonikolas         \and
        Crist\'{o}bal Guzm\'{a}n %etc.
}

\institute{J. Diakonikolas \at
              University of Wisconsin-Madison \\
              \email{jelena@cs.wisc.edu}           %  \\
           \and
           C. Guzm\'{a}n \at
              Institute for Mathematical and Computational Engineering\\
              Faculty of Mathematics and School of Engineering\\
              Pontificia Universidad Cat\'olica de Chile\\
              \email{crguzmanp@mat.uc.cl}
}

\date{Received: date / Accepted: date}
\begin{document}

\maketitle

\begin{abstract}
Composite minimization is a powerful framework in large-scale convex optimization, based on decoupling of the objective function into terms with structurally different properties and allowing for more flexible algorithmic design. We introduce a new algorithmic framework for {\em complementary composite minimization}, where the objective function decouples into a (weakly) smooth and a uniformly convex term. This particular form of decoupling is pervasive in statistics and machine learning, due to its link to regularization. 

The main contributions of our work are summarized as follows. First, we introduce the problem of complementary composite minimization in general normed spaces; second, we provide a unified accelerated algorithmic framework to address broad classes of complementary composite minimization problems; 
and third, we prove that the algorithms resulting from our framework are near-optimal in most of the standard optimization settings. 
Additionally, we show that our algorithmic framework can be used to address the problem of making the gradients small in general normed spaces. As a concrete example, we obtain a nearly-optimal method for the standard $\ell_1$ setup (small gradients in the $\ell_\infty$ norm), essentially matching the bound of~\cite{nesterov2012make} that was previously known only for the Euclidean setup. 
Finally, we show that our composite methods are broadly applicable to a number of regression \markupadd{and other classes of optimization problems, where regularization plays a key role. Our methods
lead to complexity bounds that are either new or match the best existing ones.}
\end{abstract}

\section{Introduction}
\begin{center}\emph{No function can be both smooth and strongly convex with respect to an $\ell_p$ norm and have a dimension-independent condition number, unless $p=2$.%~\cite{Ball:1994,Borwein:2009}.
}
\end{center}
This is a basic fact from convex analysis\footnote{More generally, it is known that the existence of a continuous uniformly convex function with growth bounded by the squared norm implies that the space has an equivalent $2$-uniformly convex norm \cite{Borwein:2009}; furthermore, using duality \cite{Zalinescu:1983}, we conclude that the existence of a smooth and strongly convex function implies that the space has equivalent $2$-uniformly convex and $2$-uniformly smooth norms, a rare property for a normed space (the most notable examples of spaces that are simultaneously $2$-uniformly convex and $2$-uniformly smooth are Hilbert spaces; see e.g.,~\cite{Ball:1994} for related definitions and more details).} 
and the primary reason why in the existing literature smooth and strongly convex optimization is typically considered only for Euclidean (or, slightly more generally, Hilbert) spaces. In fact, it is not only that moving away from $p = 2$ the condition number becomes dimension-dependent, but that the dependence on the dimension is polynomial for all examples of functions we know of, unless $p$ is trivially close to two. Thus, it is tempting to assert that dimension-independent linear convergence (i.e., with logarithmic dependence on the inverse accuracy $1/\epsilon$) is reserved for Euclidean (or nearly-Euclidean) spaces, which has long been common wisdom within the optimization community. 

%Contrary to this wisdom, we show that 
We show that this separation between Euclidean and non-Euclidean setups is not so clear-cut, and 
it is in fact possible to attain linear convergence even in $\ell_p$ (or, more generally, in normed vector) spaces, as long as the objective function can be decomposed into two functions with complementary properties. In particular, we show that if the objective function can be written in the following \emph{complementary composite} form 
\begin{equation}\label{eq:comp-composite-obj}
    \bar{f}(\vx) = f(\vx) + \psi(\vx),
\end{equation}
where $f$ is convex and $L$-smooth w.r.t.~a (not necessarily Euclidean) norm $\|\cdot\|$ and $\psi$ is $m$-strongly convex w.r.t.~the same norm and ``simple,'' meaning that the optimization problems of the form\footnote{This oracle should be contrasted with the proximal oracle, which would require solving problems of the form $\min_{\vx} \{\frac{1}{2}\|\vx - \vz\|_2^2 + \psi(\vx)\}$ and is primarily used with Euclidean norms. Eq.~\eqref{eq:psi-simple} reduces to a linear optimization/Frank-Wolfe oracle when $\psi$ is the indicator of a convex polytope. For our analysis, it also suffices to have an oracle of the form $\min_{\vx} \{\frac{1}{2}\|\vx - \vz\|^2 + \psi(\vx) + \phi(\vx)\},$ where $\phi$ is strongly convex w.r.t.~$\|\cdot\|$.}  
\begin{equation}\label{eq:psi-simple}
    \min_{\vx} \innp{\vz, \vx} + \psi(\vx)
\end{equation}
can be solved efficiently for any linear functional $\vz,$ then $\bar{f}(\vx)$ can be minimized to accuracy $\epsilon > 0$ in $O\Big(\sqrt{\frac{L}{m}}\log(\frac{L\phi(\vxb^{\star})}{\epsilon})\Big)$ iterations, where $\vxb^{\star} = \argmin_{\vx}\bar{f}(\vx)$. As in other standard first-order iterative methods, each iteration requires one call to the gradient oracle of $f$ and one call to a solver for the problem from Eq.~\eqref{eq:psi-simple}. To the best of our knowledge, such a result was previously known only for Euclidean spaces~\cite{nesterov2013gradient}.

This is the basic variant of our result. We also consider more general setups in which $f$ is only weakly smooth (with H\"{o}lder-continuous gradients) and $\psi$ is uniformly convex (see Section~\ref{sec:prelims} for specific definitions and useful properties). We refer to the resulting objective functions $\bar{f}$ as \emph{complementary composite objective functions} (as functions $f$ and $\psi$ that constitute $\bar{f}$ have complementary properties) and to the resulting optimization problems as \emph{complementary composite optimization problems.} The algorithmic framework\markupadd{, based on the Approximate Duality Gap Technique (ADGT)~\cite{diakonikolas2019approximate},} that we consider for complementary composite optimization in Section~\ref{sec:comp-min} is near-optimal (optimal up to logarithmic or poly-logarithmic factors) in terms of iteration complexity in most of the standard optimization settings, which we certify by providing near-matching oracle complexity lower bounds in Section~\ref{sec:LowerBounds}. \markupadd{On a conceptual level, the extension of ADGT to complementary composite settings that handles both uniform convexity and weak smoothness without much additional technical work is another contribution of our work.\footnote{\markupadd{We note that ADGT  could already handle basic composite cases (without uniform convexity of $\psi$)~\cite{diakonikolas2019approximate} and weakly smooth cases~\cite{diakonikolas2018accelerated}; however, previous analyses did not allow exploiting uniform convexity of $\psi$.}}}
We now summarize some further implications of our results.

\paragraph{Small gradients in $\ell_p$ and $\Schatten_p$ norms.} %\todo{Add a note about the $p = 1$ case.}
The original motivation for complementary composite optimization in our work comes from %the problem of 
making the gradients of smooth functions small in non-Euclidean norms. This is a fundamental optimization question, whose study was initiated in~\cite{nesterov2012make} and that is %. Despite being fundamental, this problem is 
still far from being well-understood. Prior to this work, (near)-optimal algorithms were known only for the Euclidean ($\ell_2$) and $\ell_{\infty}$ setups.\footnote{In the $\ell_{\infty}$ setup, a non-Euclidean variant of gradient descent is optimal in terms of iteration complexity.} 

For the Euclidean setup, there are two main results: due to~\cite{kim2020optimizing} and due to~\cite{nesterov2012make}. The algorithm of~\cite{kim2020optimizing} is iteration-complexity-optimal; however, the methodology by which this algorithm was obtained is crucially Euclidean, as it relies on numerical solutions to semidefinite programs, whose formulation is made possible by assuming that the norm of the space is inner-product-induced. An alternative approach, due to~\cite{nesterov2012make}, is to apply the fast gradient method to  a regularized function $\Bar{f}(\vx) = f(\vx) + \frac{\lambda}{2}\|\vx - \vx_0\|_2^2$ for a sufficiently small $\lambda >0,$ where $f$ is the smooth function whose gradient we hope to minimize. Under the appropriate choice of $\lambda >0,$ the resulting algorithm is near-optimal (optimal up to a logarithmic factor).

As discussed earlier, applying the fast gradient method directly to a regularized function as in~\cite{nesterov2012make} is out of question for $p \neq 2,$ as the resulting regularized objective function cannot simultaneously be smooth and strongly convex w.r.t.~$\|\cdot\|_p$ without its condition number growing with the problem dimension. This is where the framework of complementary composite optimization proposed in our work comes into play. Our result also generalizes to normed matrix spaces endowed with $\Schatten_p$ (Schatten-$p$) norms.\footnote{$\Schatten_p$ norm of a matrix $\mA$ is defined as the $\ell_p$ norm of $\mA$'s singular values.} As a concrete example, our approach leads to near-optimal complexity results in the $\ell_1$ and $\Schatten_1$ (a.k.a.~nuclear norm) setups, where the gradient is minimized in the $\ell_\infty$, respectively, $\Schatten_{\infty}$ (a.k.a.~spectral), norm.

It is important to note here why strongly convex regularizers are not sufficient in general and what motivated us to consider  the more general uniformly convex functions $\psi$. While for $p \in (1, 2]$ choosing $\psi(\vx) = \frac{1}{2}\|\cdot\|_p^2$ (which is $(p-1)$-strongly convex w.r.t.~$\|\cdot\|_p;$ see~\cite{nemirovski:1983,juditsky2008large}) 
%,d2018optimal}) 
is sufficient, when $p > 2$  the strong convexity parameter of  $\frac{1}{2}\|\cdot\|_p^2$ w.r.t.~$\|\cdot\|_p$ is  bounded above  by   $1/d^{1 -\frac{2}{p}}$. This is not only true for $\frac{1}{2}\|\cdot\|_p^2$, but for any convex function bounded above by a constant on a unit $\ell_p$-ball; see e.g.,~\cite[Example 5.1]{d2018optimal}. Thus, in this case, we work with $\psi(\vx) = \frac{1}{p}\|\vx\|_p^p$, which is only uniformly convex. 

\paragraph{Lower complexity bounds.} We complement the development of algorithms for complementary composite minimization and minimizing the norm of the gradient with lower bounds for the oracle complexity of these problems. Our lower bounds %borrow substantially from -- this doesn't sound good
leverage recent lower bounds for weakly smooth convex optimization from \cite{guzman:2015,diakonikolas:2019}. These existing results suffice for proving lower bounds for minimizing the norm of the gradient, and certify the near-optimality of our approach for the smooth (i.e., with Lipschitz continuous gradient) setting, when $1\leq p\leq 2$. On the other hand, proving lower bounds for complementary convex optimization requires the design of an appropriate oracle model; namely, one that takes into account that our algorithm accesses the gradient oracle of $f$ and solves subroutines of type \eqref{eq:psi-simple} w.r.t.~$\psi$. With this model in place, we combine constructions from uniformly convex nonsmooth lower bounds \cite{Sridharan,Juditsky:2014} with local smoothing \cite{guzman:2015,diakonikolas:2019} to provide novel lower bounds for complementary composite minimization. The resulting bounds show that our algorithmic framework is nearly optimal (i.e., optimal up to poly-logarithmic factors w.r.t.~dimension, target accuracy, regularity constants of the objective, and initial distance to optimum) for all interesting regimes of parameters.

\paragraph{Applications\markupdelete{ to regression problems}.} 

%\textcolor{red}{
%Currently, the applications are:
%\begin{enumerate}
    %\item[(1)] Elastic Net
    %\item[(2)] Bridge Regression with Squared $\ell_p$ norm regularization
    %\item[(3)] Symmetric PSD linear systems with maximum constraint violation guarantee
    %\item[(4)] $\ell_p$-regression
    %\item[(5)] Spectral Variants
    %\item[(6)] Entropy-Regularized Optimal Transport
%\end{enumerate}
%This might be too much, especially since reviewers were not happy with applications from the first version. Perhaps we can remove 5 (and possibly another) to make it more approachable. I also think that pushing some of the proofs of (2) (e.g., Proposition 3 and Lemma 6) to the Appendix might help with readability.
%}

The importance of complementary composite optimization and making the gradients small in $\ell_p$ and $\Schatten_p$ norms is perhaps best exhibited by considering some of the classical \markupdelete{regression} problems that are frequently used in statistics and machine learning. It turns out that considering these \markupdelete{regression} problems \markupdelete{in the appropriate} \markupadd{within the} complementary composite \markupdelete{form} \markupadd{framework} not only leads to faster algorithms in general, but also reveals some interesting properties of the solutions. For example, an application of our framework to \markupadd{risk minimization problems leads to statistically optimal rates or the order $\frac{1}{\sqrt{n}}$ in $\ell_p$ spaces for $p \in (1, 2]$, while simultaneously guaranteeing that the $\ell_p$ norm of the output predictor is within a constant factor of the minimum $\ell_p$ norm over all minimizers of the (unregularized) risk function. When $p$ is close to 1, the latter property can be interpreted as enforcing sparsity, similar to LASSO. }\markupdelete{the complementary composite form of bridge regression (a generalization of lasso and ridge regression; see Section~\ref{sec:apps}) leads to an interesting and well-characterized trade-off between the ``goodness of fit'' of the model and the $\ell_p$ norm of the regressor.}

Section~\ref{sec:apps} provides several illustrative examples of \markupdelete{regression} problems that can be addressed using our framework, including lasso, elastic net, \markupdelete{(b)ridge regression, Dantzig selector,} \markupadd{empirical risk minimization, solving positive semidefinite linear systems with maximum constraint violation guarantee,} $\ell_p$ regression (with standard and correlated errors), and related spectral variants. It is important to note that a single algorithmic framework suffices for addressing all of these problems. Most of the results we obtain in this way are either conjectured or known to be unimprovable.  \markupadd{Furthermore, to illustrate the broad applicability of our methods, we also explore the consequences of minimizing the norm of the gradient for the discrete optimal transport problem. In this case, our space of interest is $\ell_{\infty}$, and we make use of simple quadratic regularization. 
The resulting arithmetic complexity of our method matches many of the recently developed methods for this problem.}

%%%%%%%%%%%%%%%%%%%%%%%%%%%%%%%%%%%%%%%%%%%%%%%%%

%%%%%%%%%%%%%%%%%%% RELATED
\subsection{Further Related Work}\label{sec:related-work}

Nonsmooth convex optimization problems with the composite structure of the objective function $\bar{f}(\vx) = f(\vx) + \psi(\vx)$, where $f$ is smooth and convex, but $\psi$ is nonsmooth, convex, and ``simple,'' are well-studied in the optimization literature~\cite[and references therein]{beck2009fast,nesterov2013gradient,Scheinberg:2014,He:2015,gasnikov2018universal}. The main benefit of exploiting the composite structure lies in the ability to recover accelerated rates for nonsmooth problems.
One of the most celebrated  results in this domain are the FISTA algorithm from \cite{beck2009fast}, and a method based on composite gradient mapping proposed in~\cite{nesterov2013gradient}, which demonstrated that accelerated convergence (with rate $1/k^2$) is possible for this class of problems.
%\todo[inline]{CG: There might be a problem with the last claim here. I don't know whether  Nesterov 2013 or Beck and Teboulle 2009 came first (Nesterov had a preprint from 2007 \url{http://www.optimization-online.org/DB_FILE/2007/09/1784.pdf}). Perhaps rephrase...}

By comparison, the literature on complementary composite minimization is scarce. For example, in \cite{nesterov2013gradient} it was proved that in a Euclidean space complementary composite optimization \markupadd{with smooth $f$ and strongly convex $\psi$} \markupdelete{attains} \markupadd{admits} a linear convergence rate. The algorithm proposed there is different from ours, as it relies on the use of composite gradient mapping, for which the proximal operator of $\psi$ (solution to problems of the form $\min_{\vx}\{\psi(\vx) + \frac{1}{2}\|\vx -\vz\|_2^2\}$ for all $\vz;$ compare to Eq.~\eqref{eq:psi-simple}) is assumed to be efficiently computable. \markupadd{A more general setting of complementary composite optimization with non-Euclidean norms can be handled by~\cite{allen2017katyusha}; however it still requires computing a non-Euclidean version of the proximal operator of $\psi$, of the form $\min_{\vx}\{\innp{\vz, \vx} + \psi(\vx) + \frac{1}{2}\|\vx -\vx_0\|^2\}$ for any $\vz, \vx_0.$} In addition to being primarily applicable to Euclidean spaces, \markupdelete{this} \markupadd{such} assumption\markupadd{s} \markupadd{about the existence of a generalized proximal operator} further restrict\markupdelete{s} the class of functions that can be efficiently optimized compared to our approach (see Section~\ref{sec:comp-considerations} for a further discussion). 
Another composite algorithm where linear convergence has been proved is the celebrated method from \cite{Chambolle:2011}, where proximal steps are taken w.r.t.~both terms in the composite model ($f$ and $\psi$). In the case where both $f$ and $\psi$ are strongly convex, a linear convergence rate can be established. Notice that this assumption is quite different from our setting, and that this method was only investigated for the Euclidean setup.

Beyond the realm of Euclidean norms, linear convergence results have been established for functions that are \emph{relatively smooth and relatively strongly convex}~\cite{bauschke2017descent,bauschke2019linear,lu2018relatively}. The class of complementary composite functions does not fall into this category. Further, while we show 
accelerated rates (with square-root dependence on the appropriate notion of the condition number)  %is possible 
for complementary composite optimization, such 
results are not attainable 
%results cannot be established 
for relatively smooth relatively strongly convex optimization \markupadd{(see Appendix \ref{app:LB_rel_smooth} for a proof).} \markupadd{In a work independent and parallel to ours, \cite[Appendix E]{cohen2021relative} obtained a linear convergence result that handles complementary composite setting in which $f$ is smooth and $\psi$ is strongly convex, with respect to an arbitrary norm, under a similar assumption about ``simplicity'' of $\psi$ to ours, as stated in \eqref{eq:psi-simple}.}
%~\cite{dragomir2019optimal}.\footnote{Lower bounds from \cite{dragomir2019optimal} show the impossibility of acceleration for the relatively smooth setting. \markupadd{A simple reduction further shows that in the relatively smooth and relatively strongly convex setting, no acceleration is possible, thus the linear convergence rate in \cite{lu2018relatively} is nearly-optimal.}} %This is strong evidence of the impossibility of acceleration in the relatively smooth and relatively strongly convex setting.}

The problem of minimizing the norm of the gradient has become a central question in optimization and its applications in machine learning, mainly motivated by nonconvex settings, where the norm of the gradient is useful as a stopping criterion. However, the norm of the gradient is also useful in linearly constrained convex optimization problems, where the norm of the gradient of a Fenchel dual is useful in controlling the feasibility violation in the primal \cite{nesterov2012make}. Our approach for minimizing the norm of the gradient is inspired by the regularization approach proposed in \cite{nesterov2012make}. As discussed earlier, this regularization approach is not directly applicable to non-Euclidean settings, and is where our complementary composite framework becomes crucial. %, note however this approach is not directly applicable to non-Euclidean settings, where we show that the complementary composite approach leads to nearly optimal rates for $2$-uniformly smooth spaces.

Finally, our work is inspired by and uses fundamental results about the geometry of high-dimensional normed spaces; in particular, the fact that for $\ell_p$ and $\Schatten_p$ spaces the optimal constants of uniform convexity are known~\cite{Ball:1994}. These results imply that powers of the respective norm are uniformly convex, which suffices for our regularization. Moreover, those functions have explicitly computable convex conjugates (problems as in Eq.~\eqref{eq:psi-simple} can be solved in closed form), which is crucial for our algorithms to work.

%%%%%%%%%%%%%%%%%% PRELIMS
\subsection{Notation and Preliminaries}\label{sec:prelims}

Throughout the paper, we use boldface letters to denote vectors and italic letters to denote scalars. 

We consider real finite-dimensional normed vector spaces $\mathbf{E},$ endowed with a norm $\|\cdot\|,$ and denoted by $(\mathbf{E}, \|\cdot\|).$ The space dual to $(\mathbf{E}, \|\cdot\|)$ is denoted by $(\mathbf{E}^*, \|\cdot\|_*),$ where $\|\cdot\|_*$ is the norm dual to $\|\cdot\|,$ defined in the usual way by $\|\vz\|_* = \sup_{\vx \in \mathbf{E}: \|\vx\|\leq 1}\innp{\vz, \vx},$ where $\innp{\vz, \vx}$ denotes the evaluation of a linear functional $\vz$ on a point $\vx \in \mathbf{E}.$ As a concrete example, we may consider the $\ell_p$ space $(\rr^d, \|\cdot\|_p),$ where $\|\vx\|_p = \big(\sum_{i=1}^d |x_i|^p\big)^{1/p},$ $1\leq p \leq \infty.$ The space dual to $(\rr^d, \|\cdot\|_p)$ is isometrically isomorphic to the space $(\rr^d, \|\cdot\|_{p_{\ast}}),$ where $\frac{1}{p}+ \frac{1}{p_{\ast}} = 1.$ Throughout, given $1\leq p\leq \infty$, we refer to $p_{\ast}=\frac{p}{p-1}$ as the conjugate exponent to $p$ (notice that $1\leq p_{\ast}\leq \infty$, and $\frac1p+\frac{1}{p_{\ast}}=1$). %\todo[]{I added this piece of definition to not repeat over and over.}
The (closed) $\|\cdot\|$-norm ball centered at $\vx$ with radius $R>0$ is denoted by ${\cal B}_{\|\cdot\|}(\vx,R)$. 
We start by recalling some standard definitions from convex analysis. 

%We consider real finite-dimensional normed vector spaces $(\rr^d, \|\cdot\|),$ where $\|\cdot\|$ is chosen as one of the standard $\ell_p$ norms $\|\vx\|_p = \big(\sum_{i=1}^d |x_i|^p\big)^{1/p},$ where $p \geq 1.$ This space is isometrically isomorphic to the space $(\rr^d, \|\cdot\|_*),$ where $\|\cdot\|_*$ is the norm dual to $\|\cdot\|,$ defined in the usual way as: $\|\vz\|_* = \sup_{\vx \in \rr^d: \|\vx\|\leq 1}\innp{\vx, \vz},$ where $\innp{\cdot, \cdot}$ denotes the standard inner product. In the case of $\ell_p$ norms, the norm dual to $\|\cdot\|_p$ is the norm $\|\cdot\|_{p_{\ast}},$ where $\frac{1}{p}+ \frac{1}{p_{\ast}} = 1.$

%\todo[inline]{Define standard shit like convex functions, (weakly) smooth functions, complexity of a class, convex conjugates, etc.}

\begin{definition}\label{def:weak-smoothness}
A function $f:\bE \to \rr$ is said to be $(L, \kappa)$-weakly smooth w.r.t.~a norm $\|\cdot\|$, where $L > 0$ and $\kappa \in (1, 2],$ if its gradients are $(L, \kappa - 1)$ H\"{o}lder continuous, i.e., if
$$
    (\forall \vx, \vy \in \bE):\quad \|\nabla f(\vx) - \nabla f(\vy)\|_* \leq L \|\vx - \vy\|^{\kappa - 1}.
$$
We denote the class of $(L,\kappa)$-weakly smooth functions w.r.t.~$\|\cdot\|$ by ${\cal F}_{\|\cdot\|}(L,\kappa)$.
\end{definition}
Note that when $\kappa = 1,$ the function may not be differentiable. Since we will only be working with functions that are proper, convex, and lower semicontinuous, we will still have that $f$ is subdifferentiable on the interior of its domain \cite[Theorem 23.4]{Rockafellar:1970}. The definition of $(L, \kappa)$-weakly smooth functions then boils down to the bounded variation of the subgradients. %, which can be shown to be equivalent to Lipschitz continuity of $f.$\todo[inline]{CG: $(L,1)$-weak smoothness does not imply Lipschitz continuity, it's the other way around. Perhaps delte after: ''which can be shown...''.} 
%\todo[inline]{CG: I'm being picky here, but can we replace proper by real-valued above? I don't think proper, convex, and lower semicontinuous suffice for subdifferentiability. Example: indicator function does not have subgradients outside the domain.}

\begin{definition}\label{def:unif-cvx}
A function $\psi: \bE \to \rr$ is said to be $q$-uniformly convex w.r.t.~a norm $\|\cdot\|$ and with constant $\lambda$ (and we refer to such functions as $(\lambda, q)$-uniformly convex), where $\lambda \geq 0$ and $q \geq 2$, if $\forall \alpha \in (0, 1):$
$$
    (\forall \vx, \vy \in \bE):\quad \psi((1-\alpha)\vx + \alpha\vy) \leq (1-\alpha)\psi(\vx) + \alpha \psi(\vy) - \frac{\lambda}{q} \alpha(1-\alpha)\|\vy - \vx\|^q.
$$
We denote the class of $(\lambda,q)$-uniformly convex functions w.r.t.~$\|\cdot\|$ by ${\cal U}_{\|\cdot\|}(\lambda,q)$.
\end{definition}
%
%With the abuse of notation, we will often use $\nabla \psi(\vx)$ to denote an arbitrary but fixed element of $\partial \psi(\vx).$ 
When $\psi$ is only subdifferentiable (but not differentiable), we make a mild assumption that the subgradient oracle of $\psi$ is consistent, i.e., that it returns the same element of $\partial \psi(\vx)$ whenever queried at the same point $\vx.$ 

Observe that when $\lambda = 0,$ uniform convexity reduces to  standard convexity, while for $\lambda > 0$ and $q = 2$ we recover the definition of strong convexity. We only consider functions that are lower semicontinuous, convex, and proper. These properties suffice for a function to be subdifferentiable on the interior of its domain. It is then not hard to show that if $\psi$ is $(\lambda, q)$-uniformly convex w.r.t.~a norm $\|\cdot\|$ and $\vg_{\vx} \in \partial \psi(\vx)$ is its subgradient at a point $\vx,$ we have
\begin{equation}
    (\forall \vy \in \bE):\quad \psi(\vy) \geq \psi(\vx) + \innp{\vg_{\vx}, \vy - \vx} + \frac{\lambda}{q}\|\vy - \vx\|^q.
\end{equation}

\begin{definition}\label{def:cvx-conj}
Let $\psi:\bE \to \rr\cup \{+\infty\}.$ The convex conjugate of $\psi,$ denoted by $\psi^*$, is defined by
$$
    (\forall \vz \in \bE^*): \quad \psi^*(\vz) = \sup_{\vx \in \bE}\{\innp{\vz, \vx} - \psi(\vx)\}.
$$
\end{definition}
Recall that the convex conjugate of any function is convex. Some simple examples of conjugate pairs of functions that will be useful for our analysis are: (i) univariate functions $\frac{1}{p}|\cdot|^p$ and $\frac{1}{p_{\ast}}|\cdot|^{p_{\ast}},$ where 
$1<p<\infty$ %$\frac{1}{p} + \frac{1}{p_{\ast}} = 1$ 
(see, e.g.,~\cite[Exercise 4.4.2]{borwein2004techniques}) and (ii) functions $\frac{1}{2}\|\cdot\|^2$ and $\frac{1}{2}\|\cdot\|_*^2,$ where norms $\|\cdot\|$ and $\|\cdot\|_*$ are dual to each other (see, e.g.,~\cite[Example 3.27]{boyd2004convex}). The latter example can be easily adapted to prove that the functions $\frac1p\|\cdot\|^p$ and $\frac{1}{p_{\ast}}\|\cdot\|_{\ast}^{p_{\ast}}$ are conjugates of each other, for $1< p< \infty$.
%\todo[inline]{Can we include the example of conjugate pairs $\frac1p\|\cdot\|^p$ and $\frac{1}{p_{\ast}}\|\cdot\|_{\ast}^{p_{\ast}}$?}

The following auxiliary facts will be useful for our analysis.

\begin{fact}\label{fact:danskin}
Let $\psi: \bE\to \rr \cup \{+\infty\}$ be proper, convex, and lower semicontinuous, and let $\psi^*$ be its convex conjugate.  Then $\psi^*$ is proper, convex, and lower semicontinuous (and thus subdifferentiable on the interior of its domain) and $\forall \vz \in \mathrm{int\,dom}(\psi^*)$:  $\vg \in \partial\psi^*(\vz)$ if and only if $\vg \in \argsup_{\vx \in \rr^d}\{\innp{\vz, \vx} - \psi(\vx)\}.$
%\todo[inline]{Here again, I have some doubts on the subdifferentiability.\\
%JD: I am convinced I have seen it somewhere that if $f$ is proper, convex, and lsc then it is subdifferentiable on its domain. There is a variant of this statement in Rockafellar's book (Theorem 23.4), but without the lsc assumption. It states that the subdifferential set at $\vx$ is nonempty and bounded if and only if $\vx \in \mathrm{int}\mathrm{dom}(f(\vx))$. I cannot find a statement that includes lsc and would apply to the entire domain. But I guess changing the statement to the interior of the domain should be fine?\\
%CG: Yes! In the interior of the domain is for sure subdifferentiable. }
\end{fact}

%The following proposition will be useful when working with $\ell_p$ spaces. 
%\todo[inline]{CG: I suggest generalizing this as follows. Below it suffices $2\leq q<\infty$. I would keep it that way to be consistent with the setting of the paper.\\
%JD: Is it still true for $q > 2$ and a norm that is not necessarily $\|\cdot\|_q$? Do you have a reference for that? I only know how to prove it for $\frac{1}{q}\|\cdot\|_p^q$ when $q \in \{2, p\}.$\\
%CG: Write $\Phi(x)=\frac1r\big( \|\vx\| \big)^r$ as the composition of $\frac1r|\cdot|^r$ and $\|\cdot\|$ and use the rule of the subdifferential of a composition. The differentiability is obtained if $1<r<\infty$ given that the derivative of $\frac1r|\cdot|^r$ vanishes at 0. BTW, $q/q^{\ast}=q-1$, which I believe is simpler to grasp.}

The following proposition will be repeatedly used in our analysis, and we prove it here for completeness.
\begin{proposition}\label{prop:duality-map-of-p}
Let $(\bE,\|\cdot\|)$ be a normed space with $\|\cdot\|^2:\bE\rightarrow\rr$ differentiable,
%such that the unit ball w.r.t.~$\|\cdot\|$ is strictly convex,  %with $\|\cdot\|^q:\bE\rightarrow\rr$ differentiable, where 
and let 
$1<q<\infty$. Then
$$
    \Big\|\nabla \Big(\frac{1}{q}\|\vx\|^q\Big)\Big\|_{*} =  \|\vx\|^{q-1} = \|\vx\|^{q/q_{\ast}},
$$
where $q_{\ast} = \frac{q}{q-1}$ is the exponent conjugate to $q$.
\end{proposition}
\begin{proof}
We notice that $\|\cdot\|^2$ is differentiable if and only if  $\|\cdot\|^q$ is differentiable 
\cite[Thm.~3.7.2]{Zalinescu:2002}. 
%~\textcolor{red}{[ADD CITATION].} 
Since the statement clearly holds for $\vx = \zeros,$ in the following we assume that $\vx \neq \zeros.$ 
Next, write $\frac{1}{q}\|\cdot\|^q$ as a composition of functions $\frac{1}{q}|\cdot|^{q/2}$ and $\|\cdot\|^2.$ Applying the chain rule of differentiation, we now have:
\begin{align*}
    \nabla \Big(\frac{1}{q}\|\vx\|^q\Big) = \frac{1}{2}\Big(\|\vx\|^2\Big)^{\frac{q}{2}-1} \nabla \big(\|\vx\|^2\big)
    = \|\vx\|^{q -2} \nabla \Big(\frac{1}{2}\|\vx\|^2\Big).
\end{align*}
It remains to argue that $\Big\| \nabla \Big(\frac{1}{2}\|\vx\|^2\Big) \Big\|_* = \|\vx\|.$ This immediately follows by Fact~\ref{fact:danskin}, as $\frac{1}{2}\|\cdot\|^2$ and $\frac{1}{2}\|\cdot\|_*^2$ are convex conjugates of each other. 
%\todo[inline]{There is an issue with the composition that you chose, since $\frac{1}{q}|\cdot|^{q/2}$ is not convex if $q<2$. I propose an alternative proof below.\\
%JD: Why is convexity needed?\\
%CG: In principle it's not a problem, but you will also need to assume in your proof that $\|\cdot\|^2$ is differentiable.}
\qed\end{proof}
\iffalse %%%%%%%%%%%%%%%%%%%%%%%%%%%%%%
\begin{proof}
Notice that $\frac1q\|\cdot\|^q$ is the composition of convex function $\|\cdot\|$ with the convex nondecreasing function $\frac1q[\,\cdot\,]_+^q$. This way, we can use the chain rule for the subdifferential (see e.g. \cite[Thm.~3.47]{Beck:2017})
$$ \partial\Big( \frac1q\|\cdot\|^q \Big)(\vx) = \partial\Big( \frac1q[\,\cdot\,]_+^q \Big)(\|\vx\|)\cdot \partial \|\cdot\|(\vx)
=\|\vx\|^{q-1}\partial \|\cdot\|(\vx).$$
Note that the lhs is a singleton containing the gradient, thus
$\nabla \frac1q\|\vx\|^q  \in \|\vx\|^{q-1}\cdot \partial \|\cdot\|(\vx).$
If $\vx=0$ we can conclude the result by the last expression; if $\vx\neq 0$, since $\vx$ is strictly convex, then $\partial \|\cdot\|(\vx)$ is a singleton given by a vector $\vz$ with $\|\vz\|_{\ast}=1$, which proves the result in this case as well.
%We conclude using that the norm is 1-Lipschitz w.r.t.~itself, thus $\partial \|\cdot\|(\vx)\subseteq \partial{\cal B}_{\|\cdot\|_{\ast}}(0,1)$.
\todo{How does this imply the stated equality? CG: Sorry, I missed a detail here. I tried an alternative proof. If we replace the differentiability assump by strict convexity of the unit ball (which is morally equivalent), I can get the result. I am fine with either version.} 
\qed\end{proof}
\fi %%%%%%%%%%%%%%%%%%%%%%%%%%%%%%%%%
%
\iffalse
\begin{proposition}\label{prop:duality-map-of-p}
For any $p \in (1, \infty)$ and any $\vx \in \rr^d$ it holds %\textcolor{red}{Here $1<p<\infty$, right? The border cases are not differentiable.}
$$
    \Big\|\nabla \Big(\frac{1}{2}\|\vx\|_p^2\Big)\Big\|_{p_{\ast}} = \|\vx\|_p \quad \text{and} \quad \Big\|\nabla \Big(\frac{1}{p}\|\vx\|_p^p\Big)\Big\|_{p_{\ast}} = \|\vx\|_{p}^{p/p_{\ast}},
$$
where $p_{\ast} = \frac{p}{p-1}$ is dual to p.
\end{proposition}
%\todo[inline]{CG: Consider adding this observation as an example right after Fenchel conjugate definition. Leads naturally to the proposition as well.\\
%JD: Makes sense.}

Proposition~\ref{prop:duality-map-of-p} can be proved by  verifying that the stated equality holds, though a more direct proof is just a corollary of Fact~\ref{fact:danskin}, using that function pairs $\frac{1}{2}\|\cdot\|^2$, $\frac{1}{2}\|\cdot\|_*^2$ and $\frac{1}{p}\|\cdot\|_p^p,$ $\frac{1}{p_{\ast}}\|\cdot\|_{p_{\ast}}^{p_{\ast}}$ are conjugates of each other. %$\nabla \big(\frac{1}{2}\|\vx\|_p^2\big)$ being a duality map of $\vx$ in $(\rr^d, \|\cdot\|_p)$, i.e., $\vu = \nabla \big(\frac{1}{2}\|\vx\|_p^2\big)$ satisfies $\innp{\vu, \vx} = \|\vx\|_p^2 = \|\vu\|_{p_{\ast}}^2.$
\fi

We also state here a lemma that allows approximating weakly smooth functions by weakly smooth functions of a different order. A variant of this lemma (for $p = 2$) first appeared in \cite{devolder2014first}, while the more general version stated here is from~\cite{d2018optimal}.
\begin{lemma}\label{lemma:p-weakly-smooth-approx}
Let $f: \bE \to \rr$ be a function that is $(L, \kappa)$-weakly smooth w.r.t.~some norm $\|\cdot\|$. Then for any $\delta > 0$ and
\begin{equation}\label{eq:M-w-smooth}
    M \geq \Big[\frac{2(p-\kappa)}{p\kappa \delta}\Big]^{\frac{p-\kappa}{\kappa}}L^{\frac{p}{\kappa}}
\end{equation}
we have
$$
    (\forall \vx, \vy \in \bE):\quad f(\vy) \leq f(\vx) + \innp{\nabla f(\vx), \vy - \vx} + \frac{M}{p}\|\vy - \vx\|^p + \frac{\delta}{2}.
$$
\end{lemma}

Finally, the following lemma will be useful when bounding the gradient norm in Section~\ref{sec:small-grad-norm} (see also \cite[Section 3.5]{Zalinescu:2002}). 
%\todo{What is a good reference for this? I would go for \cite[Section 3.5]{Zalinescu:2002}. Or \cite{Zalinescu:1983}. They are a bit cryptic, but they have what we need.}

\begin{lemma}\label{lemma:weak-smooth-cocoercivity}
Let $f: \bE \to \rr$ be a function that is convex and $(L, \kappa)$-weakly smooth w.r.t.~some norm $\|\cdot\|$. Then:
\begin{equation*}
    (\forall \vx, \vy \in \bE): \frac{\kappa - 1}{L^{\frac{1}{\kappa -1}} \kappa}\|\nabla f(\vy) - \nabla f(\vx)\|_{\ast}^{\frac{\kappa}{\kappa - 1}} \leq f(\vy) - f(\vx) - \innp{\nabla f(\vx), \vy - \vx}.
\end{equation*}
\end{lemma}
\begin{proof}
Let $h(\vx)$ be any $(L, \kappa)$-weakly smooth function and let $\vx^{\star} \in \argmin_{\vx \in \rr^d}h(\vx).$ As $h$ is $(L, \kappa)$-weakly smooth, we have for all $\vx, \vy \in \rr^d:$
$$
    h(\vy) \leq h(\vx) + \innp{\nabla h(\vx), \vy - \vx} + \frac{L}{\kappa}\|\vy - \vx\|^{\kappa}.
$$
Fixing $\vx \in \rr^d$ and minimizing both sides of the last inequality w.r.t.~$\vy \in \rr^d$, it follows that
\begin{equation}\label{eq:h-smooth-min}
    h(\vx^{\star}) \leq h(\vx) - \frac{L^{1-\kappa_{\ast}}}{\kappa_{\ast} }\|\nabla h(\vx)\|_*^{\kappa_{\ast}},
\end{equation}
%where $\kappa_{\ast} = \frac{\kappa}{\kappa -1},$ and 
where we have used 
%the definition of a convex conjugate and the fact 
that the functions $\frac{1}{\kappa}\|\cdot\|^{\kappa}$ and $\frac{1}{\kappa_{\ast}}\|\cdot\|_*^{\kappa_{\ast}}$ are convex conjugates of each other. 

To complete the proof, it remains to apply Eq.~\eqref{eq:h-smooth-min} to function $h_{\vx}(\vy) = f(\vy) - \innp{\nabla f(\vx), \vy - \vx}$ for any fixed $\vx \in \rr^d,$ and observe that $h_{\vx}(\vy)$ is convex, $(L, \kappa)$-weakly smooth, and minimized at $\vy = \vx.$ 
\qed\end{proof}
%%%%%%%%%%%%%%%%%%%%%%
%%%%%%%%%%%%%%%%%%%%%%
\section{Complementary Composite Minimization}\label{sec:comp-min}
In this section, we consider minimizing complementary composite functions, which are of the form 
\begin{equation}\label{eq:comp-fun}
    \Bar{f}(\vx) = f(\vx) + \psi(\vx),
\end{equation}
where $f$ is $(L, \kappa)$-weakly smooth w.r.t.~some norm $\|\cdot\|$, $\kappa \in (1,2],$ and $\psi$ is $(\lambda, q)$-uniformly convex w.r.t.~the same norm, for some $q\geq 2$, $\lambda \geq 0$. We assume that the feasible set $\cx \subseteq \bE$ is closed, convex, and nonempty. 

%\todo[inline]{CG: Are we doing constrained optimization? Because of ${\cal X}$\\
%JD: In this case, it's easy to do constrained optimization (not any harder than unconstrained), so I thought why not? For the gradient norm it's a bit harder to justify, but for this one the proof works for any strongly convex or unif convex $\phi, \psi$, so it's conceivable to have some interesting constrained setups.}

%%%%%%%%%%%%%%%%%%%%%%%%%
\subsection{Algorithmic Framework and Convergence Analysis}
%We show that a generalized variant of AGD+ from \cite{cohen2018acceleration} converges at a rate that will lead to the desired iteration complexities for all $\kappa \in (1,2], p \geq 1.$ The variant of the algorithm can be stated as:
%\todo[inline]{CG: I think it would look better aligning the equations below. Or even using an Algorithm envorinment?}
The algorithmic framework we consider is a generalization of AGD+ from \cite{cohen2018acceleration},\footnote{The same method is also known as the method of similar triangles, introduced independently in~\cite{gasnikov2018universal}. This method is also related to Tseng's accelerated proximal gradient method~\cite{tseng2008accelerated}, and can be seen as its ``lazy'' or dual averaging counterpart.} stated as follows:
\begin{tcolorbox}
{\bf Generalized AGD+}
\begin{equation}\label{eq:mod-agd+}
    \begin{aligned}
        \vx_k &= \frac{A_{k-1}}{A_k}\vy_{k-1} + \frac{a_k}{A_k}\vv_{k-1}\\
        \vv_k &= \argmin_{\vu \in \cx}\Big\{\sum_{i=0}^k a_i\innp{\nabla f(\vx_i), \vu - \vx_i} + A_k \psi(\vu) + m_0 \phi(\vu)\Big\}\\
        \vy_k &= \frac{A_{k-1}}{A_k}\vy_{k-1} + \frac{a_k}{A_k}\vv_k,\\
        \vy_0 &= \vv_0, \; \vx_0 \in \cx,
    \end{aligned}
\end{equation}
\end{tcolorbox}
\noindent 
where $m_0$ and the sequence of positive numbers $\{a_k\}_{k\geq 0}$ are parameters of the algorithm specified in the convergence analysis below, $A_k = \sum_{i=0}^k a_i,$ and we take $\phi(\vu)$ to be a function that satisfies $\phi(\vu) \geq \frac{1}{q}\|\vu - \vx_0\|^q.$ For example, if $\lambda > 0,$ we can take $\phi(\vu) = \frac{1}{\lambda} D_{\psi}(\vu, \vx_0).$ Observe also that for this choice of $\phi,$ the minimization problem defining $\vv_k$ is of the form of Eq.~\eqref{eq:psi-simple}. When $\lambda = 0$, we take $\phi$ to be $(1, q)$-uniformly convex. 

%We start with the following lemma which applies to an arbitrary composite function described above (Eq.~\eqref{eq:comp-fun}).

%\todo[inline]{JD: The lemma below should probably be turned into a theorem and the proof split into supporting lemmas. Can be done later.}

The convergence analysis relies on the approximate duality gap technique (ADGT) of~\cite{diakonikolas2019approximate}. The main idea is to construct an upper estimate $G_k \geq \Bar{f}(\vy_k) - \Bar{f}(\vxb^{\star})$ of the true optimality gap, where $\vxb^{\star} = \argmin_{\vx \in \cx}\Bar{f}(\vu),$ and then argue that $A_k G_k \leq A_{k-1}G_{k-1} + E_k$, which in turn implies:
$$
    \Bar{f}(\vy_k) - \Bar{f}(\vxb^{\star}) \leq \frac{A_0G_0}{A_k} + \frac{\sum_{i=1}^k E_i}{A_k}.
$$
Thus, as long as $A_0 G_0$ is bounded and the cumulative error $\sum_{i=1}^k E_i$ is either bounded or increasing slowly compared to $A_k$, the optimality gap of the sequence $\vy_k$ converges to the optimum at rate $(1 + \sum_{i=1}^k E_i)/A_k.$ The goal is, of course, to make $A_k$ as fast-growing as possible, but that turns out to be limited by the requirement that $A_k G_k$ be non-increasing or slowly increasing compared to $A_k$.

The gap $G_k$ is constructed as the difference $U_k - L_k,$ where $U_k \geq \Bar{f}(\vy_k)$ is an upper bound on $\Bar{f}(\vy_k)$ and $L_k \leq \Bar{f}(\vxb^{\star})$ is a lower bound on $\Bar{f}(\vxb^{\star}).$ In this particular case, we make the following choices: 
\begin{equation*}
    U_k = f(\vy_k) + \frac{1}{A_k}\sum_{i=0}^k a_i \psi(\vv_i).
\end{equation*}
As $\vy_k = \frac{1}{A_k}\sum_{i=0}^k a_i \vv_i,$ we have, by Jensen's inequality:
$
    U_k \geq f(\vy_k) + \psi(\vy_k) = \Bar{f}(\vy_k), 
$ 
i.e., $U_k$ is a valid upper bound on $\Bar{f}(\vy_k).$ 

For the lower bound, we use the following inequalities:
\begin{align*}
    \Bar{f}(\vxb^{\star}) \geq \; &\frac{1}{A_k}\sum_{i=0}^k a_i f(\vx_i) + \frac{1}{A_k}\sum_{i=0}^k a_i \innp{\nabla f(\vx_i), \vxb^{\star} - \vx_i} + \psi(\vxb^{\star})\\
    &+ \frac{m_0}{A_k}\phi(\vxb^{\star})- \frac{m_0}{A_k}\phi(\vxb^{\star})\\
    \geq\; & \frac{1}{A_k}\sum_{i=0}^k a_i f(\vx_i) - \frac{m_0}{A_k}\phi(\vxb^{\star})\\
    & + \frac{1}{A_k}\min_{\vu \in \cx}\Big\{\sum_{i=0}^k a_i \innp{\nabla f(\vx_i), \vu - \vx_i} + A_k \psi(\vu) + m_0 \phi(\vu)\Big\}\\
    =:\; & L_k,
\end{align*}
where the first inequality uses $$f(\vxb^{\star}) \geq \frac{1}{A_k}\sum_{i=0}^k a_i f(\vx_i) + \frac{1}{A_k}\sum_{i=0}^k a_i \innp{\nabla f(\vx_i), \vxb^{\star} - \vx_i},$$ by convexity of $f.$

We start by bounding the initial (scaled) gap $A_0 G_0.$ 
\begin{lemma}[Initial Gap]\label{lemma:init-gap}
For any $\delta_0 > 0$ and $M_0 = \Big[\frac{2(q-\kappa)}{q\kappa \delta_0}\Big]^{\frac{q-\kappa}{\kappa}}L^{\frac{q}{\kappa}},$ if $A_0 M_0 = m_0,$ then 
$$A_0 G_0 \leq m_0 \phi(\vxb^{\star})  + \frac{A_0 \delta_0}{2}.$$
\end{lemma}
\begin{proof}
By definition, and using that $a_0 = A_0$,
\begin{align*}
    A_0G_0 =& \; A_0 \Big(f(\vy_0) + \psi(\vv_0) - f(\vx_0) - \innp{\nabla f(\vx_0), \vv_0 - \vx_0} - \psi(\vv_0) - \frac{m_0}{A_0} \phi(\vv_0)\Big)\\
    &+ m_0 \phi(\vxb^{\star})\\
    =&\; A_0(f(\vy_0) - f(\vx_0) - \innp{\nabla f(\vx_0), \vy_0 - \vx_0}) - m_0 \phi(\vy_0) + m_0 \phi(\vxb^{\star})
\end{align*}
where the second line is by $\vy_0 = \vv_0$.   

By assumption, $\phi(\vu) \geq \frac{1}{q}\|\vu - \vx_0\|^q,$ for all $\vu,$ and, in particular, $\phi(\vy_0) \geq \frac{1}{q}\|\vy_0 - \vx_0\|_q^q.$ On the other hand, by $(L, \kappa)$-weak smoothness of $f$ and using Lemma~\ref{lemma:p-weakly-smooth-approx}, we have that (below $M_0 = \big[\frac{2(q-\kappa)}{q\kappa \delta_0}\big]^{\frac{q-\kappa}{\kappa}}L^{\frac{q}{\kappa}}$):
$$
    f(\vy_0) - f(\vx_0) - \innp{\nabla f(\vx_0), \vy_0 - \vx_0} \leq \frac{M_0}{q}\|\vy_0 - \vx_0\|_q^q + \frac{\delta_0}{2}.
$$
%where $M_0 = \Big[\frac{2(q-\kappa)}{q\kappa \delta_0}\Big]^{\frac{q-\kappa}{\kappa}}L^{\frac{q}{\kappa}}$.

Therefore:
\begin{equation}\label{eq:init-gap}
    A_0 G_0 \leq \big(A_0 M_0 - m_0\big)\frac{\|\vy_0 - \vx_0\|^q}{q} + m_0 \phi(\vxb^{\star}) + \frac{A_0 \delta_0}{2} = m_0 \phi(\vxb^{\star})  + \frac{A_0 \delta_0}{2},
\end{equation}
as $m_0 = A_0M_0$. 
\qed\end{proof}

The next step is to bound $A_k G_k - A_{k-1}G_{k-1},$ as in the following lemma.

\begin{lemma}[Gap Evolution]\label{lemma:change-in-gap}
Given arbitrary $\delta_k > 0$ and  $M_k = \Big[\frac{2(q-\kappa)}{q\kappa \delta_k}\Big]^{\frac{q-\kappa}{\kappa}}L^{\frac{q}{\kappa}}$, if $\frac{{a_k}^q}{{A_k}^{q-1}}\leq \frac{\max\{\lambda A_{k-1}, m_0\}}{ M_k}$ then 
$$
    A_k G_k - A_{k-1} G_{k-1} \leq \frac{A_k \delta_k}{2}.
$$
\end{lemma}
\begin{proof}
To bound $A_k G_k - A_{k-1}G_{k-1}$, we first bound $A_k U_k - A_{k-1}U_{k-1}$ and $A_kL_k - A_{k-1}L_{k-1}.$ By definition of $U_k,$
\begin{equation}\label{eq:upper-bnd-change}
    \begin{aligned}
        A_k U_k - A_{k-1}U_{k-1} =\; & A_k f(\vy_k) - A_{k-1}f(\vy_{k-1}) + a_k \psi(\vv_k)\\
        =\; & A_k(f(\vy_k) - f(\vx_k)) + A_{k-1}(f(\vx_{k}) - f(\vy_{k-1}))\\
        &+ a_k f(\vx_k) + a_k \psi(\vv_k).
    \end{aligned}
\end{equation}
For the lower bound, define the function under the minimum in the definition of the lower bound as $h_k(\vu) := \sum_{i=0}^k a_i \innp{\nabla f(\vx_i), \vu - \vx_i} + A_k \psi(\vu) + m_0 \phi(\vu),$ so that we have:
\begin{equation}\label{eq:lb-change-1}
    A_k L_k - A_{k-1}L_{k-1} = a_k f(\vx_k) + h_k(\vv_k) - h_{k-1}(\vv_{k-1}). 
\end{equation}
Observe first that 
\begin{equation}\label{eq:h-k-1}
    h_k(\vv_k) - h_{k-1}(\vv_k) = a_k \innp{\nabla f(\vx_k), \vv_k - \vx_k} + a_k \psi(\vv_k).    
\end{equation}
On the other hand, using the definition of Bregman divergence and the fact that Bregman divergence is blind to constant and linear terms, we can bound $h_{k-1}(\vv_k) - h_{k-1}(\vv_{k-1})$ as
\begin{align*}
    h_{k-1}(\vv_k) - h_{k-1}(\vv_{k-1}) &= \innp{\nabla h_{k-1}(\vv_{k-1}), \vv_k - \vv_{k-1}} + D_{h_{k-1}}(\vv_k, \vv_{k-1})\\
    &\geq A_{k-1}D_{\psi}(\vv_k, \vv_{k-1}) + m_0 D_{\phi}(\vv_k, \vv_{k-1}),%\\
   % &\geq A_{k-1}D_{\psi}(\vv_k, \vv_{k-1}),
\end{align*}
where the second line is by $\vv_{k-1}$ being the minimizer of $h_{k-1}$.  
Combining with Eqs.~\eqref{eq:lb-change-1} and \eqref{eq:h-k-1}, we have:
\begin{equation}\label{eq:lb-change-final}
    \begin{aligned}
        A_k L_k - A_{k-1}L_{k-1} \geq\; & a_k f(\vx_k) + a_k \psi(\vv_k) + a_k \innp{\nabla f(\vx_k), \vv_k - \vx_k}\\
        &+ A_{k-1}D_{\psi}(\vv_k, \vv_{k-1}) \markupdelete{-}\markupadd{+} m_0 D_{\phi}(\vv_k, \vv_{k-1}).
    \end{aligned}
\end{equation}

Combining Eqs.~\eqref{eq:upper-bnd-change} and \eqref{eq:lb-change-final}, we can now bound $A_k G_k - A_{k-1}G_{k-1}$ as
\begin{align*}
    A_k G_k - A_{k-1}G_{k-1} \leq &\; A_k(f(\vy_k) - f(\vx_k)) + A_{k-1}(f(\vx_{k}) - f(\vy_{k-1})) \\
    &- a_k \innp{\nabla f(\vx_k), \vv_k - \vx_k}\\
    &- A_{k-1}D_{\psi}(\vv_k, \vv_{k-1}) - m_0 D_{\phi}(\vv_k, \vv_{k-1})\\
    \leq&\; A_k(f(\vy_k) - f(\vx_k) - \innp{\nabla f(\vx_k), \vy_k - \vx_k})\\
    &- A_{k-1}D_{\psi}(\vv_k, \vv_{k-1}) - m_0 D_{\phi}(\vv_k, \vv_{k-1}),
\end{align*}
where we have used $f(\vx_{k}) - f(\vy_{k-1}) \leq \innp{\nabla f(\vx_k), \vx_k - \vy_{k-1}}$ (by convexity of $f$) and the definition of $\vy_k$ from Eq.~\eqref{eq:mod-agd+}. Similarly as for the initial gap, we now use the weak smoothness of $f$ and Lemma~\ref{lemma:p-weakly-smooth-approx} to write:
\begin{align*}
    f(\vy_k) - f(\vx_k) - \innp{\nabla f(\vx_k), \vy_k - \vx_k} &\leq \frac{M_k}{q}\|\vy_k - \vx_k\|^q + \frac{\delta_k}{2}\\
    &= \frac{M_k}{q}\frac{{a_k}^q}{{A_k}^q}\|\vv_k - \vv_{k-1}\|^q + \frac{\delta_k}{2},
\end{align*}
where $M_k = \Big[\frac{2(q-\kappa)}{q\kappa \delta_k}\Big]^{\frac{q-\kappa}{\kappa}}L^{\frac{q}{\kappa}}$ and the equality is by $\vy_k - \vx_k = \frac{a_k}{A_k}(\vv_k - \vv_{k-1})$, which follows by the definition of algorithm steps from Eq.~\eqref{eq:mod-agd+}.

On the other hand, as $\psi$ is $(\lambda, q)$-uniformly convex, %with modulus $\lambda,$ 
we have that $$D_{\psi}(\vv_k, \vv_{k-1}) \geq \frac{\lambda}{q}\|\vv_k - \vv_{k-1}\|^q.$$ 
Further, if $\lambda = 0,$ we have that $D_{\phi}(\vv_k, \vv_{k-1}) \geq \frac{1}{q}\|\vv_k - \vv_{k-1}\|^q$. Thus:
\begin{align*}
    A_k G_k - A_{k-1}G_{k-1} &\leq \Big(M_k\frac{{a_k}^q}{{A_k}^{q-1}} - \max\{\lambda A_{k-1}, m_0\}\Big)\frac{\|\vv_k - \vv_{k-1}\|^q}{q} + \frac{A_k \delta_k}{2}\\
    &\leq \frac{A_k \delta_k}{2},
\end{align*}
as $\frac{{a_k}^q}{{A_k}^{q-1}}\leq \frac{\max\{\lambda A_{k-1}, m_0\}}{ M_k}.$ 
\qed\end{proof}

We are now ready to state and prove the main result from this section. 
%
%\todo[inline]{JD: We need to discuss what happens when $\lambda > L$, as then the bounds don't necessarily make sense. The catch is in the condition on step sizes $a_k,$ as, by definition, it has to be $A_k \geq a_k$. Not sure how to address this without over-complicating the discussion. Any ideas?\\
%CG: It seems that the condition $a_k\leq A_k$ corresponds to the $k$ in the statement to be $\geq 1$ (ignoring the log). If that's the case then there's no problem really (if the ineq is not satisfied, then ALGO makes a $\tilde O(1)$ steps and solves the problem). If we remove the $\Omega(\cdot)$ in the ''standard arguments'' part it should be clearer.}
%
\begin{theorem}\label{thm:comp-opt-gap-conv}
Let $\Bar{f}(\vx) = f(\vx) + \psi(\vx),$ where $f$ is convex and $(L, \kappa)$-weakly smooth w.r.t.~a norm $\|\cdot\|$, $\kappa \in (1, 2],$ and $\psi$ is $q$-uniformly convex with constant $\lambda \geq 0$ w.r.t.~the same norm for some $q\geq 2$. Let $\vxb^{\star}$ be the minimizer of $\Bar{f}.$ Let $\vx_k, \vv_k, \vy_k$ evolve according to Eq.~\eqref{eq:mod-agd+} for an arbitrary initial point $\vx_0 \in \cx,$ where $A_0M_0 = m_0$, ${{a_k}^q}\leq \frac{\max\{\lambda {A_{k-1}}^q, m_0 {A_k}^{q-1}\}}{ M_k}$ for $k \geq 1,$ and $M_k = \Big[\frac{2(q-\kappa)}{q\kappa \delta_k}\Big]^{\frac{q-\kappa}{\kappa}}L^{\frac{q}{\kappa}}$,  for $\delta_k > 0$ and $k\geq 0.$ Then, $\forall k \geq 1:$
$$
    \Bar{f}(\vy_k) - \Bar{f}(\vxb^{\star}) \leq \frac{2A_0M_0 \phi(\vxb^{\star}) + \sum_{i=0}^k A_i \delta_i}{2 A_k}.
$$
In particular, for any $\epsilon > 0,$ setting $\delta_k = \frac{a_k}{A_k}\epsilon$, for $k\geq 0,$  $a_0 = A_0= 1,$ 
and ${{a_k}^q} = \frac{\max\{\lambda {A_{k-1}}^q, m_0 {A_k}^{q-1}\}}{ M_k}$
for $k \geq 1,$ we have that $\Bar{f}(\vy_k) - \Bar{f}(\vxb^{\star}) \leq \epsilon$ after at most
\begin{align*}
    k = O\bigg(\min\bigg\{ &\Big(\frac{1}{\epsilon }\Big)^{\frac{q-\kappa}{q\kappa - q + \kappa}}\Big(\max\Big\{\frac{L^{\frac{q}{\kappa}}}{ \lambda } , 1\Big\}\Big)^{\frac{\kappa}{q\kappa - q + \kappa}}\log\Big(\frac{L \phi(\vxb^{\star})}{\epsilon}\Big),\\%\;  
    &\Big(\frac{L}{\epsilon}\Big)^{\frac{q}{q\kappa - q + \kappa}}\big(\phi(\vxb^{\star})\big)^{\frac{\kappa}{q\kappa -q + \kappa}} \bigg\}\bigg)
\end{align*}
iterations.
\end{theorem}
\begin{proof}
The first part of the theorem follows immediately by combining Lemma~\ref{lemma:init-gap} and Lemma~\ref{lemma:change-in-gap}. 

For the second part, we have 
$$
    \Bar{f}(\vy_k) - \Bar{f}(\vxb^{\star}) \leq \frac{A_0 M_0 \phi(\vxb^{\star})}{A_k} + \frac{\epsilon}{2},
$$
so all we need to show is that, under the step size choice from the theorem statement, we have $\frac{A_0 M_0 \phi(\vxb^{\star})}{A_k} \leq \frac{\epsilon}{2}.$ 

As $A_0 = a_0 = 1,$ we have that $\delta_0 = \epsilon$ and 
\begin{equation}\label{eq:M_0}
    M_0 = \Big[\frac{2(q-\kappa)}{q\kappa \epsilon}\Big]^{\frac{q-\kappa}{\kappa}}L^{\frac{q}{\kappa}}.
\end{equation}
It remains to bound the growth of $A_k.$ In this case, by theorem assumption, we have ${{a_k}^q} = \frac{\max\{\lambda {A_{k-1}}^q, m_0 {A_k}^{q-1}\}}{ M_k}$. Thus, (i) $\frac{{a_k}^q}{{A_{k-1}}^{q}} \geq  \frac{\lambda}{ M_k}$ and (ii) $\frac{{a_k}^q}{{A_k}^{q-1}} \geq \frac{ m_0}{ M_k}$, and the growth of $A_k$ can be bounded below as the maximum of growths determined by these two cases.

Consider $\frac{{a_k}^q}{{A_{k-1}}^{q}} \geq  \frac{{\lambda}}{ M_k}$ first. 
As $\delta_k = \frac{a_k}{A_k}\epsilon$ and  $M_k = \Big[\frac{2(q-\kappa)}{q\kappa \delta_k}\Big]^{\frac{q-\kappa}{\kappa}}L^{\frac{q}{\kappa}}$, the condition $\frac{{a_k}^q}{{A_k}^{q-1}A_{k-1}}\geq \frac{\lambda}{ M_k}$ can be equivalently written as: %\textcolor{red}{Below $+1$ goes in the exponent?}
$$  
    \frac{{a_k}^{q - \frac{q}{\kappa}+1}}{{A_{k-1}}^{q - \frac{q}{\kappa} + 1}} \geq \Big[\frac{2(q-\kappa)}{q\kappa \epsilon }\Big]^{-\frac{q-\kappa}{\kappa}}\frac{\lambda}{L^{\frac{q}{\kappa}}}.
$$
%As $A_{k-1} \leq A_k,$ the growth is at least as fast as determined by:
%$$  
%    \Big(\frac{{a_k}}{A_{k-1}}\Big)^{\frac{q\kappa - q + \kappa}{\kappa}} = \Big[\frac{2(q-\kappa)}{q\kappa \epsilon }\Big]^{-\frac{q-\kappa}{\kappa}}\frac{\lambda}{L^{\frac{q}{\kappa}}},
%$$
%or, equivalently
Hence,
$$
    \frac{{a_k}}{A_{k-1}} \geq \Big[\frac{2(q-\kappa)}{q\kappa \epsilon }\Big]^{-\frac{q-\kappa}{q\kappa - q + \kappa}}\Big(\frac{\lambda}{ L^{\frac{q}{\kappa}}} \Big)^{\frac{\kappa}{q\kappa - q + \kappa}}. 
$$
As $a_k = A_k - A_{k-1},$ it follows that $\frac{A_k}{A_{k-1}} \geq 1 + \Big[\frac{2(q-\kappa)}{q\kappa \epsilon }\Big]^{-\frac{q-\kappa}{q\kappa - q + \kappa}}\Big(\frac{\lambda}{ L^{\frac{q}{\kappa}}} \Big)^{\frac{\kappa}{q\kappa - q + \kappa}},$ further leading to
\begin{align*}
    A_k \geq \bigg(1 + \Big[\frac{2(q-\kappa)}{q\kappa \epsilon }\Big]^{-\frac{q-\kappa}{q\kappa - q + \kappa}}\Big(\frac{\lambda}{ L^{\frac{q}{\kappa}}} \Big)^{\frac{\kappa}{q\kappa - q + \kappa}}\bigg)^k. 
\end{align*}
%Standard arguments then lead to:
%\todo[inline]{CG: Being picky here again. The fact that $\Big[\frac{2(q-\kappa)}{q\kappa  }\Big]^{-\frac{q-\kappa}{q\kappa - q + \kappa}}$ goes in the exponent, shouldn't affect the constant in the exponent at the rhs of \eqref{eq:comp-min-Ak-1}. I don't think that matters, but technically there's also a $\Theta(\cdot)$ on the exponent.\\
%JD: I'm not sure I understand. $\Big[\frac{2(q-\kappa)}{q\kappa  }\Big]^{\frac{q-\kappa}{q\kappa - q + \kappa}}$ is always bounded above by a constant, given that $\kappa \in (1, 2]$ and $q \geq 2.$ I think it should also be bounded below by a const, but if not, we can change $\Theta$ to $\Omega$ and then it should be fine.\\
%CG: Right, but isn't then $\exp\{\Theta(\cdot)\}$, rather than $\Theta(\exp\{\cdot\})$?}
%\begin{equation}\label{eq:comp-min-Ak-1}
%    A_k = \Omega\bigg( \exp\bigg( \Big[\frac{2(q-\kappa)}{q\kappa \epsilon }\Big]^{-\frac{q-\kappa}{q\kappa - q + \kappa}}\Big(\frac{\lambda}{ L^{\frac{q}{\kappa}}} \Big)^{\frac{\kappa}{q\kappa + q - \kappa}} k\bigg) \bigg) =  \exp\bigg(\Omega\bigg( \Big[\frac{1}{\epsilon }\Big]^{-\frac{q-\kappa}{q\kappa - q + \kappa}}\Big(\frac{\lambda}{ L^{\frac{q}{\kappa}}} \Big)^{\frac{\kappa}{q\kappa - q + \kappa}} k\bigg) \bigg).
%\end{equation}
%\todo[inline]{Now I understand why you (correctly) get $\Omega(\exp(\cdot))$. We can keep it like that, but that Omega is actually $A_0=1$, right?}

On the other hand, the condition $\frac{{a_k}^q}{{A_k}^{q-1}} \geq \frac{m_0}{ M_k}$ can be equivalently written as:
$$
    \frac{{a_k}^{\frac{q\kappa - q}{\kappa}+1}}{A_k^{\frac{q\kappa - q}{\kappa}}} \geq \frac{m_0}{L^{\frac{q}{\kappa}}} \Big[\frac{q\kappa \epsilon}{2(q-\kappa)}\Big]^{\frac{q-\kappa}{\kappa}} = 1,
$$
where we have used the definition of $m_0,$ which implies
\begin{equation}\label{eq:comp-min-Ak-2}
    A_k 
    = \Omega\Big(k^{\frac{q\kappa - q + \kappa}{\kappa}}\Big),
\end{equation}
and further leads to the claimed bound on the number of iterations.
\qed\end{proof}

Let us point out some special cases of the bound from Theorem~\ref{thm:comp-opt-gap-conv}. When $f$ is smooth ($\kappa = 2$) and $\psi$ is $q$-uniformly convex, assuming $L^{q/2}\geq \lambda,$ the bound simplifies to 
\begin{equation}\label{eq:k-smooth+unif-cvx}
    k = O\bigg(\min\bigg\{\Big(\frac{1}{\epsilon }\Big)^{\frac{q-2}{q + 2}}\Big(\frac{L^{\frac{q}{2}}}{ \lambda } \Big)^{\frac{2}{q + 2}}\log\Big(\frac{L \phi(\vxb^{\star})}{\epsilon}\Big), \;  \Big(\frac{L}{\epsilon}\Big)^{\frac{q}{q+2}}\big(\phi(\vxb^{\star})\big)^{\frac{2}{q+2}}\bigg\}\bigg).
\end{equation}
In particular, if $\psi$ is strongly convex ($q=2$), we recover the same bound as in the Euclidean case:
\begin{equation}\label{eq:k-smooth+str-cvx}
    k = O\bigg(\min\bigg\{\sqrt{\frac{L}{ \lambda }}\log\Big(\frac{L \phi(\vxb^{\star})}{\epsilon}\Big),\; \sqrt{\frac{L\phi(\vxb^{\star})}{\epsilon}}\bigg\} \bigg).
\end{equation}
Note that this result uses smoothness of $f$ and strong convexity of $\psi$ with respect to the same but arbitrary norm $\|\cdot\|$. Because we do not require the same function to be simultaneously smooth and strongly convex w.r.t.~$\|\cdot\|$, the resulting ``condition number''  $\frac{L}{ \lambda }$ can be dimension-independent even for non-Euclidean norms (in particular, this will be possible for any $\ell_p$ norm with $p \in (1, 2]$). Observe further that it is generally possible for the ``condition number''  $\frac{L}{ \lambda }$ to be smaller than one. In that case, as the convergence bound in Theorem~\ref{thm:comp-opt-gap-conv} depends on $\max\{\frac{L}{ \lambda }, 1\}$, the bound becomes independent of the condition number (although it still depends on $L$).
%Another consideration is that for $p > 2$ there do not exist any strongly convex functions $\psi$ that simultaneously have a constant strong convexity parameter w.r.t.~$\|\cdot\|_p$ and whose Bregman divergence satisfies $D_{\psi}(\vy, \vx) \leq C \|\vy - \vx\|_p^2$ for a constant $C$ that does not grow polynomially with $d$ when $p$ is not trivially close to 2. This means that even if the first issue is resolved (by, e.g., devising an efficient algorithm for the sum of two functions that are smooth w.r.t.~$\|\cdot\|_p$ and strongly convex w.r.t.~$\|\cdot\|_p$, respectively) 

Because $\Bar{f}$ is $q$-uniformly convex, Theorem~\ref{thm:comp-opt-gap-conv} also implies a bound on $\|\vy_k - \vxb^{\star}\|$ whenever $\lambda > 0,$ as follows.
\begin{corollary}\label{cor:comp-opt-dist-to-opt}
Under the same assumptions as in Theorem~\ref{thm:comp-opt-gap-conv}, and assuming, in addition, that $\lambda > 0,$ we have that $\|\vy_k - \vxb^{\star}\|\leq \bar{\epsilon}$ after at most
$$
    k = O\bigg( \Big(\frac{q}{\lambda\bar{\epsilon}^q }\Big)^{\frac{q-\kappa}{q\kappa - q + \kappa}}\Big(\frac{L^{\frac{q}{\kappa}}}{ \lambda } \Big)^{\frac{\kappa}{q\kappa - q + \kappa}}\log\Big(\frac{q L \phi(\vxb^{\star})}{\bar{\epsilon}^q\lambda}\Big)\bigg)
$$
iterations.
\end{corollary}
\begin{proof}
By $q$-uniform convexity of $\Bar{f}$ and $\zeros \in \partial f(\vxb^{\star})$ (as $\vxb^{\star}$ minimizes $\Bar{f}$), we have
$$
    \|\vy_k - \vxb^{\star}\|^q \leq \frac{q}{\lambda}(\bar{f}(\vy_k) - \Bar{f}(\vxb^{\star})).
$$
Thus, it suffices to apply the bound from Theorem~\ref{thm:comp-opt-gap-conv} with the accuracy parameter ${\epsilon} = \frac{\lambda\bar{\epsilon}^q}{q}.$
\qed\end{proof}

%%%%%%%%%%%%%%%%%%%%%
\subsection{Computational Considerations}\label{sec:comp-considerations}

At a first glance, the result from Theorem~\ref{thm:comp-opt-gap-conv} may seem of limited applicability, as there are potentially four different parameters ($L, \kappa, \lambda, q$) that one would need to tune. However, we now argue that this is not a constraining factor. First, for most of the applications in which one would be interested in using this framework, function $\psi$ is a regularizing function with known uniform convexity parameters $\lambda$ and $q$ (see Section~\ref{sec:apps} for several illustrative examples). Second, the knowledge of parameters $L$ and $\kappa$ is not necessary for our results; we presented the analysis assuming the knowledge of these parameters to not over-complicate the exposition.

In particular, the only place in the analysis where the $(L, \kappa)$ smoothness of $f$ is used is in the inequality
\begin{equation}\label{eq:M_k-condition}
    f(\vy_k) \leq f(\vx_k) + \innp{\nabla f(\vx_k), \vy_k - \vx_k} + \frac{M_k}{q}\|\vy_k - \vx_k\|^q + \frac{\delta_k}{2}.
\end{equation}
But instead of explicitly computing the value of $M_k$ based on $L, \kappa,$ one could maintain an estimate of $M_k$, double it whenever the inequality from Eq.~\eqref{eq:M_k-condition} is not satisfied, and recompute all iteration-$k$ variables. This is a standard {\em line-search} trick employed in optimization (see, e.g.,~\cite{nesterov2015universal}). Observe that, due to $(L, \kappa)$-weak smoothness of $f$ and Lemma~\ref{lemma:p-weakly-smooth-approx}, there exists a sufficiently large $M_k$ for any value of $\delta_k$. In particular, under the choice $\delta_k = \frac{a_k}{A_k}\epsilon$ from Theorem~\ref{thm:grad-norm}, the total number of times that $M_k$ can get doubled is logarithmic in all of the problem parameters, which means that it can be absorbed in the overall convergence bound from Theorem~\ref{thm:comp-opt-gap-conv}. 

Finally, the described algorithm (Generalized AGD+ from Eq.~\eqref{eq:mod-agd+}) can be efficiently implemented only if the minimization problems defining $\vv_k$ can be solved efficiently (preferably in closed form, or with $\Tilde{O}(d)$ arithmetic operations). This is indeed the case for most problems of interest. In particular, when $\psi$ is uniformly convex, we will typically take $\phi(\vu)$ to be the Bregman divergence $D_{\psi}(\vu, \vx_0).$ Then, the computation of $\vv_k$ boils down to solving problems of the form \eqref{eq:psi-simple}, i.e., $\min_{\vu \in \cx}\{\innp{\vz, \vx} + \psi(\vx)\},$ for a given $\vz.$ Such problems are efficiently solvable whenever the convex conjugate of $\psi + I_{\cx}$, where $I_{\cx}$ is the indicator function of the closed convex set $\cx,$ is efficiently computable, in which case the minimizer is $\nabla(\psi + I_{\cx})^*(\vz)$. In particular, for $\cx = \bE$ and $\psi(\vx) = \frac{1}{q}\|\cdot\|^q$, $q > 1,$ (a common choice for our applications of interest; see Section~\ref{sec:apps}), the minimizer is computable in closed form as $\nabla \big(\frac{1}{q_{\ast}}\|\vz\|_{*}^{q_{\ast}}\big),$ where $q_{\ast} = \frac{q}{q-1}$ is the exponent dual to $q.$ 
This should be compared to the computation of proximal maps needed in~\cite{nesterov2013gradient, allen2017katyusha}, where the minimizer would be the gradient of the infimal convolution of $\psi$ and the \markupdelete{Euclidean} \markupadd{squared} norm \markupdelete{squared}, for which there are much fewer efficiently computable examples. Note that such an assumption would be sufficient for our algorithm to work in the Euclidean case (by taking $\phi(\vu) = \frac{1}{2}\|\vu - \vx_0\|_2^2$); however, it is not necessary.

%\todo[inline]{JD: Need to add a discussion about computing $\vv_k$ in the algorithm.}

%%%%%%%%%%%%%%%%%%%%%%%%%%%%%%%%%
\section{Minimizing the Gradient Norm in $\ell_p$ and $\mathrm{Sch}_p$ Spaces}\label{sec:small-grad-norm}

We now show how to use the result from Theorem~\ref{thm:comp-opt-gap-conv} to obtain near-optimal convergence bounds for minimizing the norm of the gradient. In particular, assuming that $f$ is $(L, \kappa)$-weakly smooth w.r.t.~$\|\cdot\|_p,$ to obtain the desired results, we apply Theorem~\ref{thm:comp-opt-gap-conv} to function $\Bar{f}(\cdot) = f(\cdot) + \lambda \psi_p(\cdot),$ where 
\begin{equation}\label{eq:psi_p}
    \psi_p(\vx)  = \begin{cases}
                        \frac{1}{2(p-1)}\|\vx-\vx_0\|_p^2, &\text{ if } p \in (1, 2],\\
                        %\frac{\min\{\frac{1}{p-1}, 2\log(d)\}}{2}\|\vx - \vx_0\|_p^2, &\text{ if } p \in (1, 2],\\
                        \frac{1}{p} \|\vx - \vx_0\|_p^p, &\text{ if } p \in(2,+\infty).
                    \end{cases}
\end{equation}
Function $\psi_p$ is then $(1, \max\{2, p\})$-uniformly convex.  The proof of strong convexity of $\psi_p$ when $1<p\leq 2$ can be found, e.g., in \cite[Example 5.28]{Beck:2017}. For $2<p<+\infty$, $\psi_p$ is a separable function, hence its $p$-uniform convexity can be proved from the duality between uniform convexity and uniform smoothness \cite{Zalinescu:1983} and direct computation.  
These $\ell_p$ results also have spectral analogues, given by the Schatten spaces $\Schatten_p=(\rr^{d\times d},\|\cdot\|_{\Schatten,p})$. Here, the functions below can be proved to be $(1,\max\{2,p\})$-uniformly convex, which is a consequence of sharp estimates of uniform convexity for Schatten spaces \cite{Ball:1994, juditsky2008large}
\begin{equation}\label{eq:psi_p-Schatten}
    \Psi_{\Schatten,p}(\vx)  = \begin{cases}
                        \frac{1}{2(p-1)}\|\vx-\vx_0\|_{\Schatten,p}^2, &\text{ if } p \in (1, 2],\\
                        %\frac{\min\{\frac{1}{p-1}, 2\log(d)\}}{2}\|\vx - \vx_0\|_p^2, &\text{ if } p \in (1, 2],\\
                        \frac{1}{p} \|\vx - \vx_0\|_{\Schatten,p}^p, &\text{ if } p \in(2,+\infty).
                    \end{cases}
\end{equation}
Finally, both for $\ell_1$ and $\Schatten_1$ spaces, our algorithms can work on the equivalent norm with power $p=\ln d/(\ln d-1).$ The cost of this change of norm is at most logarithmic in $d$ for the diameter and strong convexity constants. Similarly, our results also extend to the case $p=\infty$, by similar considerations (here, using exponent $p=\ln d$). %We will see however that the benefits of acceleration are completely lost in the $\ell_{\infty}$ case, and they effectively work as in the noncomposite smooth setting. 
%\todo[inline]{CG: Check if you like the additions here.\\
%JD: Yes. I just changed ``this'' to $\ell_{\infty}$ in the last sentence, because it was not clear if it was referring only to $\ell_{\infty}$ or both $\ell_{\infty}$ and $\ell_1.$\\
%CG: Looks good!}
%Here, we utilize the fact that $\frac{1}{2}\|\vx - \vx_0\|_p^2$ is $\max\{{p-1}, \frac{1}{2\log(d)}\}$-strongly convex when $p \in (1, 2]$, while $\frac{1}{p}\|\vx - \vx_0\|_p^p$ is $p$-uniformly convex with modulus $1$ for $p > 2$ \textcolor{red}{[ADD APPROPRIATE CITATION(S)]}.

To obtain the results for the norm of the gradient in $\ell_p$ spaces, we can apply Theorem~\ref{thm:comp-opt-gap-conv} with $\phi(\vx) = \psi_p(\vx),$ where $\psi_p$ is specified in Eq.~\eqref{eq:psi_p}. The result is summarized in the following theorem. The same result can be obtained for $\Schatten_p$ spaces, by following the same argument as in Theorem~\ref{thm:grad-norm} below, which we omit for brevity.

\begin{theorem}\label{thm:grad-norm}
Let $f$ be a convex, $(L, \kappa)$- weakly smooth function w.r.t.~a norm $\|\cdot\|_p$, where $p \in (1, \infty).$ Then, for any $\epsilon > 0,$ Generalized AGD+ from Eq.~\eqref{eq:mod-agd+}, initialized at some point $\vx_0 \in \rr^d$ and applied to $\bar{f} = f + \lambda\psi_p,$ where  $\psi_p$ is specified in Eq.~\eqref{eq:psi_p}, 
$$
     \lambda = 
    \begin{cases}
        \frac{\epsilon(p-1)}{2\|\vx^{\star} - \vx_0\|_p}, &\text{ if } p \in (1, 2], \\
        \frac{\epsilon}{2\|\vx^{\star} - \vx_0\|_p^{p-1}}, &\text{ if } p \in (2, \infty), 
    \end{cases}
$$
and with the choice $\phi = \psi_p,$ constructs a point $\vy_k$ with $\|\nabla f(\vy_k)\|_{p_{\ast}} \leq \epsilon$ in at most
\begin{align*}
    & k = \\
    &\begin{cases}
    O\bigg(\Big(\frac{2L}{\epsilon}\Big)^{\frac{\kappa}{(\kappa - 1)(3\kappa - 2)}}\Big(\frac{\kappa^{2\kappa}}{(\kappa - 1)^{2\kappa}} \frac{\|\vx^{\star} - \vx_0\|_p^2}{(p-1)^{{\kappa}}}\Big)^{\frac{1}{3\kappa - 2}}\log\Big(\frac{L\|\vx^{\star} - \vx_0\|_p}{(p-1)\epsilon}\Big)\bigg), &\text{ if } p \in (1, 2],\\
    O\bigg(\Big(\frac{2L\|\vx^{\star} - \vx_0\|_p}{\epsilon}\Big)^{\frac{\kappa(p-1)}{p\kappa - p+ \kappa}} \Big(\frac{\kappa}{\kappa - 1}\Big)^{\frac{p}{p\kappa - p + \kappa}} \log\Big(\frac{L\|\vx^{\star} - \vx_0\|_p^p}{\epsilon}\Big) \bigg), &\text{ if } p \in (2, \infty),
    \end{cases}
\end{align*}
iterations. In particular, when $\kappa = 2$ (i.e., when $f$ is $L$-smooth):
$$
    k = \begin{cases}
        \widetilde{O}\bigg(\sqrt{\frac{L\|\vx^{\star} - \vx_0\|_p}{\epsilon}}\bigg), &\text{ if } p \in (1, 2],\\
        \widetilde{O}\bigg(\Big(\frac{L\|\vx^{\star} - \vx_0\|_p}{\epsilon}\Big)^{\frac{2(p-1)}{p+2}}\bigg), &\text{ if } p \in (2, \infty),
    \end{cases}
$$
where $\widetilde{O}$ hides logarithmic factors in $L$, $\|\vx-\vx_0\|_p$, $\frac{1}{p-1}$ and $1/\epsilon$.
\end{theorem}
%\todo[inline]{CG: Can we specify that logarithmic factors are in $L$, $\|\vx-\vx_0\|_p$, $\frac{1}{p-1}$ and $1/\epsilon$?}
%
\begin{proof}
Let us first relate $\|\vxb^{\star} - \vx_0\|_p$ to $\|\vx^{\star} - \vx_0\|_p,$ where $\vxb^{\star} = \argmin_{\vx \in \rr^d} \bar{f}(\vx),$ $\vx^{\star} \in \argmin_{\vx \in \rr^d} {f}(\vx).$ By %uniform convexity and 
the definition of $\bar{f}$:
\begin{align*}
    0%\leq \frac{\lambda}{q}\|\vxb^{\star} - \vx^{\star}\|_p^q 
    &\leq \bar{f}(\vx^{\star}) - \bar{f}(\vxb^{\star})\\
    &= f(\vx^{\star}) - f(\vxb^{\star}) + \lambda \psi_p(\vx^{\star}) - \lambda \psi_p(\vxb^{\star})\\
    &\leq \lambda \psi_p(\vx^{\star}) - \lambda \psi_p(\vxb^{\star}).
\end{align*}
It follows that 
$$
    \psi_p(\vxb^{\star}) \leq \psi_p(\vx^{\star}).
$$
Thus, using the definition of $\psi_p,$
\begin{equation}\label{eq:bnd-init-dist}
    \|\vxb^{\star} - \vx_0\|_p \leq \|\vx^{\star} - \vx_0\|_p.
\end{equation}

By triangle inequality and $\vxb^{\star} = \argmin_{\vx \in \rr^d} \bar{f}(\vx)$ (which implies $\nabla \bar{f}(\vx^{\star}) = \zeros$),
\begin{align}
    \|\nabla f(\vy_k)\|_{p_{\ast}} &\leq \|\nabla f(\vy_k) - \nabla f(\vxb^{\star})\|_{p_{\ast}} + \|\nabla f(\vxb^{\star})\|_{p_{\ast}}\notag\\
    &= \|\nabla f(\vy_k) - \nabla f(\vxb^{\star})\|_{p_{\ast}} + \|\nabla \bar{f}(\vxb^{\star}) - \lambda \nabla \psi_p(\vxb^{\star})\|_{p_{\ast}}\notag\\
    &=  \|\nabla f(\vy_k) - \nabla f(\vxb^{\star})\|_{p_{\ast}} + \lambda\| \nabla \psi_p(\vxb^{\star})\|_{p_{\ast}}. \label{eq:bnd-on-grad-gen-p}
\end{align}
As $f$ is convex and $(L, \kappa)$ weakly smooth, using Lemma~\ref{lemma:weak-smooth-cocoercivity}, we also have:
\begin{align}
    \frac{\kappa - 1}{L^{\frac{1}{\kappa -1}} \kappa}\|\nabla f(\vy_k) - \nabla f(\vxb^{\star})\|_{p_{\ast}}^{\frac{\kappa}{\kappa - 1}} \leq\; & f(\vy_k) - f(\vxb^{\star})
    - \innp{\nabla f(\vxb^{\star}), \vy_k - \vxb^{\star}}\notag\\
    =\; & \bar{f}(\vy_k) - \bar{f}(\vxb^{\star}) -\lambda\psi_p(\vy_k) + \lambda\psi_p(\vxb^{\star})\notag\\
    &- \innp{\nabla \bar{f}(\vxb^{\star}) - \lambda\nabla \psi_p(\vxb^{\star}), \vy_k - \vxb^{\star}}\notag\\
    =\; & \bar{f}(\vy_k) - \bar{f}(\vxb^{\star}) - \lambda\big(\psi_p(\vy_k) - \psi_p(\vxb^{\star})\notag\\
    & - \innp{\nabla \psi_p(\vxb^{\star}), \vy_k - \vxb^{\star}}\big) \notag\\
    \leq\; & \bar{f}(\vy_k) - \bar{f}(\vxb^{\star}), \label{eq:grad-to-fn}
\end{align}
where the second line uses $\bar{f} = f + \psi_p$, the third line follows by $\nabla \bar{f}(\vxb^{\star}) = 0$ (as $\vxb^{\star} = \argmin_{\vx \in \rr^d} \bar{f}(\vx)$), and the last inequality is by convexity of $\psi_p.$

From Eqs.~\eqref{eq:bnd-on-grad-gen-p} and \eqref{eq:grad-to-fn}, to obtain $\|\nabla f(\vy_k)\|_{p_{\ast}}\leq \epsilon,$ it suffices that $\lambda\| \nabla \psi_p(\vxb^{\star})\|_{p_{\ast}} \leq \frac{\epsilon}{2}$ and $\bar{f}(\vy_k) - \bar{f}(\vxb^{\star}) \leq \big(\frac{\epsilon}{2}\big)^{\frac{\kappa}{\kappa - 1}}\frac{\kappa - 1}{L^{\frac{1}{\kappa -1}} \kappa}.$ 

The first condition determines the value of $\lambda.$ Using Proposition~\ref{prop:duality-map-of-p}, $\lambda\| \nabla \psi_p(\vxb^{\star})\|_{p_{\ast}} \leq \frac{\epsilon}{2}$ is equivalent to
$$
\begin{cases}
    \frac{\lambda}{p-1}\|\vxb^{\star} - \vx_0\|_{p} \leq \frac{\epsilon}{2}, &\text{ if }  p \in (1, 2]\\
    {\lambda}\|\vxb^{\star} - \vx_0\|_{p}^{p-1} \leq \frac{\epsilon}{2}, &\text{ if }  p \in (2, \infty).
\end{cases}
$$
Using Eq.~\eqref{eq:bnd-init-dist}, it suffices that:
\begin{equation}\label{eq:final-choice-lambda}
    \lambda = 
    \begin{cases}
        \frac{\epsilon(p-1)}{2\|\vx^{\star} - \vx_0\|_p}, &\text{ if } p \in (1, 2], \\
        \frac{\epsilon}{2\|\vx^{\star} - \vx_0\|_p^{p-1}}, &\text{ if } p \in (2, \infty). 
    \end{cases}
\end{equation}
Using the choice of $\lambda$ from Eq.~\eqref{eq:final-choice-lambda}, it remains to apply Theorem~\ref{thm:comp-opt-gap-conv} to bound the number of iterations until $\bar{f}(\vy_k) - \bar{f}(\vxb^{\star}) \leq \big(\frac{\epsilon}{2}\big)^{\frac{\kappa}{\kappa - 1}}\frac{\kappa - 1}{L^{\frac{1}{\kappa -1}} \kappa}.$ Applying Theorem~\ref{thm:comp-opt-gap-conv}, we have:
$$
    k = O\bigg( \Big(\frac{2^{\frac{\kappa}{\kappa -1}} L^{\frac{1}{\kappa - 1}} \kappa}{\epsilon^{\frac{\kappa}{\kappa -1}}(\kappa - 1) }\Big)^{\frac{q-\kappa}{q\kappa - q + \kappa}}\Big(\frac{L^{\frac{q}{\kappa}}}{ \lambda } \Big)^{\frac{\kappa}{q\kappa - q + \kappa}}\log\Big(\frac{2^{\frac{\kappa}{\kappa -1}}L^2 \kappa\psi_p(\vxb^{\star})}{\epsilon^{\frac{\kappa}{\kappa -1}}(\kappa -1)}\Big) \bigg).
$$
It remains to plug in the choice of $\lambda$ from Eq.~\eqref{eq:final-choice-lambda}, $q = \max\{p, 2\},$ and simplify.
\qed\end{proof}

\markupadd{Although the proof of Theorem~\ref{thm:grad-norm} appears fairly simple, we now discuss why it is not obvious in light of similar existing results for the Euclidean norm~\cite{nesterov2012make}. In Euclidean settings, one uses an accelerated method to minimize the smooth and strongly convex objective $\bar{f}(\vx) = f(\vx) + \frac{\lambda}{2}\|\vx - \vx_0\|_2^2$ for $\lambda = \Theta\big( \frac{\epsilon}{\|\vx^* - \vx_0\|_2}\big).$ Because the regularized function $\bar{f}$ in this setting is $(L + \lambda)$-smooth, we can conclude that the same algorithm ensures $\|\nabla \bar{f}(\vx)\|_2 \leq \epsilon/2$  within $O(\sqrt{\frac{L}{\lambda}}\log(\frac{L \|\vx^* - \vx_0\|_2}{\epsilon}))$ iterations. To guarantee that $\|\nabla f(\vx)\|_2 \leq \epsilon,$ it then suffices to use the triangle inequality:
\begin{equation}\label{eq:Euc-triangle}
    \|\nabla f(\vx)\|_2 \leq \|\nabla \bar{f}(\vx)\|_2 + \lambda \|\vx - \vx_0\|_2.
\end{equation}
This approach does not directly extend to $\ell_p$ norms (even when $\|\cdot\|_p^2$ is strongly convex, that is, when $p \in (1, 2)$), for the following reasons. First, the composite function $\bar{f}(\vx) = f(\vx) + \frac{\lambda}{2}\|\vx - \vx_0\|_p^2$ is not smooth, so it is unclear how to directly argue that $\|\nabla \bar{f}(\vx)\|_{p^*}$ is small when $\bar{f}(\vx) - \bar{f}(\vxb^*)$ is small. This is the reason why our proof never uses an equivalent of \eqref{eq:Euc-triangle} but instead applies the less obvious triangle inequality from~\eqref{eq:bnd-on-grad-gen-p} and then leverages the definitions of $\bar{f}$ and $\vxb^*.$ The second term that comes from applying either of the two triangle inequalities, $\|\nabla \big(\frac{\lambda}{2}\|\vx - \vx_0\|_p^2\big)\|_{p^*}$, can be bounded using our Proposition~\ref{prop:duality-map-of-p} and it is crucial here that the regularizer we use is of the form $\frac{\lambda}{2}\|\vx - \vx_0\|_p^2;$ otherwise, for an alternative regularizer chosen as the Bregman divergence of a strongly convex function, we would need to use smoothness, which for $p$ not trivially close to 2 would scale polynomially with the dimension, leading to polynomial in $d$ scaling of the parameter $\lambda$ (and consequently the convergence bound). Second, even efficiently minimizing the complementary composite objectives of this form, as discussed in the introduction, was not clear how to do prior to this work. Finally, it is crucial that in \eqref{eq:grad-to-fn} we relate the norm of the gradient distance to the optimality gap w.r.t.~$\bar{f}$ (as opposed to $f$); otherwise, to have $f(\vy_k) - f(\vx^*)$ that scales with $\epsilon^2$, we would need $\lambda \propto \frac{\epsilon^2}{\|\vx^* - \vx_0\|_2^2},$ leading to a worse complexity bound resulting from the application of Theorem~\ref{thm:comp-opt-gap-conv}.} 

\begin{remark}
Observe that, as the gradient norm minimization relies on the application of Theorem~\ref{thm:comp-opt-gap-conv}, the knowledge of parameters $L$ and $\kappa$ is not needed, as discussed in Section~\ref{sec:comp-considerations}. The only parameter that needs to be determined is $\lambda,$ which cannot be known in advance, as it would require knowing the initial distance to optimum $\|\vx^{\star} - \vx_0\|.$ However, tuning $\lambda$ can be done at the cost of an additional $\log(\frac{\lambda}{\lambda_0})$ multiplicative factor in the convergence bound. In particular, one could start with a large estimate of $\lambda$ (say, $\lambda = \lambda_0 = 1$), run the algorithm, and halt and restart with $\lambda \leftarrow \lambda/2$ each time $\|\nabla \bar{f}(\vy_k)\|_{*} \leq 2 \epsilon$ but $\|\nabla f(\vy_k)\|_* > \epsilon.$ This condition is sufficient because, when $\lambda$ is of the correct order, $\lambda\| \nabla \psi(\vy_k)\|_{*} = O(\lambda\| \nabla \psi(\vxb^{\star})\|_{*}) = O(\epsilon)$, %\todo[]{Can we use here $\epsilon$ instead of $O(\epsilon)$, so the conclusion follows?} 
$\|\nabla f(\vy_k)\|_* \leq \epsilon,$ and $\|\nabla \bar{f}(\vy_k)\|_* \leq \|\nabla f(\vy_k)\|_* + \lambda\| \nabla \psi(\vy_k)\|_{*} = O(\epsilon).$ 
\end{remark}

\markupadd{
\begin{remark}
The approach conducted in this section is more broadly applicable; in fact, it is not even necessary to use the same norm to a power in order to regularize the objective. In Section \ref{sec:apps} we explore an alternative regularization for minimizing the norm of the gradient, in the context of discrete optimal transport. This example also shows the importance of considering non-Euclidean norms for minimizing the norm of the gradient; particularly, the importance of quantifying the gradient error in the $\ell_1$-norm, as this permits a dimension-independent error in a rounding procedure proposed in \cite{Altschuler:2017}.
\end{remark}
}
%%%%%%%%%%%%%%%%%%%%%%%%%%%%%%%%%%%
%%%%%%%%%%%%%%%%%%%%%%%%%%%%%%%%%%
\section{Lower Bounds} \label{sec:LowerBounds}

In this section, we address the question of the optimality of our  algorithmic framework, in a formal oracle model of computation. We  first study the question of minimizing the norm of the gradient, which follows from a simple reduction to the complexity of minimizing the objective function and for which nearly tight lower bounds are known. In this case, the lower bounds show that our resulting algorithms are nearly optimal when $q=\kappa=2$. 

\markupadd{Regarding the complexity of minimizing the norm of the gradient,} in cases where either we have weaker smoothness ($\kappa<2$) or larger uniform convexity exponent ($q>2$), we observe the presence of  polynomial gaps in the complexity w.r.t.~$1/\epsilon$. 
One natural question regarding the aforementioned gaps is whether this is due to a possible suboptimality of the algorithm used for complementary composite minimization (see Eq.~\eqref{eq:mod-agd+}), or due to the approach of minimizing the norm of the gradient via the composite model itself. 
%or the reduction from the solution obtained by this method to obtain a small gradient norm. In this respect, 
Here, we discard the first possibility, showing sharp lower bounds for complementary composite optimization in a new composite oracle model. Our lower bounds show that the complementary composite minimization algorithms are optimal up to factors which depend at most logarithmically on the initial distance to the optimal solution, the target accuracy, and dimension.

Before proceeding to the specific results, we provide a short summary of the classical oracle complexity in convex optimization and some techniques that will be necessary for our results. For more detailed information on the subject, we refer the reader to the thorough monograph~\cite{nemirovski:1983}. In the oracle model of convex optimization, we consider a class of objectives ${\cal F}$, comprised of functions $f:\bE\to\rr$; an oracle ${\cal O}:{\cal F}\times \bE\to \mathbf{F}$ (where $\mathbf{F}$ is a vector space); and a target accuracy, $\epsilon>0$. An algorithm ${\cal A}$ can be described by a sequence of functions $({\cal A}_k)_{k\in\mathbb{N}}$, where 
${\cal A}_{k+1}:(\bE\times\mathbf{F})^{k+1}\to \bE$, so that the algorithm sequentially interacts with the oracle querying points
$$ \vx^{k+1}={\cal A}_{k+1}(\vx^0,{\cal O}(f,\vx^0),\ldots, \vx^k,{\cal O}(f,\vx^k)). $$
The running time of algorithm ${\cal A}$ is given by
the minimum number of queries to achieve some measure of accuracy (with accuracy parameter $\epsilon>0$), and will be denoted by $T({\cal A},f,\epsilon)$. The most classical example in optimization is achieving additive optimality gap bounded by $\epsilon$:
$$T({\cal A},f,\epsilon)=\inf\{k\geq 0: f(\vx^k)\leq f^{\star}+\epsilon\},$$
but other relevant goal for our work is achieving a (dual) norm of the gradient bounded above by $\epsilon$, i.e.,
$$T({\cal A},f,\epsilon)=\inf\{k\geq 0: \|\nabla f(\vx^k)\|_{\ast}\leq\epsilon\}.$$
Given a measure of efficiency $T$, the {\em worst-case oracle complexity} for a problem class ${\cal F}$ endowed with oracle ${\cal O}$, %algorithm $\mathcal{A}$ 
is given by
$$ \Compl({\cal F},{\cal O},\epsilon)=\inf_{\cal A} \sup_{f\in {\cal F}} T(\mathcal{A},f,\epsilon). $$

%%%%%%%%%%%%%%%%
\subsection{Lower Complexity Bounds for Minimizing the Norm of the Gradient}

We now provide lower complexity bounds for minimizing the norm of the gradient. For the sake of simplicity, we can think of these lower bounds for the gradient oracle ${\cal O}(f,x)=\nabla f(\vx)$, but we point out they work more generally for arbitrary {\em local oracles} (more on this in the next subsection). 

In short, we reduce the problem of making the gradient small to that of approximately minimizing the objective.

\begin{proposition} \label{lemma:reduction_min_norm_grad}
Let $f:\mathbf{E}\to\rr$ be a convex and differentiable function, with a global minimizer $\vx^{\ast}$. If $\|\nabla f(\vx)\|_{\ast}\leq \epsilon$ and $\|\vx-\vx^{\ast}\|\leq R$, then $f(\vx)-f(\vx^{\ast})\leq \epsilon R$.
\end{proposition}
\begin{proof}
By convexity of $f$,
$$ f(\vx)-f(\vx^{\ast}) \leq \langle \nabla f(\vx),\vx-\vx^{\ast} \rangle \leq \|\nabla f(\vx)\|_{\ast} \|\vx-\vx^{\ast}\|\leq \epsilon R,
$$
where the second inequality is by duality of norms $\|\cdot\|$ and $\|\cdot\|_*.$
%\textcolor{red}{\em Note: If we consider the constrained setting too, perhaps it's worth changing the hypothesis by $\langle \nabla f(\vx),\vx-\vx^{\ast} \rangle\leq \epsilon$.}
\qed\end{proof}

For the classical problem of minimizing the objective function value, lower complexity bounds for $\ell_p$-setups have been previously studied in both constrained \cite{guzman:2015} and unconstrained \cite{diakonikolas:2019} settings. We summarize those results in the following theorem.\footnote{More precisely, to obtain this result one can use the $p$-norm smoothing construction from \cite[Section 2.3]{guzman:2015}, in combination with the norm term used in \cite[Eq.~(3)]{diakonikolas:2019}. This leads to a smooth objective over an unconstrained domain that provides a hard function class. \markupadd{We also remark for the interested reader that the lower bounds from Theorem \ref{theorem:LB_ell_p_minimization} also apply to randomized algorithms, at least in the case $2\leq p<\infty$ and with a more restrictive dimension regime. For details we refer to \cite{diakonikolas:2019}.}}

\begin{theorem}[From \cite{guzman:2015,diakonikolas:2019}] \label{theorem:LB_ell_p_minimization}
Let $1\leq p\leq \infty$, and consider the problem class of unconstrained minimization with objectives in the class ${\cal F}_{\rr^d,\|\cdot\|_p}({\kappa},L)$, whose minima are attained in ${\cal B}_{\|\cdot\|_p}(0,R)$. 
%$$ \min\{f(\vx) \,:\, \vx\in\rr^d\}; \mbox{ with } f\in {\cal F}_{\rr^d,\|\cdot\|_p}^{\kappa}(L,R).$$
Then, the complexity of achieving additive optimality gap $\epsilon$, for any local oracle, is bounded below by:
\begin{itemize}
    \item $\Omega\Big(\Big(\frac{LR^{\kappa}}{\epsilon [\ln d]^{\kappa-1}}\Big)^{\frac{2}{3\kappa-2}}\Big)$, if $1\leq p <2$;
    \item $\Omega\Big( \Big(\frac{LR^{\kappa}}{\epsilon\min\{p,\ln d\}^{\kappa-1} }\Big)^{\frac{p}{\kappa p+\kappa-p}} \Big)$, if $2\leq p<\infty$; and,
    \item $\Omega\Big(\Big( \frac{LR^{\kappa}}{ \epsilon [\ln d]^{\kappa-1} }\Big)^{\frac{1}{\kappa-1}}\Big)$, if $p=\infty$.
\end{itemize}
The dimension $d$ for the lower bound to hold must be at least as large as the lower bound itself.
\end{theorem}

By combining the reduction from Proposition \ref{lemma:reduction_min_norm_grad} with the lower bounds for function minimization from Theorem~\ref{theorem:LB_ell_p_minimization}, we can now immediately obtain lower bounds for minimizing the (dual of the) $\ell_p$ norm of the gradient, as follows.%we conclude that the upper bounds obtained by our algorithm are unimprovable up to absolute constant factors, when $2\leq p<\infty$; and unimprovable up to a logarithmic in $d$ factor, when $p=\infty$ and $1\leq p<2$. 

%\todo[]{Given the extra scaling with $R$, I prefer to add the full statement.}
\begin{corollary}
%The lower bounds stated in Theorem~\ref{theorem:LB_ell_p_minimization} are also valid for the complexity of making the (dual) norm of the gradient bounded by $\epsilon/R$.
%Under the same assumptions as in Theorem~\ref{theorem:LB_ell_p_minimization}, the oracle complexity of finding a point $\vx$ such that $\|\nabla f(\vx)\|_{p_{\ast}}\leq \epsilon$ is lower bounded by
%\begin{itemize}
    %\item $\Omega\Big(\sqrt{\frac{LR}{\epsilon \ln d}}\Big)$ if $1\leq p <2$;
    %\item $\Omega\Big( \Big(\frac{LR}{\min\{p,\ln d\} \epsilon}\Big)^{p/(p+2)} \Big)$, if $2\leq p<\infty$; and,
    %\item $\Omega\Big( \frac{LR}{ \epsilon \ln d} \Big)$, if $p=\infty$.
%\end{itemize}
%The dimension $d$ for the lower bound to hold must be at least as large as the lower bound itself.
Let $1\leq p\leq \infty$, and consider the problem class with objectives in ${\cal F}_{\rr^d,\|\cdot\|_p}({\kappa},L)$, whose minima are attained in ${\cal B}_{\|\cdot\|_p}(0,R)$.  
Then, the complexity of achieving the dual norm of the gradient bounded by $\epsilon$, for any local oracle, is bounded below by:
\begin{itemize}
    \item $\Omega\Big(\Big(\frac{LR^{\kappa-1}}{\epsilon [\ln d]^{\kappa-1}}\Big)^{\frac{2}{3\kappa-2}}\Big)$, if $1\leq p <2$;
    \item $\Omega\Big( \Big(\frac{LR^{\kappa-1}}{\epsilon\min\{p,\ln d\}^{\kappa-1} }\Big)^{\frac{p}{\kappa p+\kappa-p}} \Big)$, if $2\leq p<\infty$; and,
    \item $\Omega\Big(\Big( \frac{LR^{\kappa-1}}{ \epsilon [\ln d]^{\kappa-1} }\Big)^{\frac{1}{\kappa-1}}\Big)$, if $p=\infty$.
\end{itemize}
The dimension $d$ for the lower bound to hold must be at least as large as the lower bound itself.
\end{corollary}

%\todo[inline]{CG: Update discussion including that $\kappa$ is now a parameter above}
Comparing to the upper bounds from Theorem~\ref{thm:grad-norm}, it follows that for $p \in (1, 2]$ and $\kappa=2$, our bound is optimal up to a $\log(d)\log(\frac{LR}{(p-1)\epsilon})$ factor; i.e., it is near-optimal. Recall that the upper bound for $p=1$ can be obtained by applying the result from Theorem~\ref{thm:grad-norm} with $p = \log(d)/[\log d-1].$ When $p > 2$ and $\kappa=2$, our upper bound is larger than the lower bound by a factor $\Big(\frac{LR}{\epsilon}\Big)^{\frac{p-2}{p+2}}\log(\frac{LR}{\epsilon})(\min\{p, \log(d)\})^\frac{p}{p+2}$. The reason for the suboptimality in the $p >2$ regime comes from the polynomial in $1/\epsilon$ factors in the upper bound for complementary composite minimization from Section~\ref{sec:comp-min}, and it is a limitation of the regularization approach used in this work to obtain bounds for the norm of the gradient. In particular, we believe that it is not possible to obtain tighter bounds via an alternative analysis by using the same regularization approach. Thus, it is an interesting open problem to obtain tight bounds for $p > 2$, and it may require developing completely new techniques.  
Similar complexity gaps are encountered when $\kappa<2;$ however, it is reasonable to suspect that here the lower bounds are not sharp. In particular, when $\kappa=1$, the objective function is not guaranteed to be differentiable (see the discussion below Definition~\ref{def:weak-smoothness} in Section~\ref{sec:prelims}) and points with small subgradients 
%are difficult (if not impossible) to attain
may be a zero measure set,\footnote{As a specific example, consider $f(\vx) = \|\vx\|_1.$ Here, the norm of subgradient is constant and bounded away from zero for any $\vx \neq \zeros$ and even at $\vx^{\star} = \zeros$ the subgradient oracle could return any point from $[-1, 1]^d$.}
making them difficult (if not impossible) to attain; notice that this is not at all reflected in the lower bound. \markupadd{Therefore, two interesting questions arise: first, how to strengthen these lower bounds for weakly smooth function classes; and second, can we study milder accuracy measures in the weakly smooth case? For example, a recent line of work has focused on the task of approximating near stationary points, by use of proximal mappings \cite{Drusvyatskiy:2018,Drusvyatskiy:2021}. We leave these interesting open questions for future research.}
%Therefore, another interesting open problem is to investigate how to strengthen these lower bounds for weakly smooth function classes. %\todo{OK?}
%\todo[inline]{Should we discuss here that our bounds improve upon gradient descent when $p<4$ (maybe we don't want, because the result might look of limited interest). I would definitely stress that $p=1$ was not known and it is an important result (perhaps in the intro).}

%%%%%%%%%%%%%%%%
\subsection{Lower Complexity Bounds for Complementary Composite Minimization}

%\todo[inline]{JD: Can you please provide the correct reference(s) for the basic lower bound below and the theorem statement? Also, can you please organize the material for the new lower bound? I suggest stating the assumptions in an ``assumption'' environment, providing some context for when the lower bounds apply, and moving the proof to an actual proof environment. I would also upgrade the stated proposition to a lemma. This will make it easier to move things around for a possible COLT submission.}

%The bounds we have provided for $\lambda = 0$ are optimal in the $\ell_p$ setups for almost all values of $p$ and $\kappa.$ This is a known result summarized in the following theorem. %

We investigate the (sub)optimality of the composite minimization algorithm in an oracle complexity model. To accurately reflect how our algorithms work (namely, using gradient information on the smooth term and regularized proximal subproblems w.r.t.~the uniformly convex term), we introduce a new problem class and oracle for the complementary composite problem. We observe that existing constructions in the literature of lower bounds for nonsmooth uniformly convex optimization (e.g., \cite{Juditsky:2014,Sridharan}) apply to our composite setting for $\kappa=1$. The main idea of the lower bounds in this section is to combine these constructions with local smoothing, to obtain composite functions that match our assumptions.

\begin{assumptions} \label{assump:composite_oracle}
Consider the problem class ${\cal P}({\cal F}_{\|\cdot\|}(L,\kappa),{\cal U}_{\|\cdot\|}(\lambda, q),R)$, given by composite objective functions 
$$ (P_{f,\psi})~~~ \min_{\vx\in \mathbf{E}} [\bar{f}(\vx)= f(\vx) + \psi(\vx)],$$
%$\bar{f}(\vx)= f(\vx) + \psi(\vx)$ 
with the following assumptions:
\begin{enumerate}[label=(A.\arabic*), leftmargin=1cm]
%, series=l_after]
\item %[(A.1)] 
$f\in {\cal F}_{\|\cdot\|}(L,\kappa)$; \label{assump:A1} %(recall that ${\cal F}_{\|\cdot\|}(\kappa,L)$ is the class of convex functions with $\kappa$-H\"older continuous gradients with constant $L\geq0$ w.r.t.~$\|\cdot\|$).
\item%[(A.2)] 
$\psi\in{\cal U}_{\|\cdot\|}(\lambda,q)$; \label{assump:A2} and,% , where ${\cal U}_{\|\cdot\|}(q,\lambda)$ is the class of $q$-uniformly convex objectives w.r.t.~$\|\cdot\|$ with constant $\lambda>0$, such that their minima are attained at the origin.
\item%[(A.3)] 
the optimal solution of $(P_{f,\psi})$ is attained within ${\cal B}_{\|\cdot\|}(0,R)$. \label{assump:A3}
\end{enumerate}
The problem class is additionally endowed with oracles ${\cal O}_{\cal F}$ and ${\cal O}_{\cal U}$, for function classes ${\cal F}_{\|\cdot\|}(L,\kappa)$ and ${\cal U}_{\|\cdot\|}(\lambda,q)$, respectively; which satisfy
\begin{enumerate}[label=(O.\arabic*), leftmargin=1cm]
    \item ${\cal O}_{\cal F}$ is a local oracle: if $f,g\in {\cal F}_{\|\cdot\|}(L,\kappa)$ are such that there exists $r>0$ such that they coincide in a neighborhood ${\cal B}_{\|\cdot\|}(\vx,r)$, then ${\cal O}_{\cal F}(\vx,f)={\cal O}_{\cal F}(\vx,g)$; and,
    \item ${\cal U}_{\|\cdot\|}(\lambda,q)$ is any oracle (not necessarily local).
\end{enumerate}
\end{assumptions}

In brief, we are interested in the oracle complexity of achieving $\epsilon$-optimality gap for the family of problems $(P_{f,\psi})$,
where $f\in {\cal F}_{\|\cdot\|}(L,\kappa)$ is endowed with a local oracle, $\psi\in {\cal U}_{\|\cdot\|}(\lambda,q)$ is endowed with any oracle, and the optimal solution of problem $(P_{f,\psi})$ lies in ${\cal B}_{\|\cdot\|}(0,R)$.  A simple observation is that in the case $\lambda=0$, our model coincides with the
classical oracle mode, which was discussed in the previous section. 
%first-order oracle model \cite{nemirovski:1983}, for which tight complexity bounds are known \cite{guzman:2015}. 
The goal now is to prove a more general lower complexity bound for the composite model.

Before proving the theorem, we first provide some building blocks in this construction, borrowed from past work \cite{guzman:2015, diakonikolas:2019}. In particular, our lower bound works generally for $q$-uniformly convex and {\em locally smoothable} spaces.
\begin{assumptions} \label{assump:smooth_space}
Given the normed space $(\mathbf{E},\|\cdot\|)$, we consider the following properties:
\begin{enumerate}[label=(B.\arabic*), leftmargin=1cm]
    \item $\psi(\vx)=\frac{1}{q}\|\vx\|^q$ is $q$-uniformly convex with constant $\bar{\lambda}$ w.r.t. $\|\cdot\|$. \label{assump:B1}
    %$$ \psi(\vy)\geq \psi(\vx)+\langle \nabla \psi(\vx),\vy-\vx \rangle + \frac{\bar{\lambda}}{q}\|\vy-\vx\|^q. $$
    \item The space $(\mathbf{E},\|\cdot\|)$ is $(\kappa,\eta,\eta,\bar{\mu})$-locally smoothable. That is, there exists a mapping ${\cal S}:{\cal F}_{(\bE,\|\cdot\|)}(0,1)\to{\cal F}_{(\bE,\|\cdot\|)}(\kappa,\overline{\mu})$ (denoted as the smoothing operator in \cite[Definition 2]{diakonikolas:2019}), such that $\|{\cal S}f-f\|_{\infty} \leq \eta$, %, that ${\cal S}f$ is $(\bar{\mu},\kappa)$-weakly smooth w.r.t. $\|\cdot\|$, 
    and this operator preserves the equality of functions when they coincide in a ball of radius $2\eta$; i.e., if $f|_{{\cal B}_{\|\cdot\|}(0,2\eta)}=g|_{{\cal B}_{\|\cdot\|}(0,2\eta)}$ then ${\cal S}f|_{{\cal B}_{\|\cdot\|}(0,\eta)}={\cal S} g|_{{\cal B}_{\|\cdot\|}(0,\eta)}.$ \label{assump:B2}
    \item There exists $\Delta>0$ and vectors $\vz^1,\ldots,\vz^M\in \mathbf{E}$ with $\|\vz^i\|_{\ast}\leq1$, such that for all $s_1,\ldots,s_M\in\{-1,+1\}^M$
\begin{equation}\label{eqn:norm_LB_gap}
\inf_{\valpha\in \vDelta_M} \Big\|\sum_{i\in[M]}\alpha_i s_i\vz^i \Big\|_{\ast}\geq \Delta, 
\end{equation}
where $\vDelta_M=\{\valpha\in \rr_+^M: \sum_i\alpha_i=1\}$ is the discrete probability simplex in $M$-dimensions. \label{assump:B3}
%\textcolor{red}{This condition should be modified to work for general $\psi$. For now, it's only done for norm powers.}
\end{enumerate}
\end{assumptions}
The three assumptions in Assumptions~\ref{assump:smooth_space} are common in the literature, and can be intuitively understood as follows. Assumption \ref{assump:B1} is the existence of a simple function that we can use as the uniformly convex term in the composite model. Assumption \ref{assump:B2} appeared in \cite{guzman:2015}, and provides a simple way to reduce the complexity of smooth convex optimization to its nonsmooth counterpart. We emphasize that there is a canonical way to construct smoothing operators, which is stated in  Observation~\ref{obs:assump_ell_p} below. Finally, Assumption \ref{assump:B3} comes from the hardness constructions in nonsmooth convex optimization from \cite{nemirovski:1983}, which are given by piecewise linear objectives that are learned one by one by an adversarial argument. The fact that the resulting piecewise linear function has a sufficiently negative optimal value (for any adversarial choice of signs) can be directly obtained  by minimax duality from Eq.~\eqref{eqn:norm_LB_gap}. 

Note that $d$-dimensional $\ell_p$ spaces (denoted by $\ell_p^d$) satisfy the assumptions above when $2\leq p<\infty$.
\begin{observation}[From \cite{guzman:2015}] \label{obs:assump_ell_p}
Let $2\leq p<\infty$ and $\eta>0$, and consider the space $\ell_p^d=(\mathbb{R}^d,\|\cdot\|_p)$. We now verify the Assumptions \ref{assump:smooth_space} for $q=p$, $\bar{\lambda}=1$, $\bar{\mu}=2^{2-\kappa}(\min\{p,\ln d\}/\eta)^{\kappa-1}$ and $\Delta=1/M^{1/p}$. Indeed,
\begin{itemize}
    \item \ref{assump:B1} The $p$-uniform convexity of $\psi$ was discussed after Eq.~\eqref{eq:psi_p}.
    %can be obtained from the duality between uniform convexity and uniform smoothness \cite{Zalinescu:1983} and direct computation (notice the function is separable, so the computation reduces to a one-dimensional inequality).
    \item \ref{assump:B2} The smoothing operator can be obtained by infimal convolution, with kernel function $\phi(\vx)=2\|\vx\|_r^2$ (with $r=\min\{p,3\ln d\}$. We recall that the infimal convolution of two functions $f$ and $\phi$ is given by
    $$ (f\square \phi)(\vx)=\inf_{h\in {\cal B}_p(0,1)}[f(\vx+\vh)+\phi(\vh)]. $$
    The infimal convolution above can be adapted to obtain arbitrary uniform approximation to $f$ and the preservation of equality of functions
    (see \cite[Section 2.2]{guzman:2015} for details). 
    \item \ref{assump:B3} Letting $\vz^i=\ve_i$, $i\in[M]$, be the first $M$ canonical vectors, we have
    $$ \Big\| \sum_{i\in[M]} \alpha_i s_i \vz^i \Big\|_{p_{\ast}}=\|\valpha\|_{p_{\ast}} \geq M^{1/p_{\ast}-1}\|\valpha\|_1=
    %\Big(\sum_{i\in[M]}\alpha_i^{p_{\ast}} \Big)^{1/p_{\ast}}\geq
    M^{-1/p}. $$ 
    %\alpha_i s_i \bz^i \Big\|_{p_{\ast}}=\Big(\sum_{i\in[M]}\alpha_i^{p_{\ast}} \Big)^{1/p_{\ast}}\geq M^{-1/p}. $$
    This bound is achieved when $\alpha_i=1/M$, for all $i$.
\end{itemize}
\end{observation}

Before proving the result for $\ell_p$-spaces, we provide a general lower complexity bound for the composite setting, which we later apply to derive the lower bounds for $\ell_p$ setups.

\begin{lemma} \label{lem:LB_composite}
Let $(\mathbf{E},\|\cdot\|)$ be a normed space that satisfies Assumptions  \ref{assump:smooth_space} and let $${\cal P}({\cal F}_{\|\cdot\|}(L,\kappa),{\cal U}_{\|\cdot\|}(\lambda,q),R)$$ be a class of complementary composite problems that satisfies Assumptions \ref{assump:composite_oracle}. Suppose the following relations between parameters are satisfied:
\begin{enumerate}
    \item[(a)] $2qL\bar{\lambda}/[\lambda\bar\mu]\leq R^{q-1}$.
    \item[(b)] $(M+3)\eta\leq 4R$.
    \item[(c)] $\frac{L}{4\bar{\mu}}(M+7)\eta\leq \frac{1}{2q_{\ast}} \big(\frac{L\Delta}{\bar{\mu}}\big)^{q_{\ast}}\big(\frac{\bar\lambda}{\lambda} \big)^{\frac{1}{q-1}}$.
\end{enumerate}
Then, %\markupadd{for any algorithm interacting with oracles ${\cal O}_{\cal F}$ and ${\cal U}_{\|\cdot\|}(\lambda,q)$ that generates a trajectory $(\vx^t)_{t\in[M]}$,  there exists a function $f\in{\cal F}$ such that}
the worst-case optimality gap for the problem class \markupadd{where algorithms are constrained to make $M$ queries to oracles ${\cal O}_{\cal F}$, ${\cal U}_{\|\cdot\|}(\lambda,q)$,} is bounded below by
$$ %\min_{t\in[M]}f(\vx^t)-\min_{x\in\mathbf{E}} \overline{f}(\vx) \geq
\frac{1}{2q_{\ast}} \Big(\frac{L\Delta}{\bar{\mu}}\Big)^{q_{\ast}}\Big(\frac{\bar\lambda}{\lambda} \Big)^{\frac{1}{q-1}}. $$
\end{lemma}

\begin{proof}
 Given $M\in\mathbb{N}$, scalars $\delta_1,\ldots,\delta_{M}>0$, and $s_1,\ldots,s_M\in \{-1,+1\}$, we consider the functions
$$ f_{s}(\vx)= \frac{L}{\bar \mu} {\cal S}\Big( \max_{i\in[M]}[\langle s_i\vz^i, \cdot \rangle-\delta_i] \Big)(\vx), $$
%and 
%$$ \psi(\vx) = \frac{\lambda}{q}\|\vx\|^q. $$
%The subfamily of problems we consider for the lower bound is given by $\{\bar{f}=f_s+\psi\,: s\in \{-1,+1\}^M, \delta\in[0,1/2)^M\}$. 
and $\bar{f}_s(\vx)=f_s(\vx)+(\lambda/\bar{\lambda})\psi(\vx)$, where $\psi$ is given by Assumption \ref{assump:B1}. %\ref{assump:smooth_space}. 

We now show the composite objective $\bar{f}_s$ satisfies Assumption \ref{assump:composite_oracle}. Properties \ref{assump:A1} and \ref{assump:A2} 
%(A.1) and (A.2) 
are clearly satisfied. Regarding \ref{assump:A3}, %(A.3), 
we prove next that the optimum of these functions lies in ${\cal B}_{\|\cdot\|}(0,R)$. For this, notice that by Assumption \ref{assump:B2}: %\ref{assump:smooth_space}, Property 2:
\begin{eqnarray*}
\bar{f}_s(\vx) &\geq& 
\frac{L}{\bar{\mu}}\max_{i\in[M]}[\langle s_i\vz^i, \vx \rangle-\delta_i]-\frac{L\eta}{\bar{\mu}}+\frac{\lambda}{q \bar{\lambda}}\|\vx\|^q \\
&\geq&\|\vx\| \Big[ \frac{\lambda}{\bar{\lambda}q}\|\vx\|^{q-1}-\frac{L}{\bar{\mu}} \Big]-\frac{L}{\bar{\mu}}(\eta+\max_i\delta_i).
%-\frac{L}{\bar{\mu}}[\|\vx\|-\max_{i}\delta_i-\eta]+\frac{\lambda}{q}\|\vx\|^q.
\end{eqnarray*}
We will later show that $\eta+\max_i\delta_i\leq (M+3)\eta/4\leq R$ (the last inequality by (b)), hence for $\|\vx\|\geq R$
$$ \bar{f}_s(\vx) \geq \Big(  \frac{\lambda}{\bar{\lambda}q}\|\vx\|^{q-1} -\frac{2L}{\bar\mu} \Big) \|\vx\|\geq 0,$$
where the last inequality follows from (a). %in the statement, $2qL\bar{\lambda}/[\lambda\bar\mu]\leq R^{q-1}$; 
%We will later show that (b) implies $\eta+\max_i\delta_i\leq 1$, hence $\bar{f}_s(\vx)\geq 0$ for any $\|\vx\|\geq R$. 
To conclude the verification of
Assumption \ref{assump:A3},  %(A.3), 
we now prove that $\min_{\vx\in \mathbb{R}}\bar{f}_s(\vx)< 0$. By Assumption \ref{assump:B2}:  %\ref{assump:smooth_space}, Property 2:
\begin{eqnarray*}
\inf_{\vx\in\mathbf{E}} \bar{f}(\vx) &\leq&   \inf_{\vx\in\mathbf{E}} \Big( \frac{L}{\bar \mu} \max_{i\in[M]}[\langle s_i\vz^i,x\rangle-\delta_i]+ \frac{L}{\bar \mu}\eta+\frac{\lambda}{q\bar{\lambda}}\|\vx\|^q\Big)\\
&=& \max_{\valpha\in\vDelta_M} \inf_ {x\in \mathbf{E}} \Big(\Big\langle \frac{L}{\overline{\mu}}\sum_{i\in[M]}\alpha_i s_i\vz^i, x\Big\rangle+ \frac{\lambda}{q\bar{\lambda}}\|\vx\|^q -\frac{L}{\overline{\mu}}\sum_{i\in[M]}\alpha_i\delta_i+\frac{L}{\overline{\mu}}\eta\Big)\\ 
&=& \max_{\valpha\in\vDelta_M}-\frac{1}{q_{\ast}}\Big(\frac{L}{\bar{\mu}}\Big)^{q_{\ast}}\Big(\frac{\bar{\lambda}}{\lambda}\Big)^{\frac{1}{q-1}}\Big\|\sum_{i\in[M]}\alpha_i s_i\vz^i \Big\|_{\ast}^{q_{\ast}}- \frac{L}{\overline{\mu}}\sum_{i\in[M]}\alpha_i\delta_i +\frac{L}{\overline{\mu}}\eta\\
%&=& -\min_{\vlambda\in\Delta_M}\Big\{\psi^{\ast}\Big(-\sum_{i\in[M]}\lambda_i s_i\vz^i \Big) +\frac{L}{\overline{\mu}}\sum_{i\in[M]}\lambda_i\delta_i\Big\} +\frac{L}{\overline{\mu}}\eta\\
%&\leq&-\min_{\vlambda\in\Delta_M}\psi^{\ast}\Big(-\frac{L}{\overline{\mu}}\sum_{i\in[M]}\lambda_i s_i\vz^i \Big)+ \frac{L}{\overline{\mu}}\eta.
&=&-\frac{1}{q_{\ast}}\Big(\frac{L}{\bar{\mu}}\Big)^{q_{\ast}}\Big(\frac{\bar{\lambda}}{\lambda}\Big)^{\frac{1}{q-1}}\Delta^{q_{\ast}} +\frac{L}{\overline{\mu}}\eta.
\end{eqnarray*}
Notice that the second step above follows from the Sion Minimax Theorem \cite{Sion:1958}. We conclude that the optimal value of $(P_{f,\psi})$ is negative by (c).
%Therefore, the value of $(P_{f,\psi})$ is negative iff $\lambda< \frac{\bar{\lambda} \Delta^q}{(q_{\ast}\eta)^{1/(q-1)}}\big(\frac{L}{\bar{\mu}}\big)^q $. 
%Since $\psi(\vx)=\frac{\lambda}{q\bar{\lambda}}\|\vx\|^q$, we can use that $\psi^{\ast}(\vy)=(q_{\ast}\lambda^{\frac{1}{q-1}})^{-1}\|\vy\|_{\ast}^{q_{\ast}}$, so
%$$ \inf_{\vx\in\mathbf{E}} \bar{f}(\vx) 
%\leq -\frac{1}{q_{\ast}}\Big(\frac{L^q}{\lambda\bar{\mu}^q}\Big)^{\frac{1}{q-1}}\Delta^{q_{\ast}}+\frac{L}{\bar{\mu}}\eta. $$

Following the arguments provided in \cite[Proposition 2]{guzman:2015}, one can prove that for any algorithm interacting with oracle ${\cal O}_{\cal F}$, after $M$ steps there exists a choice of $s_1,\ldots,s_M\in\{-1,+1\}^M$ such that
%The parameters of these functions are chosen adversarially, so that after $M$ steps (details can be found in \cite[Proposition 2]{guzman:2015}):
$$ \min_{t\in[M]} f_{s}(\vx^t) \geq \frac{L}{\bar{\mu}}[ -\eta-\max_{i\in[M]}\delta_i]; $$
further, for this adversarial argument it suffices that $\min_{i\in [M]}\delta_i=0$, and $\max_{i\in[M]}\delta_i \geq (M-1)\eta/4$. 

We conclude that the optimality gap after $M$ steps is bounded below by
\begin{align*}
\min_{t\in[M]} \bar{f}_s(\vx^t)-\min_{\vx\in\mathbf{E}}\bar{f}_s(\vx) &\geq -\frac{L}{4\bar{\mu}}(M+7)\eta + \frac{1}{q_{\ast}}\Big(\frac{L}{\bar{\mu}}\Big)^{q_{\ast}}\Big(\frac{\bar{\lambda}}{\lambda}\Big)^{\frac{1}{q-1}}\Delta^{q_{\ast}}\\
&\geq \frac{1}{2q_{\ast}} \Big(\frac{L\Delta}{\bar{\mu}}\Big)^{q_{\ast}}\Big(\frac{\bar\lambda}{\lambda} \Big)^{\frac{1}{q-1}},
\end{align*} 
%-\frac{L}{\bar{\mu}}(M+1)\eta+\frac{L}{\bar{\mu}} (q_{\ast}\lambda^{\frac{1}{q-1}})^{-1}\Delta^{q_{\ast}}. $$
where we used the third bound from the statement.
\qed\end{proof}

We now proceed to the lower bounds for $\ell_p$-setups, with $2\leq p\leq \infty$. 
\begin{theorem}\label{thm:p-norm-lbs}
Consider the space $\ell_p^d=(\mathbb{R}^d,\|\cdot\|_p)$, where $2\leq p<\infty$. Then, the oracle complexity of problem class ${\cal P}:={\cal P}({\cal F}_{\|\cdot\|}(L,\kappa),{\cal U}_{\|\cdot\|}(\lambda,p),R)$, comprised of composite problems in the form $(P_{f,\psi})$ under Assumptions \ref{assump:composite_oracle}, is bounded below by
$$ %\mathrm{Compl}({\cal P},({\cal O}_{\cal F},{\cal O}_{\psi}),\epsilon) \geq
\left\{ 
\begin{array}{ll}
\Big\lfloor\sqrt{\frac{L}{2\lambda}}-7\Big\rfloor, & \mbox{ if }p=\kappa=2, \,
\epsilon<2\sqrt{2\lambda L}R^2 \min\{\frac{2\lambda}{L},1 \},\\
%\epsilon<\lambda R^2\min\Big\{\frac{\lambda}{16L},1 \Big\}\\ 
 \frac{C(p,\kappa)}{\min\{p,\ln d\}^{2(\kappa-1)}}\left(\frac{L^p}{\lambda^{\kappa}\epsilon^{p-\kappa}}\right)^{\frac{1}{\kappa p+\kappa-p}}, & \mbox{ if } 1\leq \kappa< p,\, p\in[2,\infty], \, \mbox{and } \lambda\geq \tilde\lambda .
\end{array}
\right.
$$
where $C(p,\kappa):=\left(\Big(\frac{p-1}{p}\Big)^{\kappa(p-1)} 2^{\frac{(p-\kappa)(1-2p)+(\kappa-1)p(2p-3)}{(p-1)}}\right)^{\frac{1}{\kappa p+\kappa-p}}$ is bounded below by an absolute constant, and
%$C(p,\kappa)=\big[2^{\frac{1+\kappa p-3p}{p-1}+\big(\frac{p}{p-1}\big)\big(\frac{p+\kappa-3}{p-\kappa}\big)(\kappa-1)}\cdot\Big(\frac{p-1}{p}\Big)^{\kappa(p-1)}\big]^{\frac{1}{\kappa p+\kappa-p}}$, and
\begin{equation} \label{eqn:bd_lambda_LB}
\begin{aligned}
    \tilde\lambda :=  \bar{C}\max\Bigg\{ &
    \min\{p,\ln d\}^{3} \Big(\dfrac{\epsilon^{\kappa}}{LR}\Big)^{\frac{1}{\kappa-1}}
    %\dfrac{L^{\frac{1}{\kappa-1}}\epsilon^{\frac{\kappa}{\kappa-1}}}{\min\{p,\ln d\}^{\frac{(p-1)(\kappa p+\kappa-p)-p\kappa^2}{(p-1)(p-\kappa)}}}
   ,\\
   &\min\{p,\ln d\}^5
\left(\dfrac{\epsilon^{p}}{ L^{(p+1)}R^{\frac{(p-1)(\kappa p+\kappa-p)}{(\kappa-1)}}}\right)^{\frac{\kappa-1}{\kappa p+1-p}}
%\left(\dfrac{\epsilon^{p}}{\min\{p,\ln d\}^{\frac{(p-1)^2(\kappa p+\kappa-p)-p^2\kappa(\kappa-1)}{(p-\kappa)(p-1)}} L^{{p+1}}R^{\frac{(p-1)(\kappa p+\kappa-p)}{\kappa -1}}}\right)^{\frac{\kappa - 1}{p\kappa + 1 - p}} 
\Bigg\},
\end{aligned}
\end{equation}
where $\bar{C}>0$ is a universal constant.
%where $C(p,\kappa)=2^{\frac{1+\kappa p-3p}{p-1}+\frac{p}{p-1}\frac{p+\kappa-3}{p-\kappa}(\kappa-1)}\Big(\frac{p-1}{p}\Big)^{\kappa(p-1)}$
\end{theorem}

In particular, our lower bounds show that the algorithm presented in the previous section---particularly the rates stated in Theorem \ref{thm:comp-opt-gap-conv}---are nearly optimal. In the case $p=\kappa=2$, the gap between upper and lower bounds is only given by a factor which grows at most logarithmically in $L\phi(\bar\vx^{\ast})/\epsilon$, and in the case $\kappa<p$, the gap is $O\big(\log(L\phi(\bar\vx^{\ast})/\epsilon)/\min\{p,\ln d\}^{\Theta(1)}\big)$. In both cases, the gaps are quite moderate, so the proposed algorithm is proved to be nearly optimal. Finally, we would also like to emphasize that the constant $C(p,\kappa)=\Theta(1)$, as a function of $1<\kappa\leq 2$ and $2\leq p\leq \infty$. Therefore, the lower bounds also apply to the case $p=\infty$. 
%In fact, $C(\infty,\kappa)=1$ if $\kappa>1$.
%
\begin{proof}[of Theorem~\ref{thm:p-norm-lbs}]
By Observation \ref{obs:assump_ell_p}, in the case of $\ell_p^d$, with $2\leq p<\infty$, Assumptions~\ref{assump:smooth_space} are satisfied if $q=p$, $\Delta=1/M^{1/p}$, $\overline{\lambda}=1$, and $\bar{\mu}=2^{2-\kappa}(\min\{p,\ln d\}/\eta)^{\kappa-1}$ (for given $\eta>0$). This way, hypotheses (a), (b), (c) in Lemma~\ref{lem:LB_composite} become
\begin{enumerate}
    \item[(a)] $\eta\leq\frac{\min\{p,\ln d\}}{2}\big(\frac{\lambda R^{p-1}}{pL}\big)^{\frac{1}{\kappa-1}}$.  %$2^{\kappa-1}p\eta^{\kappa-1}/[R^{p-1}\min\{p,\ln d\}^{\kappa-1}] $ $\lambda\geq2^{\kappa-1}p\eta^{\kappa-1}/[R^{p-1}\min\{p,\ln d\}^{\kappa-1}]$.
    \item[(b)] $(M+3)\eta\leq 4R$.
    \item[(c)] 
    $\eta^{p-\kappa}\leq  \frac{2^{p+\kappa-3}L}{p_{\ast}^{(p-1)}\min\{p,\ln d\}^{(\kappa-1)} \lambda M (M+7)^{(p-1)}}$. 
    %$\frac{M(M+7)^{p-1}\eta^{p-\kappa}}{2^{p+\kappa-3}} \leq \Big(\frac{p-1}{p}\Big)^{p-1}\frac{L}{\lambda \min\{p,\ln d\}^{\kappa-1}}.$
\end{enumerate}
%thus (c) in Lemma \ref{lem:LB_composite} is met iff
%$$ \frac{M(M+7)^{p-1}\eta^{p-\kappa}}{2^{p+\kappa-3}} \leq \Big(\frac{p-1}{p}\Big)^{p-1}\frac{L}{\lambda \min\{p,\ln d\}^{\kappa-1}}. $$
%for the lower bound to be nontrivial, it is necessary that
%$$ \frac{1}{p_{\ast}}\Big( \frac{L^p \eta^{p(\kappa-1)}}{2^{p(2-\kappa)}\lambda M \min\{p,\ln d\}^{p(\kappa-1)} } \Big)^{\frac{1}{p-1}}\geq \frac{L(M+7)\eta^{\kappa}}{2^{4-\kappa}\min\{p,\ln d\}^{\kappa-1}}. $$
\noindent\textbf{Case 1: $p=\kappa=2$.}  %First, consider the case 
In order to satisfy (c), it suffices to choose
    $ M=\Big\lfloor \sqrt{ \frac{L}{2\lambda}}-7 \Big\rfloor.$ 
   % In order to satisfy (a), (b) of the lemma, we can choose
   % $$ \eta = \frac{\sqrt{\lambda}}{2}\min\Big\{ \frac{1}{\sqrt L}, \frac{R}{2}\Big\}. $$
    Given such a choice, to satisfy (a), (b) of the lemma, we can choose 
    $$ \eta = \min\Big\{ \frac{\lambda R}{2L}, \frac{4R}{M+3}\Big\} 
    \geq R\sqrt{\frac{2\lambda}{L}}\min\Big\{\frac{1}{4}\sqrt{\frac{2\lambda}{L}},4\Big\}. $$
    %And (c) is equivalent to
%\begin{equation} \label{eqn:UBforLB}
%M(M+7) \leq L/\lambda.
%\end{equation}
Now, under the conditions imposed above, the lemma provides an optimality gap lower bound of
\begin{eqnarray*}
\frac{1}{4\lambda}\Big(\frac{L \eta}{2\sqrt{M} }\Big)^2 
\geq 2\sqrt{2\lambda L}R^2 \min\Big\{\frac{2\lambda}{L},1 \Big\}.
\end{eqnarray*}
In conclusion, if $\epsilon<2\sqrt{2\lambda L}R^2 \min\{2\lambda/L,1 \}$, 
%$\frac{\sqrt{2\lambda L}R^2}{8}\min\Big\{\frac{\lambda}{8L},1 \Big\},$ 
then
$$
\mathrm{Compl}({\cal P},({\cal O}_{\cal F},{\cal O}_{\psi}),\epsilon) \geq \Big\lfloor\sqrt{\frac{L}{2\lambda}}\Big\rfloor -1. 
$$

\noindent\textbf{Case 2: $p>\kappa$ (where $1<\kappa\leq 2,$ $2\leq p<\infty$).} Here, to ensure (a) and (b), it suffices that
\begin{equation} \label{eqn:eta_p>kappa}
     \eta\leq \min\Big\{ \frac{4R}{M+3},\frac{\min\{p,\ln d\}}{2}\Big(\frac{\lambda R^{p-1}}{p} \Big)^{\frac{1}{\kappa-1}} \Big\}.
\end{equation}
We later certify these conditions hold. On the other hand, for (c) it suffices to let
%we the necessary condition becomes
%$$ \eta \leq \Big[\Big(\frac{p-1}{p}\Big)^{p-1}\frac{2^{2p+\kappa-4} L}{\lambda\min\{p,\ln d\}^{\kappa-1}M(M+7)^{p-1}} \Big]^{\frac{1}{p-\kappa}}. $$
%Let us choose, in fact, 
$$ \eta = \Big[\Big(\frac{p-1}{p}\Big)^{p-1}\frac{2^{p+\kappa-3} L}{\lambda\min\{p,\ln d\}^{\kappa-1}M(M+7)^{p-1}} \Big]^{\frac{1}{p-\kappa}}.$$
Then by Lemma \ref{lem:LB_composite} the optimality gap is  bounded below as
\begin{align*}
&\frac{1}{2p_{\ast}}\Big( \frac{L^p \eta^{p(\kappa-1)}}{2^{p(2-\kappa)}\lambda M \min\{p,\ln d\}^{p(\kappa-1)} } \Big)^{\frac{1}{p-1}} \\
=& \left[C(p,\kappa)^{\kappa p+\kappa-p}
%\Big(\frac{p-1}{p}\Big)^{\kappa(p-1)} 2^{\frac{(p-\kappa)(1-2p)+(\kappa-1)p(2p-3)}{(p-1)}}
\cdot
\frac{L^{p}}{\min\{p,\ln d\}^{\frac{p(\kappa-1)(\kappa p-2\kappa+1)}{p-1}}\lambda^{\kappa}(M+7)^{\kappa p+\kappa-p}}\right]^{\frac{1}{p-\kappa}},
%C(p,\kappa)^{\kappa p+\kappa-p} \left(\frac{L^{p}}{\min\{p,\ln d\}^{\frac{p\kappa(\kappa-1)}{p-1}}\lambda^{\kappa}M^{\kappa}(M+7)^{p(\kappa-1)}}\right)^{\frac{1}{p-\kappa}},%\\
\end{align*}
where $C(p,\kappa):=\left(\Big(\frac{p-1}{p}\Big)^{\kappa(p-1)} 2^{\frac{(p-\kappa)(1-2p)+(\kappa-1)p(2p-3)}{(p-1)}}\right)^{\frac{1}{\kappa p+\kappa-p}}$.
%where $C(p,\kappa)=\big[2^{\frac{1+\kappa p-3p}{p-1}+\big(\frac{p}{p-1}\big)\big(\frac{p+\kappa-3}{p-\kappa}\big)(\kappa-1)}\cdot\Big(\frac{p-1}{p}\Big)^{\kappa(p-1)}\big]^{\frac{1}{\kappa p+\kappa-p}}$, as defined in the statement of the theorem. 
In particular, if $\epsilon$ is smaller than the gap above, resolving for $M$ gives
\begin{equation} \label{eqn:LB_composite_kappa_less_p}
 \mathrm{Compl}({\cal P},({\cal O}_{\cal F},{\cal O}_{\psi}),\epsilon)  \geq M=
 \frac{C(p,\kappa)}{\min\{p,\ln d\}^{2(\kappa-1)}}\left(\frac{L^p}{\lambda^{\kappa}\epsilon^{p-\kappa}}\right)^{\frac{1}{\kappa p+\kappa-p}},
 %C(p,\kappa)\left(\frac{L^p}{\min\{p,\ln d\}^{\frac{p\kappa(\kappa-1)}{p-1}}\lambda^{\kappa}\epsilon^{p-\kappa}}\right)^{\frac{1}{\kappa p+\kappa-p}}.
%\Big(\frac{p-1}{p}\Big)^{\kappa(p-1)} \frac{L^{p}2^{\frac{p}{p-1}(\kappa^2-\kappa p^2+3\kappa p-p-4\kappa+1)}}{\min\{p,
%\ln d\}^{\frac{p\kappa(\kappa-1)}{p-1}}\lambda^{\kappa}\epsilon^{(p-\kappa)}}\right)^{\frac{1}{\kappa p+\kappa-p}}. 
\end{equation}
where we further simplified the bound, noting  that $\frac{p(\kappa-1)(\kappa p-2\kappa+1)}{(p-1)(p\kappa+\kappa-p)}\leq 2(\kappa-1)$.

Now, given the chosen value of $M$, we will verify that Eq.~\eqref{eqn:eta_p>kappa} holds. For this, we note that Eq.~\eqref{eqn:eta_p>kappa} is implied by the following pair of inequalities
\begin{eqnarray}
\lambda &\geq& 
C^{\prime}(p,\kappa)\min\{p,\ln d\}^{(\kappa-1)(2\kappa-1)} \Big(\dfrac{\epsilon^{\kappa}}{LR}\Big)^{\frac{1}{\kappa-1}}
%\dfrac{L^{\frac{1}{\kappa-1}}\epsilon^{\frac{\kappa}{\kappa-1}}}{\min\{p,\ln d\}^{\frac{(p-1)(\kappa p+\kappa-p)-p\kappa^2}{(p-1)(p-\kappa)}}}
\label{eqn:cond_1_LB}\\
\lambda &\geq&  C^{\prime\prime}(p,\kappa) \min\{p,\ln d\}^{5}
\left(\dfrac{\epsilon^{p}}{ L^{(p+1)}R^{\frac{(p-1)(\kappa p+\kappa-p)}{(\kappa-1)}}}\right)^{\frac{\kappa-1}{\kappa p+1-p}}
%\Big(\frac{p}{p-1}\Big)^{\frac{(\kappa-1)(p-1)}{\kappa p-p+1}} \frac{ \min\{p,\ln d\}^{\frac{p^2\kappa(\kappa-1)^2}{(p-1)(p-\kappa)(\kappa p -p+1)}} p^{\frac{\kappa p-p+\kappa}{\kappa p-p+1}} 2^{(\frac{\kappa-1}{p-\kappa})(\frac{(2p-3)(\kappa p+\kappa -p)(p-1)-p^2(\kappa^2-\kappa p^2+3\kappa p-p-4\kappa+1)}{(p-1)(\kappa p+1-p)})}\epsilon^{\frac{p(\kappa-1)}{(\kappa-1)p+1}}}{L^{(\frac{\kappa-1}{p-\kappa})\frac{p(p-1)(\kappa-1)-\kappa}{(p-\kappa)(\kappa p+1-p)}} R^{\frac{(p-1)(\kappa p-p+\kappa)}{\kappa p-p+1}}}. 
\label{eqn:cond_2_LB}
\end{eqnarray}
%with 
%\begin{eqnarray*}
%C^{\prime}(p,\kappa) &=& 2^{\frac{\kappa}{(p-1)(\kappa p+\kappa-p)}\big[-1-\kappa-3p-p(\kappa-1)(\frac{p+\kappa-3}{p-\kappa})\big]}\cdot\Big(\frac{p}{p-1}\Big)^{\frac{(p-1)(1-\kappa^2)}{p\kappa+\kappa-p}}\\
%C^{\prime\prime}(p,\kappa) &=& \Big(2^{2p-3+\frac{p}{(\kappa p+\kappa-p)(p-1)}[(3-\kappa)p-1-p(\kappa-1)(\frac{p+\kappa-3}{p-\kappa})]}p^{\frac{p-\kappa}{\kappa-1}}\Big(\frac{p-1}{p}\Big)^{(p-1)(\kappa+1)}\Big)^{\frac{(\kappa-1)(\kappa p+\kappa-p)}{(p-\kappa)(\kappa p+1-p)}}.
%\end{eqnarray*}
%Notice that 
with $C^{\prime}(p,\kappa), C^{\prime\prime}(p,\kappa)\geq \bar{C}>0$, are bounded below by a universal positive constant. 
%$C^{\prime\prime}(p,\kappa)$ are both bounded away from $0$, uniformly on $1<\kappa\leq 2$, $2\leq p<\infty$, thus
Therefore, there exists a universal constant $\bar{C}>0$ such that if $\lambda$ satisfies Eqs.~\eqref{eqn:cond_1_LB} and \eqref{eqn:cond_2_LB} where $C^{\prime}(p,\kappa), C^{\prime\prime}(p,\kappa)$ are replaced by $\bar{C}$, then the lower complexity bound from Eq.~\eqref{eqn:LB_composite_kappa_less_p} holds.
\qed\end{proof}

\begin{remark}
%Before presenting the proof, we should also emphasize that, as shown in the theorem, the 
Observe that the lower bounds from Theorem~\ref{thm:p-norm-lbs} apply only when $\lambda$ is sufficiently large, which is consistent with the behavior of our algorithm, which for small values of $\lambda$ obtains iteration complexity matching the classical smooth setting %\todo[]{Let me know if you agree with this statement.}
(as if we ignore the uniform convexity of the objective).
\end{remark}

%%%%%%%%%%%%%%%%%%%%%%%%%
%%%%%%%%%%%%%%%%%%%%%%%%%
\section{Applications}\label{sec:apps}%$\ell_p$ and $\mathrm{Sch}_p$ Regression}

We now provide some illustrative applications of the results from Sections~\ref{sec:comp-min} and \ref{sec:small-grad-norm} to different regression problems. In typical applications, the data matrix $\mA$ is assumed to have fewer rows than columns, so that the system $\mA\vx = \vb$, where $\vb$ is the vector of labels, is underdetermined, and one seeks a sparse solution $\vx^{\star}$ that provides a good linear fit between the data and the labels. %First, we show how the results from Section~\ref{sec:comp-min} can be used to obtain a surprisingly fast algorithm for a variant of the classical Dantzig selector problem. We then show how to apply the results from Section~\ref{sec:small-grad-norm} to obtain fast near-linear-time algorithms for $\ell_p$ and $\mathrm{Sch}_p$ regression.

%%%%%%%%%%%%%%%%
\subsection{Elastic Net} 

One of the simplest applications of our framework is to the elastic net regularization, introduced by~\cite{zou2005regularization}. Elastic net regularized problems are of the form:
$$
    \min_{\vx \in \rr^d} f(\vx) + \frac{\lambda_2}{2}\|\vx\|_2^2 + \lambda_1 \|\vx\|_1,
$$
%\todo[inline]{CG: Should be $\|\vx\|_1$ instead of $\|\vx\|_2$.}
i.e., the elastic net regularization combines the lasso and ridge regularizers. Function $f$ is assumed to be $(L, 2)$-weakly smooth (i.e., $L$-smooth) w.r.t.~the Euclidean norm $\|\cdot\|_2$. It is typically chosen as either the linear least squares loss or the logistic loss. 

We can apply results from Section~\ref{sec:comp-min} to this problem for $q = \kappa = 2,$ choosing $\psi(\vx) = \frac{\lambda_2}{2}\|\vx\|_2^2+\lambda_1\|\vx\|_1$ and $\phi(\vx) = \frac{1}{2}\|\vx - \vx_0\|_2^2.$ 
Observe that our algorithm only needs to solve subproblems of the form
$$
    \min_{\vx \in \rr^d}\Big\{\innp{\vz, \vx} + \frac{\lambda''}{2}\|\vx\|_2^2 + \lambda'\|\vx\|_1\Big\},
$$
for fixed vectors $\vz \in \rr^d$ and fixed parameters $\lambda', \lambda''$, which is computationally inexpensive, as the problem under the min is separable.

Applying Theorem~\ref{thm:comp-opt-gap-conv}, the elastic net regularized problems can be solved to any accuracy $\epsilon > 0$ using 
$$
    k = O\bigg(\min\bigg\{\sqrt{\frac{L}{\lambda_2}}\log\bigg(\frac{L\|\vx^{\star} - \vx_0\|_2}{\epsilon}\Big), \; \sqrt{\frac{L\|\vx^{\star} - \vx_0\|_2^2}{\epsilon}}\bigg\}\bigg)
$$
iterations, where $\vx^{\star} \in \rr^d$ is the problem minimizer. We note in passing that an upper bound of the same order can be obtained by the composite accelerated method in \cite{nesterov2013gradient}.
%%%%%%%%%%%%%%%%
\subsection{\markupdelete{Bridge Regression} \markupadd{Risk Minimization with Strongly Convex $\ell_p$-norm Regularization}}

\markupadd{In this section, we argue that the results obtained in this paper are useful for solving certain regularized empirical risk problems and that the particular choice of regularizers proposed here leads to non-asymptotic consistency and generalization bounds.  We emphasize most of the theory used in what follows (particularly regarding regularization, stability and generalization) is classical (e.g., \cite{Bousquet:2002}), and we only provide the necessary tools for completeness.}

\markupadd{Concretely, our main object of interest are population risk minimization problems of the form
\begin{equation}\label{eq:pop-risk-minimization}
    \min_{\vx \in \rr^d} \cL(\vx), \quad \cL(\vx) = \ee_{\vz \sim \mathcal{D}}[\ell(\vx; \vz)],
\end{equation}
where $\mathcal{D}$ is an unknown distribution, individual loss functions $\ell$ are convex, and we are given a set of $n$ i.i.d.~samples $\mathcal{S} = \{\vz_1, \vz_2, \dots, \vz_n\}$ from $\mathcal{D}$. For this problem to be (computationally and statistically) tractable, further assumptions are required, which we state in Assumption~\ref{assump:pop-risk-min}.}

\markupadd{To address \eqref{eq:pop-risk-minimization}, we apply AGD+ to the following regularized empirical version of the problem
\begin{equation}\label{eq:reg-emp-risk-min}
    \min_{\vx \in \rr^d} \cL_{\mathcal{S}}(\vx) + \frac{\lambda}{2}\|\vx\|_p^2,
\end{equation}
where $\cL_{\mathcal{S}}(\vx) = \frac{1}{n}\sum_{i=1}^n \ell(\vx; \vz_i)$ is the empirical risk, $\lambda$ is a regularization parameter (specified in the analysis below), and $p \in (1, 2].$ We let $\vxh$ denote the output of AGD+, invoked with accuracy parameter $\epsilon_n > 0,$ specified later in this subsection. Further, we let $\vxh^*$ denote the unique}\footnote{\markupadd{The solution to \eqref{eq:reg-emp-risk-min} is unique as the problem is strongly convex, due to the regularizer.}} \markupadd{solution to \eqref{eq:reg-emp-risk-min}. We further use $\cx^*_\mathcal{S}$ to denote the set of minimizers of the empirical loss $\cL_{\mathcal{S}}$ and denote the minimum value of the empirical loss $\cL_{\mathcal{S}}$ by $\cL_{\mathcal{S}}^*$. }

\markupadd{First, we provide a justification for the regularization used in \eqref{eq:reg-emp-risk-min}, from an optimization perspective, in the following proposition. In particular, we argue that the utilized regularization ensures that the solution output by AGD+ has $\ell_p$-norm almost as small as the smallest norm among empirical minimizers, $\min_{\vx \in \cx^*_{\mathcal{S}}}\|\vx\|_p$, while requiring only a modest amount of computation. When $p$ is close to 1, we can interpret this property as enforcing sparsity of the output solution, similar to LASSO.}

\markupadd{
\begin{proposition}\label{prop:reg-emp-risk-opt}
Given \eqref{eq:reg-emp-risk-min}, assume that $\cx^*_{\mathcal{S}} = \argmin_{\vx \in \rr^d}\cL_{\mathcal{S}}(\vx)$ is non-empty. Let $\vxh^*$ denote the solution to \eqref{eq:reg-emp-risk-min}. Then:
\begin{equation}\notag
\begin{aligned}
    &\|\vxh^*\|_p \leq \min_{\vx^*\in \cx^*_{\mathcal{S}}} \|\vx^*\|_p, \quad \text{and} \\
    &\cL_{\mathcal{S}}(\vxh^*) - \cL_{\mathcal{S}}^* \leq \frac{\lambda}{2}\min_{\vx^*\in \cx^*_{\mathcal{S}}} \|\vx^*\|_p^2.
\end{aligned}
\end{equation}
As a consequence, if $\vxh$ is the solution output by AGD+ for optimality gap $\epsilon_n$, then 
\begin{equation}\notag
\begin{aligned}
    &\|\vxh\|_p \leq \min_{\vx^*\in \cx^*_{\mathcal{S}}} \|\vx^*\|_p + \sqrt{\frac{2\epsilon_n}{\lambda(p-1)}}, \quad \text{and} \\
    &\cL_{\mathcal{S}}(\vxh) - \cL_{\mathcal{S}}^* \leq \frac{\lambda}{2}\min_{\vx^*\in \cx^*_{\mathcal{S}}} \|\vx^*\|_p^2 + \epsilon_n.
\end{aligned}
\end{equation}
\end{proposition}}
\begin{proof}
\markupadd{The first inequality follows by \eqref{eq:bnd-init-dist}, already proved within the proof of Theorem~\ref{thm:grad-norm}, as $\|\vxh\|_p \leq  \|\vx^*\|_p$ holds for any $\vx^* \in \cx^*_{\mathcal{S}}$. For the second inequality, using the assumption that $\vxh^*$ is the solution to \eqref{eq:reg-emp-risk-min}, we have, for all $\vx^* \in \cx^*_{\mathcal{S}},$
\begin{equation}\label{eq:reg-opt-gap}
    \cls(\vxh^*) + \frac{\lambda}{2}\|\vxh^*\|_p^2 - \Big(\cls(\vx^*) + \frac{\lambda}{2}\|\vx^*\|_p^2\Big) \leq 0.
\end{equation}
Hence, the second inequality follows by rearranging \eqref{eq:reg-opt-gap}, using that $\frac{\lambda}{2}\|\vxh^*\|_p^2 \geq 0,$ and taking the minimum of both sides over $\vx^* \in \cx^*_{\mathcal{S}}.$}

\markupadd{For the third inequality, as $\cls(\vx) + \frac{\lambda}{2}\|\vx\|_p^2$ is $\lambda(p-1)$-strongly convex w.r.t.~$\|\cdot\|_p$ (as $\cls(\vx)$ is convex and $\frac{1}{2}\|\vx\|_p^2$ is $(p-1)$-strongly convex w.r.t.~$\|\cdot\|_p$), $\vxh^*$ minimizes $\cls(\vx) + \frac{\lambda}{2}\|\vx\|_p^2$, and $\vxh$ is an $\epsilon_n$-approximate solution to \eqref{eq:reg-emp-risk-min}, we have
\begin{equation}\notag
    \frac{\lambda(p-1)}{2}\|\vxh - \vxh^*\|_p^2 \leq f(\vxh)  - f(\vxh^*) \leq \epsilon_n.
\end{equation}
Hence, we have $\|\vxh - \vxh^*\|_p \leq \sqrt{\frac{2 \epsilon_n}{\lambda(p-1)}}.$ The claimed inequality now follows using triangle inequality and the first part of the proposition, since
\begin{equation}\notag
    \|\vxh\|_p \leq \|\vxh^*\|_p + \|\vxh - \vxh^*\|_p \leq \min_{\vx^*\in \cx^*_{\mathcal{S}}} \|\vx^*\|_p + \sqrt{\frac{2\epsilon_n}{\lambda(p-1)}}.
\end{equation}}

\markupadd{For the final part, let $\vx^* \in \argmin_{\vx \in \cx^*_{\cS}} \|\vx\|_p.$ Then
\begin{align*}
    \cls(\vxh) - \cls^* =\; & \cls(\vxh) + \frac{\lambda}{2}\|\vxh\|_p^2 - \cls^* - \frac{\lambda}{2}\|\vx^*\|_p^2 - \frac{\lambda}{2}\|\vxh\|_p^2 + \frac{\lambda}{2}\|\vx^*\|_p^2\\
    \leq \; & \cls(\vxh) + \frac{\lambda}{2}\|\vxh\|_p^2 - \cls(\vxh^*) - \frac{\lambda}{2}\|\vxh^*\|_p^2 + \frac{\lambda}{2}\|\vx^*\|_p^2\\
    \leq \; & \epsilon_n + \frac{\lambda}{2}\|\vx^*\|_p^2,
\end{align*}
where the first inequality follows from $\vxh^*$ being the minimizer of \eqref{eq:reg-emp-risk-min} and $\frac{\lambda}{2}\|\vxh\|_p^2 \geq 0,$ and the last inequality is by the definition of $\epsilon_n.$}
\qed\end{proof}

\markupadd{Note that Proposition~\ref{prop:reg-emp-risk-opt} allows us to treat the empirical problem~\eqref{eq:reg-emp-risk-min} as if it were a constrained optimization problem, with constraint set $\{\vx \in \rr^d: \|\vx\|_p \leq \min_{\vx^*\in \cx^*_{\mathcal{S}}} \|\vx^*\|_p + \sqrt{\frac{2\epsilon_n}{\lambda(p-1)}}\}$, as the predictor $\vxh$ is guaranteed to lie in this set.}

\markupadd{We now specify the assumptions used for obtaining statistical and computational guarantees.
\begin{assumptions}\label{assump:pop-risk-min}
Given the population risk minimization problem \eqref{eq:pop-risk-minimization}, the following all hold
\begin{enumerate}[label=(C.\arabic*), leftmargin=1cm]
    \item For any set $\mathcal{S} = \{\vz_1, \dots, \vz_n\}$ of empirical samples, the set of empirical minimizers $\cx^*_{\mathcal{S}}$ is non-empty and  $\min_{\vx \in \cx^*_{\mathcal{S}}}\|\vx\|_p \leq B$, where $B < \infty$; \label{assump-item:bounded-emp-min}
    \item The loss function $\ell$ is differentiable and satisfies that for any two samples $\vz, \vz'$ drawn from $\mathcal{D}$ and $\vx$ such that $\|\vx\|_p \leq B,$ we have $\|\nabla \ell(\vx; \vz) - \nabla \ell(\vx; \vz')\|_* \leq M,$ where $M < \infty;$ \label{assump-item:RO-grad-loss}
    \item The loss function $\ell$ is $L$-smooth for some $L < \infty.$ \label{assump-item:smoothness}
    \item For any $\vx$ such that $\|\vx\|_p \leq 2B,$ $\ee_{\vz \sim \mathcal{D}}[|\ell(\vx; \vz) - \cL(\vx)|] \leq G.$ \label{assump-item:bnded-var-fn}
\end{enumerate}
\end{assumptions}}

\markupadd{Assumption~\ref{assump-item:bounded-emp-min} ensures learnability of the population risk minimization problem, and some variant of it is necessary. It is usually enforced via a stronger condition that the minimization is performed over a bounded convex set. Assumption~\ref{assump-item:RO-grad-loss} is looser than the assumption about Lipschitz-continuity of $\ell$ that is typically enforced in the literature on stability and generalization. Further, due to Assumptions~\ref{assump-item:bounded-emp-min} and \ref{assump-item:smoothness} and the criterion $\|\vx\|_p \leq B$ in its statement, Assumption~\ref{assump-item:RO-grad-loss} holds as in this case $\|\nabla \ell(\vx,\vz)\|_* \leq M$ is bounded (by $2LB$). Assumption~\ref{assump-item:smoothness} is made for computational tractability via the complexity bound of AGD+, and can be replaced with an assumption that $\ell$ is $(\kappa, L)$-weakly smooth for any $\kappa \in [1, 2],$ with all of the analysis still being applicable and just by invoking the appropriate complexity bound for this class (observe that constant $M$ in Assumption~\ref{assump-item:RO-grad-loss} can still be bounded by $2LB^{\kappa - 1}$ in this case). However, for concreteness and simplicity of exposition, we carry out the analysis under the assumption that $\ell$ is $L$-smooth. Finally, Assumption~\ref{assump-item:bnded-var-fn} is made to be able to apply \cite[Theorem~8]{shalev2010learnability}, which relates uniform replace one stability (as in Lemma~\ref{lemma:uniform-RO}) to consistency and generalization. Note that this assumption can be omitted, since it is automatically satisfied for any non-degenerate problem, as $\ell$ is a continuous function and as such bounded on the compact domain $\|\vx\|_p \leq 2B$. In particular, by Assumption~\ref{assump-item:bounded-emp-min}, $\ell$ has at least one minimizer $\vx^*$ that belongs to the $\ell_p$-ball $\|\vx\|_p \leq B.$ As $\ell$ is $L$-smooth (by Assumption~\ref{assump-item:smoothness}), $\sup_{\vx: \|\vx\|_p \leq 2B}\{\ell(\vx) - \ell(\vx^*)\} \leq \sup_{\vx: \|\vx\|_p \leq 2B}  \frac{L}{2}\|\vx - \vx^*\|_p^2 \leq \frac{9 LB^2}{2}.$ Thus, $G \leq \frac{9 LB^2}{2}.$}

\markupadd{The key property that allows us to prove consistency and generalization bounds is uniform replace one (RO) stability. For completeness, we first define uniform RO stability, consistency, and generalization, and then move on to proving the claimed bounds. The definitions provided below can be found in, e.g.,~\cite{shalev2010learnability}.}

\markupadd{
\begin{definition}\label{def:RO-uniformly-stable}
Let $A$ be a rule that given a sample $\mathcal{S} = \{\vz_1, \dots, \vz_n\}$ and problem \eqref{eq:pop-risk-minimization} outputs a predictor $A(\mathcal{S}).$ $A$ is said to be uniform-RO stable with rate $\epsilon_{\mathrm{st}}(n)$, if for all possible sets $\mathcal{S}^{(i)} = \{\vz_1, \dots, \vz_{i-1}, \vz_i', \vz_{i+1}, \dots \vz_n\}$ that replace the $i^{\mathrm{th}}$ sample $\vz_i$ by some $\vz_i'$ and for any $\vz'$ from the support of $\mathcal{D},$ we have
$$
\frac{1}{n}\sum_{i=1}^n |\ell(A(\mathcal{S}); \vz') - \ell(A(\mathcal{S}^{(i)}); \vz')| \leq  \epsilon_{\mathrm{st}}(n). 
$$
\end{definition}}

\markupadd{\begin{definition}
A learning rule $A$ is said to be consistent with rate $\epsilon_{\mathrm{cons}}(n)$ under distribution $\mathcal{D}$ if for all $n \geq 1,$
$$
    \ee_{\mathcal{S}\sim \mathcal{D}^n} [\cL(A(\mathcal{S})) - \inf_{\vx \in \rr^d}\cL(\vx)] \leq \epsilon_{\mathrm{cons}}(n).
$$
A learning problem is learnable if there exists a learning rule $A$ that is consistent with rate $\epsilon_{\mathrm{st}}(n)$ under any distribution $\mathcal{D}$ and where $\epsilon_{\mathrm{st}}(n) \stackrel{n \rightarrow\infty}{\longrightarrow} 0.$ If it exists, such a rule is then called a universally consistent learning rule.
\end{definition}
}

\markupadd{
\begin{definition}
A rule $A$ is said to generalize with rate $\epsilon_{\mathrm{gen}}(n)$ under distribution $\mathcal{D}$ if for all $n \geq 1,$
$$
    \ee_{\mathcal{S}\sim \mathcal{D}^n}[|\cL(A(\mathcal{S})) - \cls(A(\mathcal{S}))|] \leq \epsilon_{\mathrm{gen}}(n).
$$
A rule is said to universally generalize with rate $\epsilon_{\mathrm{gen}}(n)$ if it generalizes with rate $\epsilon_{\mathrm{gen}}(n)$ irrespective of the distribution over the given support. 
\end{definition}}

\markupadd{To prove the consistency and generalization rates, we prove the following lemma that certifies uniform RO stability of the outputs of AGD+ (under a suitable choice of $\epsilon_n$ and $\lambda$). We then obtain the consistency and generalization rates as an application of \cite[Theorem~8]{shalev2010learnability}.}

\markupadd{
\begin{lemma}\label{lemma:uniform-RO}
Given the population risk minimization problem \eqref{eq:pop-risk-minimization}, let $\vxh_{\mathcal{S}}$ be an $\epsilon_n$-approximate solution for the regularized empirical problem formulation in \eqref{eq:reg-emp-risk-min}. Then $\vxh_{\mathcal{S}}$ is a uniform RO stable learning rule with rate
\begin{equation}
    \epsilon_{\mathrm{st}}(n) = L \Big(2B + \sqrt{\frac{2\epsilon_n}{\lambda(p-1)}}\Big) \Big(\frac{M}{n \lambda(p-1)} + 2 \sqrt{\frac{2 \epsilon_n}{\lambda(p-1)}}\Big),
\end{equation}
under any distribution that satisfies Assumption~\ref{assump-item:bounded-emp-min} and for any loss function $\ell$ that satisfies Assumptions~\ref{assump-item:RO-grad-loss} and \ref{assump-item:smoothness}.  
\end{lemma}}
\begin{proof}
\markupadd{Let $\mathcal{S} = \{\vz_1, \dots, \vz_n\}\stackrel{i.i.d.}{\sim} \mathcal{D}.$ Let $\mathcal{S}^{(i)}$ be any set obtained from $\mathcal{S}$ by replacing the $i^{\mathrm{th}}$ element ($\vz_i$) by an independent sample $\vz_i'\sim {\cal D}$. Let
$
    \vxh^*_{\mathcal{S}^{(i)}}
$ 
be the minimum $\ell_p$-norm solution to \eqref{eq:reg-emp-risk-min} with sample $\mathcal{S}^{(i)}$ and $\vxh_{\mathcal{S}^{(i)}}$ be the approximate solution to the same problem output by AGD+. Similarly, let $
    \vxh^*_{\mathcal{S}}
$ 
be the minimum $\ell_p$-norm solution to \eqref{eq:reg-emp-risk-min} with sample $\mathcal{S}$ and $\vxh_{\mathcal{S}}$ be the approximate solution to the same problem output by AGD+. To prove the lemma, we first bound $\|\vxh^*_{\cS} - \vxh^*_{\cS^{(i)}}\|_p$ and then use Proposition~\ref{prop:reg-emp-risk-opt} with smoothness of $\ell$ to conclude that $\vxh_{\cS}$ is uniform RO stable.}

\markupadd{As $\vxh^*_{\cS}$ minimizes $\cL_{\cS}(\vx) + \frac{\lambda}{2}\|\vx\|_p^2,$ we have $\nabla \cL_{\cS}(\vxh^*_{\cS}) + \nabla \big(\frac{1}{2}\|\vxh^*_{\cS}\|_p^2\big) = \zeros$, and, similarly,  $\nabla \cL_{\cS^{(i)}}(\vxh^*_{\cS^{(i)}}) + \nabla \big(\frac{1}{2}\|\vxh^*_{\cS^{(i)}}\|_p^2\big) = \zeros$. Further, as $\cL_{\cS}(\vx) + \frac{\lambda}{2}\|\vx\|_p^2$ is $\lambda(p-1)$-strongly convex w.r.t.~$\|\cdot\|_p,$
\begin{align}
    &\innp{\nabla \cL_{\cS}(\vxh^*_{\cS}) - \nabla \cL_{\cS}(\vxh^*_{\cS^{(i)}}), \vxh^*_{\cS} -  \vxh^*_{\cS^{(i)}}}\notag\\ 
    \geq\;& \lambda(p-1)\|\vxh^*_{\cS} -  \vxh^*_{\cS^{(i)}}\|_p^2 - \innp{\nabla \Big(\frac{1}{2}\|\vxh^*_{\cS}\|_p^2\Big) - \nabla \Big(\frac{1}{2}\|\vxh^*_{\cS^{(i)}}\|_p^2\Big), \vxh^*_{\cS} -  \vxh^*_{\cS^{(i)}}}\notag\\
    =\;& \lambda(p-1)\|\vxh^*_{\cS} -  \vxh^*_{\cS^{(i)}}\|_p^2 + \innp{\nabla \cL_{\cS}(\vxh^*_{\cS}) - \nabla \cL_{\cS^{(i)}}(\vxh^*_{\cS^{(i)}}), \vxh^*_{\cS} -  \vxh^*_{\cS^{(i)}}}\label{eq:stability-left-bnd}
\end{align}
Further, observe that by definition of $\cL_{\cS},$ we have $\nabla \cL_{\cS}(\vxh^*_{\cS}) - \nabla \cL_{\cS}(\vxh^*_{\cS^{(i)}}) = \nabla \cL_{\cS}(\vxh^*_{\cS}) - \nabla \cL_{\cS^{(i)}}(\vxh^*_{\cS^{(i)}}) + \frac{1}{n}(\nabla\ell(\vxh^*_{\cS^{(i)}}, \vz_i') - \nabla\ell(\vxh^*_{\cS^{(i)}}, \vz_i))$. Hence:
\begin{equation}\label{eq:stability-right-bnd}
    \begin{aligned}
         &\innp{\nabla \cL_{\cS}(\vxh^*_{\cS}) - \nabla \cL_{\cS}(\vxh^*_{\cS^{(i)}}), \vxh^*_{\cS} -  \vxh^*_{\cS^{(i)}}} \\
         \leq\; & \innp{\nabla \cL_{\cS}(\vxh^*_{\cS}) - \nabla \cL_{\cS^{(i)}}(\vxh^*_{\cS^{(i)}}), \vxh^*_{\cS} -  \vxh^*_{\cS^{(i)}}}\\
         &+ \frac{1}{n}\innp{\nabla\ell(\vxh^*_{\cS^{(i)}}, \vz_i') - \nabla\ell(\vxh^*_{\cS^{(i)}}, \vz_i), \vxh^*_{\cS} -  \vxh^*_{\cS^{(i)}}}
    \end{aligned}
\end{equation}
Hence, combining \eqref{eq:stability-left-bnd} and \eqref{eq:stability-right-bnd}, we have
\begin{align}
    \lambda(p-1) \|\vxh^*_{\cS} -  \vxh^*_{\cS^{(i)}}\|_p^2 &\leq \frac{1}{n} \innp{\nabla\ell(\vxh^*_{\cS^{(i)}}, \vz_i') - \nabla\ell(\vxh^*_{\cS^{(i)}}, \vz_i), \vxh^*_{\cS} -  \vxh^*_{\cS^{(i)}}}\notag\\
    &\leq \frac{M}{n}\|\vxh^*_{\cS} -  \vxh^*_{\cS^{(i)}}\|_p.\notag
\end{align}
Thus, we conclude that
\begin{equation}\notag
    \|\vxh^*_{\cS} -  \vxh^*_{\cS^{(i)}}\|_p \leq \frac{M}{n \lambda(p-1)}, 
\end{equation}
and, consequently, using strong convexity and the guarantee of AGD+ as in the proof of Proposition~\ref{prop:reg-emp-risk-opt}, we have
\begin{equation}\label{eq:bnd-dist-predictors}
\begin{aligned}
    \|\vxh_{\cS} -  \vxh_{\cS^{(i)}}\|_p &\leq \|\vxh^*_{\cS} -  \vxh^*_{\cS^{(i)}}\|_p + \|\vxh_{\cS} -  \vxh^*_{\cS}\|_p + \|\vxh_{\cS^{(i)}} -  \vxh^*_{\cS^{(i)}}\|_p \\
    &\leq \frac{M}{n \lambda(p-1)} + 2 \sqrt{\frac{2 \epsilon_n}{\lambda(p-1)}}.
\end{aligned}
\end{equation}
Finally, let $\vx^*_1$ be the minimizer of $\ell(\cdot; \vz')$ with the minimum $\ell_p$ norm. By Assumption~\ref{assump-item:bounded-emp-min}, $\|\vx^*_1\|_p \leq B.$ Hence, using smoothness of $\ell, $ we have 
\begin{align*}
    |\ell(\vxh_{\cS}; \vz') - \ell(\vxh_{\cS^{(i)}}; \vz')| &\leq |\innp{\nabla \ell(\vxh_{\cS}; \vz'), \vxh_{\cS} - \vxh_{\cS^{(i)}}}|\\
    &\leq \|\nabla \ell(\vxh_{\cS}; \vz') - \nabla \ell(\vx^*_1; \vz')\|_{p^*}\|\vxh_{\cS} - \vxh_{\cS^{(i)}}\|_p\\
    &\leq L \|\vxh_{\cS} - \vx^*_1\|_p \|\vxh_{\cS} - \vxh_{\cS^{(i)}}\|_p\\
    &\leq L(\|\vxh_{\cS}\|_p + \|\vx^*_1\|_p)\|\vxh_{\cS} - \vxh_{\cS^{(i)}}\|_p\\
    &\leq L \Big(2B + \sqrt{\frac{2\epsilon_n}{\lambda(p-1)}}\Big)\|\vxh_{\cS} - \vxh_{\cS^{(i)}}\|_p\\
    &\leq L \Big(2B + \sqrt{\frac{2\epsilon_n}{\lambda(p-1)}}\Big) \Big(\frac{M}{n \lambda(p-1)} + 2 \sqrt{\frac{2 \epsilon_n}{\lambda(p-1)}}\Big),
\end{align*}
where the second to last inequality uses the third part of Proposition~\ref{prop:reg-emp-risk-opt} and the last inequality uses \eqref{eq:bnd-dist-predictors}.}
\end{proof}

\markupadd{We are now ready to state and prove computational and statistical guarantees for a learning rule $\vxh_{\cS}$ defined as an $\epsilon_n$-approximate solution to \eqref{eq:reg-emp-risk-min}, which can be computed using AGD+.}

\markupadd{
\begin{theorem}\label{thm:reg-emp-risk-min}
Given the population risk minimization problem \eqref{eq:pop-risk-minimization} that satisfies Assumptions~\ref{assump-item:bounded-emp-min}--\ref{assump-item:bnded-var-fn}, let $\vxh_{\mathcal{S}}$ be an $\epsilon_n$-approximate solution for the regularized empirical problem formulation in \eqref{eq:reg-emp-risk-min}, where $\epsilon_n$ and $\lambda$ are defined by:
\begin{equation}
    \epsilon_n = \min\Big\{ \frac{B^2 \lambda(p-1)}{2},\, \frac{M^2}{8n^2 \lambda(p-1)} \Big\},\; \lambda = 2 \sqrt{3}\sqrt{\frac{LM}{Bnp(p-1)}}.
\end{equation}
Then $\vxh_{\mathcal{S}}$ can be computed with 
\begin{equation}
\begin{aligned}
    k &= O\bigg(\sqrt{\frac{L}{\lambda}}\log\Big(\frac{LB}{\epsilon_n}\Big)\bigg)\\
    &= O \bigg(\Big(\frac{nBL}{M}\Big)^{1/4}\log\Big(\frac{nBL}{p-1}\Big)\bigg)
\end{aligned}
\end{equation}
iterations of AGD+ and it is consistent and generalizes under $\mathcal{D}$ with rates
\begin{equation}\notag
    \begin{aligned}
        \epsilon_{\mathrm{cons}}(n) = 2\sqrt{3}\sqrt{\frac{p}{p-1}\cdot \frac{LM}{n}}B^{3/2}\\
        \epsilon_{\mathrm{gen}}(n) = 2\sqrt{3}\sqrt{\frac{p}{p-1}\cdot \frac{LM}{n}}B^{3/2} + \frac{2G}{\sqrt{n}}.
    \end{aligned}
\end{equation}
\end{theorem}}
\begin{proof}
\markupadd{The bound on the number of iterations of AGD+ required to compute $\vxh_{\cS}$ follows as a direct application of Theorem~\ref{thm:comp-opt-gap-conv}.}

\markupadd{For the consistency and generalization rates, we apply \cite[Theorem~8]{shalev2010learnability}, by which
\begin{align}
    \epsilon_{\mathrm{cons}}(n) &\leq \epsilon_{\mathrm{st}}(n) + \epsilon_{\mathrm{erm}}(n),\notag\\
    \epsilon_{\mathrm{gen}}(n) &\leq \epsilon_{\mathrm{st}}(n) + \epsilon_{\mathrm{erm}}(n) + \frac{2G}{\sqrt{n}},\notag
\end{align}
where
\begin{equation}
    \epsilon_{\mathrm{erm}}(n) := \cls(\vxh) - \cls^* \leq \frac{\lambda}{2}B^2 + \epsilon_n\notag
\end{equation}
(by Proposition~\ref{prop:reg-emp-risk-opt}) and $\epsilon_{\mathrm{st}}(n)$ was bounded in Lemma~\ref{lemma:uniform-RO}. The claimed bounds now follow after plugging in the choice of $\lambda, \epsilon_n$ from the theorem statement, and simplifying.}
\qed\end{proof}

\markupadd{A few remarks are in order regarding the practical use of the proposed learning rule based on \eqref{eq:reg-emp-risk-min}. In Theorem~\ref{thm:reg-emp-risk-min}, we choose $\epsilon_n$ and $\lambda$ according to the problem parameters specified in Assumptions~\ref{assump-item:bounded-emp-min}-\ref{assump-item:bnded-var-fn}. However, parameters $B, M, L$ are usually not available at the input. Further, if the minimum loss $\cls^*$ is not known, we cannot use $f(\vy_k) - f(\vx^*) \leq \epsilon_n$ as a stopping criterion in AGD+. (Recall that, as discussed in Section~\ref{sec:comp-considerations}, knowledge of $L$ is not required for running AGD+, as $L$ can be adaptively estimated.) In such situations, it suffices to let $\lambda = \Theta\Big(\sqrt{\frac{1}{np(p-1)}}\Big)$,  $\epsilon_n = \Theta \big(\frac{p-1}{n^2}\big),$ and run AGD+ until the number of iterations reaches the estimate $C\sqrt{\frac{L}{\lambda}\log(nL/(p-1))}$ for some constant $C$ and the adaptively estimated value of $L.$ It is simple to verify that under this choice, $\epsilon_{\mathrm{cons}}(n)$ and $\epsilon_{\mathrm{gen}}(n)$ grow as $\frac{1}{\sqrt{n(p-1)}}$ as functions of $n$ and $p$ (albeit with a worse, though still polynomial, dependence on the remaining parameters).}

\markupadd{Finally, when $p = 1 + \frac{c}{\log(d)}$ for $c >0,$ we have $\|\cdot\|_p\leq \|\cdot\|_1 \leq (1 + c)\|\cdot\|_p.$ Hence, the proposed regularized empirical risk minimization approach can be used as an alternative to LASSO under this choice of $p$, based on the discussion preceding Proposition~\ref{prop:reg-emp-risk-opt}. Note that, as shown in \cite[Section~4.3]{shalev2010learnability}, \cite[Section~3.3]{feldman2016generalization}, for general loss functions and under the assumptions from this section (Assumptions~\ref{assump:pop-risk-min}), the same consistency and generalization rates cannot be obtained using LASSO (regardless of whether it is formulated using $\ell_1$ regularization or via bounding the optimization problem using $\ell_1$ ball constraints).}

\markupdelete{Bridge regression problems were originally introduced by~\cite{frank1993statistical}, and are defined by
\begin{equation}\label{eq:bridge-reg}
    \begin{aligned}
        \min_{\substack{\vx \in \rr^d:\\ \|\vx\|_p \leq t}}&\; \frac{1}{2}\|\mA \vx - \vb\|_2^2
    \end{aligned}
\end{equation}
where $t$ is a positive scalar, $p \in [1, 2],$ $\mA$ is the matrix of observations, and $\vb$ is the vector of labels. In particular, for $p = 1,$ the problem reduces to lasso, while for $p = 2$ we recover ridge regression.}

\markupdelete{Bridge regression has traditionally been used either as an interpolation between lasso and ridge regression, or to model Bayesian priors with the exponential power distribution (see~\cite{park2008bayesian} and~\cite[Section~3.4.3]{hastie2009elements}). By reformulation with the introducion of a parameter $\lambda>0$, the problem is often posed in the equivalent (due to Lagrangian duality) penalized (or regularized) form:
$$
    \min_{\vx \in \rr^d} \Big\{\frac{1}{2}\|\mA \vx - \vb\|_2^2 + \frac{\lambda}{p}\|\vx\|_p^p\Big\}.
$$
Writing the regularizer as $\frac{1}{p}\|\vx\|_p^p$ is typically chosen due to its separable form. However, using different parametrization, the problem from Eq.~\eqref{eq:bridge-reg} is also equivalent to 
\begin{equation}\label{eq:bridge-reg-1}
    \min_{\vx \in \rr^d} \Big\{\frac{1}{2}\|\mA \vx - \vb\|_2^2 + \frac{\lambda}{2}\|\vx\|_p^2\Big\},
\end{equation}
which is more convenient for the application of our results, as $\frac{1}{2}\|\vx\|_p^2$ is $(p-1)$-strongly convex w.r.t.~$\|\cdot\|_p$. }

\markupdelete{Further, looking at the gradient $\nabla f(\vx) = \mA^T\mA \vx - \mA^T \vb$ of $f(\vx) = \frac{1}{2}\|\mA \vx - \vb\|_2^2$, it is not hard to argue that $f(\vx)$ is $L_p$-smooth w.r.t.~$\|\cdot\|_p,$ where $L_p = \|\mA^T\mA\|_{p \to p_{\ast}} = \sup_{\vx \in \rr^d: \|\vx\|_p \neq 0}\frac{\|\mA^T\mA\vx\|_{p_{\ast}}}{\|\vx\|_p}$. Namely, this follows as}
  \markupdelete{\begin{equation*}
  \|\nabla f(\vx) - \nabla f(\vy)\|_{p_{\ast}} = \|\mA^T\mA(\vx - \vy)\|_{p_{\ast}} \leq \|\mA^T\mA\|_{p \to p_{\ast}}\|\vx - \vy\|_p.
\end{equation*}}
\markupdelete{An interesting feature of the formulation in Eq.~\eqref{eq:bridge-reg-1} is that it implies a certain trade-off between the $p_{\ast}$-fit of the data and the $p$-norm of the regressor. Namely, if $\vxb^{\star}$ solves the problem from Eq.~\eqref{eq:bridge-reg-1}, then}
  \markupdelete{\begin{equation}\label{eq:grad-of-bridge}
  \|\mA^T(\mA\vxb^{\star} - \vb)\|_{p_{\ast}} = \lambda \|\vxb^{\star}\|_p. 
\end{equation}}
\markupdelete{This simply follows by setting the gradient of $\frac{1}{2}\|\mA \vx - \vb\|_2^2 + \frac{1}{2}\|\vx\|_p^2$ to zero, and using that $\big\| \nabla \big(\frac{1}{2}\|\vx\|_p^2\big)\big\|_{p_{\ast}} = \|\vx\|_p,$ $\forall \vx \in \rr^d$ (see Proposition~\ref{prop:duality-map-of-p}). }

\markupdelete{More recently, related problems of the form
$$
    \min_{\vx \in \rr^d} \Big\{\sqrt{\ell(\vx, \mA, \vb)} + \lambda'\|\vx\|_p\Big\},
$$
where $\ell(\vx, \mA, \vb)$ is a more general loss function, have been used in distributionally robust optimization (see~\cite{blanchet2019robust}). Again, a different parametrization of the same problem leads to the equivalent form
\begin{equation}
    \min_{\vx \in \rr^d} \Big\{{\ell(\vx, \mA, \vb)} + \frac{\lambda}{2}\|\vx\|_p^2\Big\}.
\end{equation}
{Consider the complementary composite formulation where $f(\vx)=\ell(\vx, \mA, \vb)$ and $\psi(\vx)=\frac{\lambda}{2}\|\vx\|_p^2$. Then our results can be applied as long as $\ell(\vx, \mA, \vb)$ is $L_p$-smooth w.r.t.~$\|\cdot\|_p$.}\footnote{Note that, by the inequalities relating $\ell_p$-norms, any function that is $L$-smooth w.r.t.~\mbox{$\|\cdot\|_2$}, is also $L$-smooth w.r.t.~$\|\cdot\|_p$ for $p \in [1, 2]$. That is, for $p \in [1, 2],$ the smoothness parameter w.r.t.~$\|\cdot\|_p$ can only be lower than the smoothness parameter w.r.t.~$\|\cdot\|_2$, often being significantly lower.}}

\markupdelete{A direct application of our result from  Theorem~\ref{thm:comp-opt-gap-conv} tells us that we can approximate the problem from Eq.~\eqref{eq:bridge-reg-1} with accuracy $\epsilon > 0$ using 
\begin{equation}\label{eq:k-bridge}
    k = O\bigg(\min\bigg\{\sqrt{\frac{L_p}{ \lambda(p-1) }}\log\Big(\frac{L_p \|\vxb^{\star} - \vx_0\|_p}{\epsilon}\Big),\; \sqrt{\frac{L_p\|\vxb^{\star} - \vx_0\|_p^2}{\epsilon}}\bigg\} \bigg)
\end{equation}
iterations of Generalized AGD+ from Eq.~\eqref{eq:mod-agd+}. We are not aware of any other results that provide such a guarantee, for $p \neq 2$. }

\markupdelete{Further, using Corollary~\ref{cor:comp-opt-dist-to-opt}, we get that within the same number of iterations the output point $\vy_k$ of the algorithm satisfies $\|\vy_k - \vxb^{\star}\|_p \leq \sqrt{\frac{2\epsilon}{\lambda(p-1)}}.$ Additionally, for quadratic losses, using triangle inequality and Eq.~\eqref{eq:grad-of-bridge}, we have the following ``goodness of fit'' guarantee
\begin{align*}
    \|\mA^T(\mA\vy_k - \vb)\|_{p_{\ast}} \leq \|\mA^T \mA (\vy_k - \vxb^{\star})\|_{p_{\ast}} + \lambda\|\vxb^{\star}\|_p \leq L_p \sqrt{\frac{2\epsilon}{\lambda(p-1)}} + \lambda\|\vxb^{\star}\|_p.
\end{align*}}

\markupdelete{Finally, note that it is possible to apply our algorithm to $\ell_1$ regularized problems (lasso), applying results from Theorem~\ref{thm:comp-opt-gap-conv} with $\psi(\vx) = \lambda \|\vx\|_1$ and $\phi(\vx) = \frac{1}{2}\|\vx - \vx_0\|_2^2.$ In this case, as $\psi$ is not strongly convex, the resulting bound is $k = O\Big(\sqrt{\frac{L_2 \|\vxb^{\star} - \vx_0\|_2^2}{\epsilon}}\Big)$, which matches the iteration complexity of FISTA~\cite{beck2009fast} and composite minimization framework of~\cite{nesterov2013gradient}.} %Alternatively, if we let $\phi(\vx) = \frac{\log(d)}{2}\|\vx - \vx_0\|_q^2$ for $q = 1 + \frac{1}{\log(d)},$ which is $1$-strongly convex w.r.t.~$\|\cdot\|_1,$ the resulting bound is then $k = O\Big(\sqrt{\frac{L_1 \phi(\vxb^{\star})}{\epsilon}}\Big) = O\Big(\sqrt{\frac{\log(d) L_1 \|\vxb^{\star} - \vx_0\|_1^2 }{\epsilon}}\Big)$. For the quadratic loss function, $L_1 = \|\mA^T\mA\|_{1\to\infty} = \max_{1\leq i, j \leq d}(\mA^T\mA)_{ij}$. Using the well-known Gershgorin circle theorem, it follows that $L_1$ can be up to a factor of $d$ smaller than $L_2 = \|\mA^T\mA\|_{2\to2} = \lambda_{\max}(\mA^T\mA),$ where $\lambda_{\max}$ denotes the largest eigenvalue of a matrix. Thus, the second bound is generally tighter by a factor of $\widetilde{O}(\sqrt{d}).$ 

%note that, as $\|\cdot\|_1$ is approximated by $\|\cdot\|_p$ to a multiplicative factor $1 + \epsilon$ for $p = 1 + \Theta(\frac{\epsilon}{\log(d)})$, if our results are applied to lasso problems using $\frac{\lambda}{p}\|\cdot\|_p^2$ as the regularizer (or even using the $\ell_1$ regularizer and $\phi(\vx) = \frac{1}{p}\|\vx - \vx_0\|_p^2$ with $p = 1 + \frac{1}{\log(d)}$, which is $\Theta(\frac{1}{\log(d)})$-strongly convex w.r.t.~$\|\cdot\|_1$), we recover the same iteration complexity as that of FISTA~\cite{beck2009fast}, which is unimprovable. 

%%%%%%%%%%%%%%%%
\markupdelete{\subsection{Dantzig Selector Problem}\label{sec:dantzig-selector}}

\markupdelete{Dantzig selector problem, introduced by %Cand\'{e}s and Tao~
\cite{candes2007dantzig}, consists in solving problems of the form
$$
    \min_{\substack{\vx \in \rr^d:\\ \|\vx\|_1 \leq t}}\|\mA^T(\mA\vx - \vb)\|_{\infty}, \quad\text{ or, equivalently }\quad \min_{\substack{\vx \in \rr^d:\\ \|\mA^T(\mA\vx - \vb)\|_{\infty} \leq t}} \|\vx\|_1,
$$
where $t$ is some positive parameter. }

\markupdelete{Similar to other regression problems described in this section, Dantzig selector problem can be considered in its unconstrained, regularized form. One variant of the problem that can be addressed with our algorithm is
\begin{equation}\label{eq:dantzig-selector}
    \min_{\vx \in \rr^d} \frac{1}{2}\|\mA^T(\mA\vx - \vb)\|_{p_{\ast}}^2 + \frac{\lambda}{2} \|\vx\|_p^2, 
\end{equation}
where $p$ is chosen sufficiently close to one so that $\|\cdot\|_p$ closely approximates $\|\cdot\|_{1}$ and $\|\cdot\|_{p_{\ast}}$ closely approximates $\|\cdot\|_{\infty}$, where $\frac{1}{p}+\frac{1}{p_{\ast}} = 1$. In particular, when %$p =1+ \Theta( \frac{\epsilon}{\log(d)}),
$p^{\ast}=[\log d]/\ln(1+\epsilon)$, %\todo[]{Modified slightly} 
we have that $(1-\epsilon) \|\vx\|_{1} \leq \|\vx\|_p \leq \|\vx\|_{1}$ and $\|\vx\|_{\infty} \leq \|\vx\|_p \leq (1+\epsilon)\|\vx\|_{\infty},$ $\forall \vx \in \rr^d.$}

\markupdelete{As discussed at the beginning of Section~\ref{sec:small-grad-norm}, in this case, $\psi(\vx) = \frac{\lambda}{2} \|\vx\|_p^2$ is $\lambda(p-1) = \Theta( \frac{\lambda\epsilon}{\log(d)})$-strongly convex w.r.t.~$\|\cdot\|_p$ and, by the relationship between norms, is also strongly convex w.r.t.~$\|\cdot\|_1$ with the strong convexity constant of the same order. Further, $f(\vx) = \frac{1}{2}\|\mA^T(\mA\vx - \vb)\|_{p_{\ast}}^2$ can be shown to be $L_1$-smooth w.r.t.~$\|\cdot\|_1$, for $L_1 = (1+\epsilon)(p_{\ast}-1) A_{\max} = \Theta(\frac{\log{d}}{\epsilon}A_{\max}),$ where $A_{\max} = \max_{1\leq i, j \leq d}|(\mA^T\mA)_{ij}|.$ This can be done as follows. Using that $\frac{1}{2}\|\cdot\|_{p_{\ast}}^2$ is $(p_{\ast}-1)$-smooth w.r.t.~$\|\cdot\|_{p_{\ast}}$ (as $p > 2$), we have, $\forall \vx, \vy \in \rr^d,$
\begin{align*}
    \|\nabla f(\vx) - \nabla f(\vy)\|_{\infty} &\leq \|\nabla f(\vx) - \nabla f(\vy)\|_{p}\\
    &\leq (p_{\ast}-1)\|(\mA^T\mA)(\vx - \vy)\|_{p_{\ast}}\\
    &\leq (p_{\ast}-1)\|\mA^T\mA\|_{1 \to p_{\ast}}\|\vx - \vy\|_1\\
    &\leq (p_{\ast}-1)(1+\epsilon)\|\mA^T\mA\|_{1 \to \infty}\|\vx - \vy\|_1\\
    &= (1+\epsilon)(p_{\ast}-1) \max_{1\leq i, j \leq d}|(\mA^T\mA)_{ij}|\cdot\|\vx - \vy\|_1.
\end{align*}
Hence, applying Theorem~\ref{thm:comp-opt-gap-conv}, we have that the problem from Eq.~\eqref{eq:dantzig-selector} can be approximated to arbitrary additive error $\bar{\epsilon}$ with $$k = O\bigg(\sqrt{\frac{A_{\max}}{\lambda}}\frac{\log(d)}{\bar{\epsilon}}\log\Big( \frac{\log(d)A_{\max}\|\vxb^{\star} - \vx_0\|}{\bar{\epsilon}}\Big)\bigg)$$ iterations of Generalized AGD+ from Section~\ref{sec:comp-min}.} %This bound is conjectured to be unimprovable (see~\cite{nesterov2013first}).\todo[]{Not sure how $\lambda$ shows up in [NN:13]. They also get linear dependence on the optimal value of the problem, which I don't see here. It might be OK if we keep the optimality claim w.r.t. $\epsilon$.}

\markupdelete{Similar to bridge regression, there is an interesting trade-off between %\todo[]{Parhaps say ''a relaxed form of sparsity'' or ''parsimony''? And then clarify at the end of the paragraph that $1$-norm is usually used as a proxy for sparsity.} 
the $\ell_1$ norm of the regressor and goodness of fit revealed by the formulation we consider (Eq.~\eqref{eq:dantzig-selector}). In particular, using that at an optimal solution $\bar{\vx}^{\star}$ the gradient of the objective from Eq.~\eqref{eq:dantzig-selector} is zero and using Proposition~\ref{prop:duality-map-of-p},
\begin{align*}
    ({1-\epsilon})\lambda \|\vxb^{\star}\|_1 \leq {\lambda}\|\vxb^{\star}\|_p &= {\lambda}\Big\|\nabla \Big(\frac{1}{2}\|\vxb^{\star}\|_p^2\Big)\Big\|_{p_{\ast}}\\
    &= \Big\|\nabla \Big(\frac{1}{2}\|\mA^T(\mA\vxb^{\star} - \vb)\|_{p_{\ast}}^2\Big) \Big\|_{p_{\ast}}\\
    &\leq \|\mA^T \mA\|_{p \to p_{\ast}} \|\mA^T(\mA\vxb^{\star} - \vb)\|_{p_{\ast}}\\
    &\leq \frac{1+\epsilon}{1-\epsilon}A_{\max}\|\mA^T(\mA\vxb^{\star} - \vb)\|_{\infty}.
\end{align*}
Hence, $\lambda \|\vxb^{\star}\|_1 \leq (1 + O(\epsilon))A_{\max}\|\mA^T(\mA\vxb^{\star} - \vb)\|_{\infty}.$  As the $\ell_1$ norm of the regressor is considered a proxy for sparsity, this bound provides a trade-off between the parsimony of the model and the goodness of fit, as a function of the regularization parameter $\lambda$. We remark that the Dantzig selector problem has also been investigated in \cite{nesterov2013first}, however their results provide guarantees for a saddle-point formulation of the problem, which is not directly comparable to the optimization guarantee that we provide here. }

%%%%%%%%%%%%%%%%%%
\subsection{\markupadd{Solving Symmetric PSD Linear Systems with Maximum Constraint Violation Guarantee}}

\markupadd{When solving linear systems $\mA \vx = \vb$, we are often interested in the maximum constraint violation as opposed to the $\ell_2$ norm of the error vector $\mA \vx - \vb$, $\|\mA\vx - \vb\|_2$, obtained by minimizing the quadratic function $\|\mA \vx - \vb\|_2^2.$ When $\mA$ is symmetric and positive semidefinite (PSD), a common approach to solving linear systems is by minimizing the quadratic function $f(\vx) = \frac{1}{2}\vx^T \mA \vx - \vb^T \vx.$ The gradient of this quadratic function is precisely the error vector $\mA \vx - \vb$ for the linear system $\mA \vx = \vb,$ thus in this case we are interested in minimizing the gradient of $f$. }

\markupadd{If one uses a Euclidean first-order approach to minimize  the gradient of $f,$ then the resulting gradient oracle complexity to obtain $\|\mA \vx - \vb\|_{2} \leq \epsilon$ is $$\Theta\Big(\min\Big\{d,\, \sqrt{\frac{\|\mA\|_2\|\vx_0 - \vx^*\|_2}{\epsilon}}\Big\}\Big),$$ where $\vx_0$ is an initial point and $\vx^*$ is a solution to the linear system $\mA \vx = \vb$~\cite{nemirovsky1991optimality,nemirovsky1992information}. Note that $\|\mA \vx - \vb\|_\infty$ can be as large as $\|\mA \vx - \vb\|_2$ in the worst case. On the other hand, applying our result from Theorem~\ref{thm:grad-norm} with $p = 1 + \frac{c}{\log(d)}$, $c > 0,$ we get that $\|\mA \vx - \vb\|_{\infty} \leq \epsilon$ with gradient oracle  complexity $\widetilde{O}\big(\sqrt{\frac{\|\mA\|_{p \to p^*}\|\vx_0 - \vx^*\|_p}{c\epsilon}}\big) = \widetilde{O}\big(\sqrt{\frac{\max_{1 \leq i, j \leq d}A_{ij}\|\vx_0 - \vx^*\|_1}{c\epsilon}}\big).$ If the system $\mA \vx - \vb$ has sparse solutions, then selecting $\vx_0 = \zeros$, the obtained bound can be smaller by a factor $\sqrt{d}$ for constant $c$, as $\|\mA\|_2$ can be as large as $d\max_{1 \leq i, j \leq d}A_{ij}$.}\footnote{\markupadd{This bound is tight for the matrix of all ones.}} \markupadd{Further, as a consequence of the results from Section~\ref{sec:small-grad-norm}, our algorithm enforces small $\ell_p$-norm of the output solutions (and thus a small $\ell_1$ norm). In particular, due to Eq.~\eqref{eq:bnd-init-dist}, when solving the regularized problem from Section~\ref{sec:small-grad-norm} to obtain a solution with small gradient norm, it is guaranteed that $\|\vxb\|_p^* \leq \|\vx^*\|_p$, where $\vxb^*$ is the solution to the regularized problem and $\vx^*$ is any solution to the linear system $\mA \vx = \vb$. As a consequence, $\|\vxb\|_p^* \leq \min_{\vx: \mA\vx = \vb}\|\vx\|_p.$ Since the regularized problem is strongly convex, we further have that the output solution $\vy_k$ satisfies $\|\vy_k - \vxb^*\|_p^2 \leq \frac{1}{2\lambda}(f(\vy_k) - f(\vxb^*)).$ If we slightly change the target error of Generalized AGD+ in Theorem~\ref{thm:grad-norm} to guarantee that $f(\vy_k) - f(\vxb^*) \leq \frac{(p-1)\epsilon^3}{2}$ (which only affects the terms under the log factor and does not change the resulting asymptotic complexity), we get that $\|\vy_k - \vxb^*\|_p \leq \epsilon$. As a consequence, using the triangle inequality, we have}
\begin{equation}\notag
    \markupadd{\|\vy_k\|_p \leq \|\vy_k - \vxb^*\|_p + \|\vxb^*\|_p \leq (1 + \epsilon)\|\vxb^*\|_p \leq (1 + \epsilon)\min_{\vx: \mA\vx = \vb}\|\vx\|_p.}
\end{equation}
\markupadd{Finally, using properties of $\ell_p$ norms and our choice of $p,$ we have that }
\begin{equation}\notag
    \markupadd{\|\vy_k\|_1  \leq \Big(1 + O(c + \epsilon)\Big)\min_{\vx: \mA\vx = \vb}\|\vx\|_1.}
\end{equation}
%%%%%%%%%%%%%%%%%%
\subsection{$\ell_p$ Regression}\label{sec:ellp-reg}

Standard $\ell_p$-regression problems have as their goal finding a vector $\vx^{\star}$ that minimizes $\|\mA \vx - \vb\|_p,$ where $p \geq 1.$ When $p =1$ or $p = \infty,$ this problem can be solved using linear programming. More generally, when $p \notin \{1, \infty\},$ the problem is nonlinear, and multiple approaches have been developed for solving it, including, e.g., a homotopy-based solver~\cite{bubeck2018homotopy}, solvers based on iterative refinement~\cite{adil2019iterative,Adil:2020}, and solvers based on the classical method of iteratively reweighted least squares \cite{Ene:2019,Adil:2019}. 
Such solvers typically rely on fast linear system solves and  attain logarithmic dependence on the inverse accuracy $1/\epsilon,$ at the cost of iteration count scaling polynomially with one of the dimensions of  $\mA$ (typically the lower dimension, which is equal to the number of rows $m$), each iteration requiring a constant number of linear system solves.

Here, we consider algorithmic setups in which the iteration count is dimen\-sion-independent and no linear system solves are required, but the dependence on $1/\epsilon$ is polynomial. First, for standard $\ell_p$-regression problems, we can use a non-composite variant of the algorithm (with $\psi(\cdot) = 0$), while relying on the fact that the function $\frac{1}{q}\|\cdot\|_p^q$ with $q = \min\{2, p\}$ is $(1, p)$-weakly smooth for $p \in (1, 2)$ and $(p-1, 2)$-weakly smooth for $p \geq 2.$ Using this fact, it follows that the function
$$
    f_p(\vx) = \frac{1}{q}\|\mA \vx - \vb\|_p^q
$$
is $(L_p, q)$-weakly smooth w.r.t.~$\|\cdot\|_p$, with $L_p = \max\{p-1, 1\} \|\mA\|_{p\to p_{\ast}}^{q-1}.$ On the other hand, function $\phi(\vx) = \frac{1}{\bar{q}\min\{p-1, 1\}}\|\vx - \vx_0\|_p^{\bar{q}}$, where $\bar{q} = \max\{2, p\}$ is $(1, \bar{q})$-uniformly convex w.r.t.~$\|\cdot\|_p.$ Thus, applying Theorem~\ref{thm:comp-opt-gap-conv}, we find that we can construct a point $\vy_k \in \rr^d$ such that $f_p(\vy_k) - f_p(\vx^{\star}) \leq \epsilon,$ where $\vx^{\star} \in \argmin_{\vx \in \rr^d}f_p(\vx),$ with at most
$$
    k = \begin{cases}
        O\Big(\Big(\frac{\|\mA\|_{p\to p_{\ast}}^{p-1}}{\epsilon}\Big)^{\frac{2}{3p - 2 }}\Big(\frac{\|\vx^{\star} - \vx_0\|_p^2}{p-1}\Big)^{\frac{p}{3p - 2}}\Big), &\text{ if } p \in (1, 2)\\
        O\Big(\Big(\frac{(p-1)\|\mA\|_{p\to p_{\ast}}}{\epsilon}\Big)^{\frac{p}{p + 2}}\Big(\frac{\|\vx^{\star} - \vx_0\|_p^p}{p}\Big)^{\frac{2}{p+2}}\Big), &\text{ if } p \geq 2
    \end{cases}
$$
iterations of Generalized AGD+. The same result can be obtained by applying the iteration complexity-optimal algorithms for smooth minimization over $\ell_p$-spaces~\cite{nemirovskii1985optimal,d2018optimal}.

More interesting for our framework is the $\ell_p$ regression on correlated errors, described in the following.
\paragraph{$\ell_p$-regression on correlated errors.} 

As argued in~\cite{candes2007dantzig}, there are multiple reasons why minimizing the correlated errors $\mA^T(\mA\vx - \vb)$ in place of the standard errors $\mA\vx - \vb$ is more meaningful for many applications. First, unlike standard errors, correlated errors are invariant to orthonormal transformations of the data. Indeed, if $\mU$ is a matrix with orthonormal columns, then $(\mU\mA)^T(\mU\mA\vx - \mU\vb) = \mA^T(\mA\vx - \vb)$, but the same cannot be established for the standard error $\mA\vx - \vb$. Other reasons involve ensuring that the model includes explanatory variables that are highly correlated with the data, which is only possible to argue when working with correlated errors (see~\cite{candes2007dantzig} for more information).

Within our framework, minimization of correlated errors in $\ell_p$-norms can be reduced to making the gradient small in the $\ell_p$-norm; i.e., to applying results from Section~\ref{sec:small-grad-norm}. In particular, consider the function:
$$
    f(\vx) = \frac{1}{2}\|\mA\vx - \vb\|_2^2.
$$
The gradient of this function is precisely the vector of correlated errors, i.e., $\nabla f(\vx) = \mA^T(\mA\vx - \vb).$ Further, function $f$ is $L_{p_{\ast}}$-smooth w.r.t.~$\|\cdot\|_{p_{\ast}}$, where $L_{p_{\ast}} = \|\mA^T\mA\|_{p_{\ast}\to p}.$ 

Applying the results from Theorem~\ref{thm:grad-norm}, it follows that, for any $\epsilon >0,$ we can construct a vector $\vy_k \in \rr^d$ with $\|\mA^T(\mA\vy_k - \vb)\|_{p} \leq \epsilon,$ where $\frac{1}{p}+\frac{1}{p_{\ast}} = 1,$ with at most 
$$
    k = \begin{cases}
        \widetilde{O}\bigg(\Big(\frac{\|\mA^T\mA\|_{p_{\ast}\to p}\|\vx^{\star} - \vx_0\|_{p_{\ast}}}{\epsilon}\Big)^{\frac{2}{3p-2}}\bigg), &\text{ if } p \in (1, 2)\\
        \widetilde{O}\bigg(\sqrt{\frac{\|\mA^T\mA\|_{p_{\ast}\to p}\|\vx^{\star} - \vx_0\|_{p_{\ast}}}{\epsilon}}\bigg), &\text{ if } p > 2
    \end{cases}
$$
iterations of generalized AGD+, where $\widetilde{O}$ hides a factor that is logarithmic in $1/\epsilon$ and where each iteration takes time linear in the number of non-zeros of $\mA$. We are not aware of results of this type in the literature.

%%%%%%%%%%%%%%%%%
\subsection{Spectral Variants of Regression Problems}\label{sec:schp-reg}
The algorithms we propose in this work are not limited to $\ell_p$ settings, but apply more generally to uniformly convex spaces. A notable example of such spaces are the {\em Schatten spaces}, $\Schatten_p:=(\rr^{d\times d},\|\cdot\|_{\Schatten,p}),$ where $\|\mX\|_{\Schatten,p}=(\sum_{j\in[d]}\sigma_j(\mX)^p)^{1/p},$ where $\sigma_1(\mX),\ldots,\sigma_d(\mX)$ are the singular values of $\mX$. In particular, the aforementioned $\ell_p$-regression problems have their natural spectral counterparts, e.g., given a linear operator ${\cal A}:\rr^{d\times d}\to\rr^k$, and $\vb\in \rr^{k}$, 
$$ 
\min_{\mX\in\rr^{d\times d}}\frac{1}{s}\|{\cal A}\mX-\vb\|_q^s + \frac{\lambda}{r}\|\mX\|_{\Schatten,p}^r. 
$$
The most popular example of such a formulation comes from the nuclear norm relaxation for {\em low-rank matrix completion} \cite{Recht:2010, Chandrasekaran:2012, nesterov2013first}. We observe that the exact formulation of the problem may vary, but by virtue of Lagrangian relaxation we can interchangeably consider these different formulations as equivalent (modulo appropriate choice of regularization/constraint parameter choice). \\
To apply our algorithms to Schatten norm settings, we observe the functions below are $(1,r)$-uniformly convex, with $r=\max\{2,p\}$:
\begin{equation}\notag%\label{eq:psi_p-Schatten}
    \Psi_{\Schatten,p}(\mX)  = \begin{cases}
                        \frac{1}{2(p-1)}\|\mX\|_{\Schatten,p}^2, &\text{ if } p \in (1, 2],\\
                        \frac{1}{p} \|\mX\|_{\Schatten,p}^p, &\text{ if } p \in(2,+\infty).
                    \end{cases}
\end{equation}
On the other hand, notice that more generally than regression problems, for composite objectives
$$ f(\mX)+\lambda\Psi_{\Schatten,p}(\mX-\mX_0), $$
if the function $f$ is unitarily invariant and convex, there is a well-known formula for its subdifferential, based on the subdifferential of its vector counterpart (there is a one-to-one correspondence between unitarily invariant functions $\rr^{d\times d}$ and absolutely symmetric functions on $\rr^d$) \cite{Lewis:1995}. Even if $f$ is not unitarily invariant, in the case of regression problems the gradients can be computed explicitly. On the other hand, the regularizer $\Psi_{\Schatten,p}$ admits efficiently computable solutions to problems from Eq.~\eqref{eq:psi-simple}, given its unitary invariance (see, e.g., \cite[Section 7.3.2]{Beck:2017}).\\
Iteration complexity bounds obtained with these regularizers are analogous to those obtained in the $\ell_p$ setting. On the other hand, the lower complexity bounds proved in Section \ref{sec:LowerBounds} also apply to Schatten spaces by diagonal embedding from $\ell_p^d$, hence all the optimality/suboptimality results established for $\ell_p$ carry over into $\Schatten_p$.

\subsection{\markupadd{Entropy-Regularized Optimal Transport}}
\markupadd{
Consider the entropic regularization \cite{Fang:1992} of the discrete optimal transport problem and its dual \cite{Cuturi:2013, Lin:2022}. Given a transport cost $\mC\in \mathbb{R}_+^{m\times n}$, marginals $\vmu\in \Delta_m$, $\vnu\in \Delta_n$,\footnote{Without loss of generality, we may assume that both probability distributions have full support, thus $\mu_i,\nu_j>0$ for all $i,j$.} and regularization parameter $r>0$, let
\begin{align*}
    (P^{r}) & \min_{\mX\in \mathbb{R}_+^{m\times n}} \{\langle \mC,\mX\rangle +r \langle \mX,\ln (\mX)-\mathbf{U}\rangle : \mX\mathbf{1}=\vmu,\, \mX^{\top}\mathbf{1}=\vnu,\, \langle \mathbf{U},\mX\rangle=1\}\\
    (D^{r}) & \min_{\vu\in\mathbb{R}^m,\vv\in \mathbb{R}^n} \varphi((\vu,\vv)):= \Big\{ r\ln\Big(\sum_{i,j} \exp\big(\frac{1}{r}[u_i+v_j-c_{ij}]\big)\Big)-\langle \vmu,\vu\rangle-\langle \vnu,\vv\rangle \Big\}.
\end{align*}
where $\mathbf{1}$ denotes the all-ones vector (of the corresponding dimension) and $\mathbf{U}\in\mathbb{R}^{m\times n}$ is the all-ones matrix, $\langle\cdot,\cdot\rangle$ applied to vectors is the standard inner product, which for matrices is the Frobenius inner product, and $\ln(\cdot)$ applied to a matrix denotes the component-wise application of the natural logarithm. Note that $\vu,\vv$ are the dual variables associated with the marginal constraints $\mX\mathbf{1}=\vmu$ and $\mX^{\top}\mathbf{1}=\vnu$, respectively. Further, we emphasize that the dual objective $\varphi$ is $L$-smooth with respect to the $\|\cdot\|_{\infty}$ norm with $L=1/r$ \cite{Beck:2017}.
}

\markupadd{
Observe that by denoting
$\mB(\vu,\vv)=\big(\exp([u_i+v_j-c_{ij}]/r) \big)_{i,j\in[n]}$ and $\mX(\vu,\vv)=\mB(\vu,\vv)/\langle \mathbf{U},\mB(\vu,\vv)\rangle$, then 
\[ \nabla \varphi((\vu,\vv)) =\big(\mX(\vu,\vv)\mathbf{1}-\vmu \,\,,\,\, \mX(\vu,\vv)^{\top}\mathbf{1}-\vnu\big)^{\top}. \]
In particular, a global minimum of $(D^{r})$ induces a feasible solution for $(P^{r})$, and a dual solution with small ($\ell_1$-norm of the) gradient implies a primal solution with small ($\ell_1$-norm) infeasibility. The utility of finding such nearly feasible transports is related to the possibility of constructing (exactly) feasible transports with small additional transport cost. 
}
\markupadd{
\begin{lemma}[From \cite{Altschuler:2017}]
There exists an algorithm that runs in time $O(mn)$ which takes as input  $\mX\in\Delta_{m\times n}$ (infeasible w.r.t.~the marginal constraints), and produces a $\hat \mX\in\Delta_{m\times n}$, which is feasible for the marginal constraints, and such that $\|\mX-\hat \mX\|_1 \leq 2[\|\mX\mathbf{1}-\vmu\|_1+\|\mX^{\top}\mathbf{1}-\vnu\|_1]$.
\end{lemma}
We conclude that, in order to solve an (unregularized) optimal transport problem, it suffices to find a (regularized) dual solution with small norm of the gradient. More specifically, noticing that if $\mX^r$ is an optimal solution for $(P^r)$, then $\langle \mC,\mX^r\rangle-\langle \mC,\mX^0\rangle\leq r\ln(mn)$ \cite{Cominetti:1994, Weed:2018}, and therefore in order to obtain an $\epsilon$ optimal solution for $(P^0)$ (the unregularized problem), it suffices to choose $r=\epsilon/[2\ln(mn)]$ and target accuracy for the norm of the gradient $\|\nabla \varphi((\vu,\vv))\|_1\leq \varepsilon/[4\|\mC\|_{\infty}]=:\delta$, where $\|\mC\|_{\infty}=\max_{i,j}|c_{ij}|$.
} 

\markupadd{
We now use the methods developed in previous sections to compute such a solution. For this purpose, we endow the space $\mathbb{R}^{m+n}$ with the $\ell_{\infty}$ norm, and we use the regularizer $\psi(\vz)=\frac12\|\vz\|_2^2$, which is $1$-strongly convex w.r.t.~$\ell_{\infty}$-norm. %We consider the problem of finding a small norm of the gradient for the penalized objective, using a regularization strategy with the regularized $\Psi$. 
Notice that this setting does not exactly coincide with that of Section \ref{sec:small-grad-norm}, in particular since $\psi$ is not the $\ell_{\infty}$ norm to some power. Nevertheless, the same rationale used in the aforementioned section shows that running AGD+ on the regularized objective $\bar\varphi:=\varphi+\lambda \psi$ with $\lambda$ set such that $\lambda\|\nabla\psi((\vu^{\ast},\vv^{\ast}))\|_1\leq\delta/2$ will provide the desired vector with $\ell_1$ norm of the gradient $\delta$ with complexity 
$O\big( \sqrt{\frac{L}{\lambda}} \log\big(\frac{L\psi((\vu^{\ast},\vv^{\ast}))}{\delta}\big) \big)$.
}

\markupadd{Our last task is then to provide an a-priori bound on the $\ell_{\infty}$-norm of the optimal dual solution. In this respect, multiple results can be found in the literature \cite{Cominetti:1994,Dvurechensky:2018,Chambolle:2022}. The following is particularly useful for our purposes.}

\markupadd{
\begin{proposition}[Adapted from \cite{Lin:2022}]
%Let $(u^{\ast},v^{\ast})$ be an optimal solution for $(D^{r})$. %Then $\max_i u_i^{\ast} - \min_i u_i^{\ast}\leq \|C\|_{\infty}+r\ln\bar\vmu$, and $\max_j v_j-\min_jv_j\leq \|C\|_{\infty}+r\ln\bar\vnu$, where $\|C\|_{\infty}=\max_{i,j}|c_{ij}|$,  $\bar\vmu=\max_{i,k}\mu_i/\mu_k$, and $\bar\vnu=\max_{j,l}\nu_j/\nu_l$. In particular, 
There exists an optimal solution $(\vu^{\ast},\vv^{\ast})$ for $(D^{r})$ such that
$\|(\vu^{\ast},\vv^{\ast})\|_{\infty}\leq  r[\ln\max\{\bar\vmu,\bar\vnu\}+\ln (mn)]+\|\mC\|_{\infty}$, where   $\bar\vmu=\max_{i}1/\mu_i$, and $\bar\vnu=\max_{j}1/\nu_j$.
\end{proposition}
}

\markupadd{
Finally, the resulting oracle complexity to obtain accuracy $\delta$ in the norm of the gradient is given by
\begin{align*}
    &O\Big( \sqrt{\frac{L}{\lambda}}\log\big( \frac{L\psi((\vu^{\ast},\vv^{\ast})}{\delta} \big) \Big)\\ 
    &= \textstyle O\Big( \sqrt{\frac{2\|(\vu^{\ast},\vv^{\ast})\|_1}{r\delta}}\log\big( \frac{\|\mC\|_{\infty}\|(\vu^{\ast},\vv^{\ast})\|_2^2}{r\varepsilon} \big) \Big) \\
    &= \textstyle O\Big( \sqrt{(m+n)\log(mn)}\ln\big(\frac{(m+n)\|\mC\|_{\infty}}{\varepsilon}\big)\Big[ \frac{\|\mC\|_{\infty}}{\varepsilon}+\sqrt{\frac{\|\mC\|_{\infty}\varepsilon\log\max\{\bar\mu,\bar\nu\}}{\log(mn)}} \Big] \Big).
\end{align*}
Noticing that each step of this method requires arithmetic complexity $O(mn)$, we finally obtain a total complexity of $O((m+n)^{5/2}\|C\|_{\infty}/\varepsilon)$, which matches the state of the art of the existing practically scalable methods (see discussions in \cite{Chambolle:2022}). The only existing method that obtains improved complexity is based on area convexity \cite{Jambulapati:2019}, which unfortunately is not competitive unless the dimension is extremely large. To summarize, this example shows how our methods can directly reproduce existing complexity bounds that have been obtained by arguably much more sophisticated and ad-hoc methods.
}
%%%%%%%%%%%%%%%%%%%%%%%%%
\section{Conclusion and Future Work}

We presented a general algorithmic framework for \emph{complementary composite optimization}, where the objective function is the sum of two functions with complementary properties---(weak) smoothness and uniform/strong convexity. The framework has a number of interesting applications, including in making the gradient of a smooth function small in general norms and in different regression problems that frequently arise in machine learning. We also provided lower bounds that certify near-optimality of our algorithmic framework for the majority of standard $\ell_p$ and $\Schatten_p$ setups.

Some challenging questions for future work remain. First, the regularization-based approach that we employed for gradient norm minimization leads to near-optimal oracle complexity bounds only when the objective function is smooth and the norm of the space is strongly convex (i.e., when the $p_{\ast}$-norm of the gradient is sought for $p_{\ast} \geq 2$). The primary reason for this is that these are the only settings in which the complementary composite minimization leads to linear convergence. As the bounds we obtain for complementary composite minimization are near-tight, this represents a fundamental limitation of direct regularization-based approach. It is an open question whether the non-tight bounds for gradient norm minimization can be improved using some type of recursive regularization, as in~\cite{allen2018make}. Of course, there are clear challenges in trying to generalize such an approach to non-Euclidean norms, caused by the fundamental limitation that non-Euclidean norms cannot be simultaneously smooth and strongly convex, as discussed at the beginning of the paper. Another interesting question is whether there exist direct (not regularization-based) algorithms for minimizing general gradient norms and that converge with (near-)optimal oracle complexity. 
%\newpage

\begin{acknowledgements}
The authors would like to thank Carlos Sing-Long and Adrien Taylor for valuable discussions and feedback on a first version of this paper. \markupadd{We would also like to thank Juan Pablo Contreras, Roberto Cominetti and Mario Bravo for insightful discussions on optimal transport and its entropic regularization.}
\end{acknowledgements}

\bibliography{references.bib}
\bibliographystyle{plain}

\appendix

\section{\markupadd{Impossibility of Acceleration in the Relatively Smooth and Relatively Strongly Convex Setting}} \label{app:LB_rel_smooth}

\markupadd{Below we exploit a simple reduction based on regularization to prove a lower bound on relatively smooth and relatively strongly convex optimization. The problem we will reduce to is relatively smooth convex optimization, for which tight lower complexity bounds are known \cite{dragomir2019optimal}.
\begin{proposition}
The complexity of $L$-relatively smooth and $\mu$-relative strongly convex minimization is  bounded below by $\Omega\big(\frac{L}{\mu}\big)$.
\end{proposition}
\vspace{-0.6cm}
}
\markupadd{
\begin{proof} 
Suppose that the oracle complexity of solving relatively smooth and relatively strongly convex functions is $o\big(\frac{L}{\mu}\big)$. Then, given a function $f$ which is $L$ relatively smooth w.r.t.~$h$, consider the optimization problems
\begin{eqnarray*} 
(P_0) && \min_x f(x) ,\\
(P_{\lambda}) &&\min_x f(x)+\lambda h(x).
\end{eqnarray*}
Note that $(P_{\lambda})$ is $\lambda$-relative strongly convex and $(L+\lambda)$-relatively smooth, both w.r.t. $h$. Proceeding as in the proof of Theorem \ref{thm:grad-norm}, we have that $h(x^{\lambda})\leq h(x^0)$, where $x^{0}$ and $x^{\lambda}$ are optimal solutions for $(P_0)$ and $(P_{\lambda})$, respectively. Hence, it suffices to solve $(P_{\lambda})$ to accuracy $\varepsilon/2$ and $\lambda=\varepsilon/[2h(x^0)]$, to obtain a solution with accuracy $\varepsilon$ for problem $(P_0)$.\\
Now, using an optimal algorithm for problem $(P_{\varepsilon/[2h(x^0)]})$, we have by our assumption that its oracle complexity is at most
\[ o\Big( \frac{L+\lambda}{\lambda} \Big)=o\Big( \frac{Lh(x^0)}{\varepsilon} \Big), \]
a contradiction with the lower bound from \cite{dragomir2019optimal}. \qed
\end{proof}
}
\end{document}